\documentclass[a4paper,twoside]{article}

\usepackage[english]{babel}
\usepackage[utf8]{inputenc}
\usepackage[T1]{fontenc}
\usepackage{csquotes}
\usepackage{amsmath}
\usepackage{amsthm}
\usepackage{amsfonts}
\usepackage{amssymb}
\usepackage{fancyhdr}
\usepackage[mathscr]{eucal}
\usepackage{pdfpages}
\usepackage{multicol}
\usepackage[all]{xy}
\usepackage{pgf,tikz}
\usepackage[g]{esvect}
\usepackage{color}
\usetikzlibrary{arrows}
\usepackage[pdftex,pdfborder={0 0 0},linktoc=all]{hyperref}

\theoremstyle{plain}
	\newtheorem{thm}{Theorem}[section]
	\newtheorem{cor}[thm]{Corollary}
	\newtheorem{lem}[thm]{Lemma}
	\newtheorem{prop}[thm]{Proposition}

\theoremstyle{definition}
	\newtheorem{dfn}[thm]{Definition}
	\newtheorem{ntn}[thm]{Notation}

\theoremstyle{remark}
	\newtheorem{rem}[thm]{Remark}
	\newtheorem{rems}[thm]{Remarks}

\numberwithin{equation}{section}
\numberwithin{figure}{section}

\newcommand{\C}{\mathbb{C}}
\newcommand{\D}{\mathcal{D}}
\newcommand{\E}{\mathcal{E}}
\newcommand{\N}{\mathbb{N}}
\newcommand{\R}{\mathbb{R}}

\newcommand{\X}{\mathcal{X}}

\newcommand{\cov}[2]{\Cov\!\left( #1, #2 \right)}
\newcommand{\dx}{\dmesure\!}
\newcommand{\debar}{\overline{\partial}}
\newcommand{\deron}[2]{\frac{\partial #1}{\partial #2}}

\newcommand{\esp}[2][]{\mathbb{E}_{#1}\!\left[ #2 \right]}
\newcommand{\espcond}[3][]{\mathbb{E}_{#1}\!\left[ #2\hspace{-1mm} \rule{0pt}{4mm}\mvert \! #3 \right]}
\newcommand{\odet}[1]{\norm{\det\! ^\perp\left(#1\right)}}

\newcommand{\mvert}{\mathrel{}\middle|\mathrel{}}
\newcommand{\norm}[1]{\left\lvert #1 \right\rvert}
\newcommand{\Norm}[1]{\left\lVert #1 \right\rVert}
\newcommand{\prsc}[2]{\left\langle #1\,, #2 \right\rangle}
\newcommand{\rmes}[1]{\norm{\dmesure\! V_{#1}}}
\newcommand{\var}[1]{\Var\!\left( #1 \right)}
\newcommand{\vol}[1]{\Vol\left(#1\right)}

\renewcommand{\L}{\mathcal{L}}
\renewcommand{\P}{\mathbb{P}}
\renewcommand{\S}{\mathbb{S}}
\renewcommand{\H}{H^0(\X,\E \otimes \L^d)}

\renewcommand{\epsilon}{\varepsilon}
\renewcommand{\geq}{\geqslant}
\renewcommand{\leq}{\leqslant}
\renewcommand{\tilde}{\widetilde}

\DeclareMathOperator{\Cov}{Cov}
\DeclareMathOperator{\dmesure}{d}
\DeclareMathOperator{\End}{End}
\DeclareMathOperator{\ev}{ev}

\DeclareMathOperator{\Id}{Id}

\DeclareMathOperator{\tr}{Tr}
\DeclareMathOperator{\Var}{Var}
\DeclareMathOperator{\Vol}{Vol}

\author{Thomas Letendre \thanks{Thomas Letendre, École Normale Supérieure de Lyon, Unité de Mathématiques Pures et Appliquées, UMR CNRS 5669, 46 allée d'Italie, 69634 Lyon Cedex 07, France; \newline e-mail address: \url{thomas.letendre@ens-lyon.fr}.}}
\date{\today}
\title{Variance of the volume of random real algebraic submanifolds}

\begin{document}

\maketitle

\begin{abstract}
Let $\X$ be a complex projective manifold of dimension $n$ defined over the reals and let $M$ denote its real locus. We study the vanishing locus $Z_{s_d}$ in $M$ of a random real holomorphic section $s_d$ of $\E \otimes \L^d$, where $\L \to \X$ is an ample line bundle and $\E\to \X$ is a rank $r$ Hermitian bundle. When $r\in \{1,\dots,n-1\}$, we obtain an asymptotic of order $d^{r-\frac{n}{2}}$, as $d$ goes to infinity, for the variance of the linear statistics associated with $Z_{s_d}$, including its volume. Given an open set $U \subset M$, we show that the probability that $Z_{s_d}$ does not intersect $U$ is a $O$ of $d^{-\frac{n}{2}}$ when $d$ goes to infinity. When $n\geq 3$, we also prove almost sure convergence for the linear statistics associated with a random sequence of sections of increasing degree. Our framework contains the case of random real algebraic submanifolds of $\R\P^n$ obtained as the common zero set of $r$ independent Kostlan--Shub--Smale polynomials.
\end{abstract}



\section{Introduction}
\label{sec introduction}

\paragraph*{Framework.}

Let us first describe our framework and state the main results of this article (see Section~\ref{sec random real algebraic submanifolds} for more details). Let $\X$ be a smooth complex projective manifold of positive complex dimension $n$. Let $\L$ be an ample holomorphic line bundle over $\X$ and let $\E$ be a rank $r$ holomorphic vector bundle over $\X$, with $r \in \{1,\dots,n\}$. We assume that $\X$, $\E$ and $\L$ are endowed with compatible real structures and that the real locus $M$ of $\X$ is not empty. Let $h_\E$ and $h_\L$ denote Hermitian metrics on $\E$ and $\L$ respectively that are compatible with the real structures. We assume that $h_\L$ has positive curvature $\omega$. Then $\omega$ is a Kähler form on $\X$ and it induces a Riemannian metric $g$ on $M$.

For any $d \in \N$, the Kähler form $\omega$, $h_\E$ and $h_\L$ induce a $L^2$-inner product on the space $\R \H$ of real holomorphic sections of $\E \otimes \L^d \to \X$ (see \eqref{eq definition inner product}). Let $d \in \N$ and $s \in \R \H$, we denote by $Z_s$ the real zero set $s^{-1}(0)\cap M$ of $s$. For $d$ large enough, for almost every $s \in \R \H$, $Z_s$ is a codimension $r$ smooth submanifold of $M$ and we denote by $\rmes{s}$ the Riemannian measure on $Z_s$ induced by $g$ (see Sect.~\ref{subsec random submanifolds}). In the sequel, we will consider $\rmes{s}$ as a positive Radon measure on $M$. Let us also denote by $\rmes{M}$ the Riemannian measure on $M$.

Let $s_d$ be a standard Gaussian vector in $\R \H$. Then $\rmes{s_d}$ is a random positive Radon measure on $M$. We set $Z_{d}=Z_{s_d}$ and $\rmes{d}=\rmes{s_d}$ to avoid too many subscripts. In a previous paper \cite[thm.~1.3]{Let2016}, we computed the asymptotic of the expected Riemannian volume of $Z_d$ as $d \to +\infty$. Namely, we proved that:
\begin{equation}
\label{eq reminder expectation volume}
\esp{\vol{Z_d}} = d^\frac{r}{2}\vol{M} \frac{\vol{\S^{n-r}}}{\vol{\S^n}} + O\!\left(d^{\frac{r}{2}-1}\right),
\end{equation}
where $\vol{M}$ is the volume of $M$ for $\rmes{M}$ and the volumes of spheres are Euclidean volumes. Here and throughout this paper, $\esp{\cdot}$ denotes the expectation of the random variable between the brackets, and $\S^m$ stands for the unit Euclidean sphere of dimension $m$.

Let $\phi \in \mathcal{C}^0(M)$, we denote by $\Norm{\phi}_\infty = \max_{x \in M} \norm{\phi(x)}$ its norm sup. Besides, we denote by $\prsc{\cdot}{\cdot}$ the duality pairing between $\left(\mathcal{C}^0(M),\Norm{\cdot}_\infty\right)$ and its topological dual. Then, \eqref{eq reminder expectation volume} can be restated as:
\begin{equation*}
\esp{\prsc{\rmes{d}}{\mathbf{1}}} = d^\frac{r}{2}\vol{M} \frac{\vol{\S^{n-r}}}{\vol{\S^n}} + O\!\left(d^{\frac{r}{2}-1}\right),
\end{equation*}
where $\mathbf{1} \in \mathcal{C}^0(M)$ stands for the unit constant function on $M$. The same proof gives similar asymptotics for $\esp{\prsc{\rmes{d}}{\phi}}$ for any continuous $\phi : M \to \R$ (see \cite[section 5.3]{Let2016}).
\begin{thm}
\label{thm reminder expectation test function}
Let $\X$ be a complex projective manifold of positive dimension $n$ defined over the reals, we assume that its real locus $M$ is non-empty. Let $\E \to \X$ be a rank $r \in \{1,\dots,n\}$ Hermitian vector bundle and let $\L \to \X$ be a positive Hermitian line bundle, both equipped with compatible real structures. For every $d \in \N$, let $s_d$ be a standard Gaussian vector in $\R \H$. Then the following holds as $d \to +\infty$:
\begin{equation}
\label{eq reminder expectation test function}
\forall \phi \in \mathcal{C}^0(M), \qquad \esp{\prsc{\rmes{d}}{\phi}} = d^\frac{r}{2} \left(\int_M \phi \rmes{M}\right) \frac{\vol{\S^{n-r}}}{\vol{\S^n}} + \Norm{\phi}_\infty O\!\left(d^{\frac{r}{2}-1}\right).
\end{equation}
Moreover the error term $O\!\left(d^{\frac{r}{2}-1}\right)$ does not depend on $\phi$.
\end{thm}
In particular, we can define a sequence of Radon measures $\left(\esp{\rmes{d}}\right)_{d\geq d_0}$ on $M$ by: for every $d \geq d_0$ and every $\phi \in \mathcal{C}^0(M)$, $\prsc{\esp{\rmes{d}}}{\phi} = \esp{\prsc{\rmes{d}}{\phi}}$. Then Thm.~\ref{thm reminder expectation test function} implies that:
\begin{equation}
\left(d^{-\frac{r}{2}}\right)\esp{\rmes{d}} \xrightarrow[d \to + \infty]{} \frac{\vol{\S^{n-r}}}{\vol{\S^n}} \rmes{M},
\end{equation}
as continuous linear functionals on $\left(\mathcal{C}^0(M),\Norm{\cdot}_\infty\right)$.

\paragraph*{Statement of the results.}

The main result of this paper is an asymptotic for the covariances of the linear statistics $\left\{\prsc{\rmes{d}}{\phi}\mvert \phi \in \mathcal{C}^0(M)\right\}$. Before we can state our theorem, we need to introduce some additional notations.

As usual, we denote by $\var{X} = \esp{\left(X-\esp{X}\right)^2}$ the variance of the real random variable $X$, and by $\cov{X}{Y} = \esp{\left(X-\esp{X}\right)\left(Y-\esp{Y}\right)}$ the covariance of the real random variables $X$ and $Y$. We call \emph{variance} of $\rmes{d}$ and we denote by $\var{\rmes{d}}$ the symmetric bilinear form on $\mathcal{C}^0(M)$ defined by:
\begin{equation}
\label{def cov dVd}
\forall \phi_1,\phi_2 \in \mathcal{C}^0(M), \qquad \var{\rmes{d}}\left(\phi_1,\phi_2\right) = \cov{\prsc{\rmes{d}}{\phi_1}}{\prsc{\rmes{d}}{\phi_2}}.
\end{equation}

\begin{dfn}
\label{def continuity modulus}
Let $\phi \in \mathcal{C}^0(M)$, we denote by $\varpi_\phi$ its \emph{continuity modulus}, which is defined by:
\begin{equation*}
\begin{array}{cccc}
\varpi_\phi: & (0,+\infty) & \longrightarrow & [0,+\infty) \\
 & \epsilon & \longmapsto & \sup \left\{ \norm{\phi(x)-\phi(y)} \mvert (x,y) \in M^2, \rho_g(x,y)\leq \epsilon \right\},
\end{array}
\end{equation*}
where $\rho_g(\cdot,\cdot)$ stands for the geodesic distance on $(M,g)$.
\end{dfn}
Since $M$ is compact, $\varpi_\phi$ is well-defined for every $\phi \in \mathcal{C}^0(M)$. Moreover every $\phi \in \mathcal{C}^0(M)$ is uniformly continuous and we have:
\begin{equation*}
\forall \phi \in \mathcal{C}^0(M), \qquad \varpi_\phi (\epsilon) \xrightarrow[\epsilon \to 0]{} 0.
\end{equation*}
Note that, if $\phi: M \to \R$ is Lipschitz continuous, then $\varpi_\phi(\epsilon) = O\!\left(\epsilon\right)$ as $\epsilon \to 0$.

\begin{dfn}
\label{def Jacobian}
Let $L : V \to V'$ be a linear map between two Euclidean spaces, we denote by $\odet{L}$ the \emph{Jacobian} of $L$:
\begin{equation*}
\label{eq def odet}
\odet{L} = \sqrt{\det\left(LL^*\right)},
\end{equation*}
where $L^*:V' \to V$ is the adjoint operator of $L$.
\end{dfn}
See Section~\ref{subsec Kac--Rice formula} for a quick discussion of the properties of this Jacobian. If $A$ is an element of $\mathcal{M}_{rn}(\R)$, the space of matrices of size $r\times n$ with real coefficients, we denote by $\odet{A}$ the Jacobian of the linear map from $\R^n$ to $\R^r$ associated with $A$ in the canonical bases of~$\R^n$ and $\R^r$.

\begin{dfn}
\label{def XYt}
For every $t>0$, we define $\left(X(t),Y(t)\right)$ to be a centered Gaussian vector in $\mathcal{M}_{rn}(\R)\times \mathcal{M}_{rn}(\R)$ with variance matrix:
\begin{equation*}
\left(\begin{array}{ccccc|ccccc}
1-\frac{te^{-t}}{1-e^{-t}} & 0 & \cdots & \cdots & 0 & e^{-\frac{t}{2}}-\frac{te^{-\frac{t}{2}}}{1-e^{-t}} & 0 & \cdots & \cdots & 0\\
0 & 1 & \ddots & & \vdots & 0 & e^{-\frac{t}{2}} & \ddots & & \vdots\\
\vdots & \ddots & \ddots & \ddots & \vdots & \vdots & \ddots & \ddots & \ddots & \vdots \\
\vdots & & \ddots & 1 & 0 & \vdots & & \ddots & e^{-\frac{t}{2}} & 0\\
0 & \cdots & \cdots & 0 & 1 & 0 & \cdots & \cdots & 0 & e^{-\frac{t}{2}}\\
\hline
e^{-\frac{t}{2}}-\frac{te^{-\frac{t}{2}}}{1-e^{-t}} & 0 & \cdots & \cdots & 0 & 1-\frac{te^{-t}}{1-e^{-t}} & 0 & \cdots & \cdots & 0\\
0 & e^{-\frac{t}{2}} & \ddots & & \vdots & 0 & 1 & \ddots & & \vdots\\
\vdots & \ddots & \ddots & \ddots & \vdots & \vdots & \ddots & \ddots & \ddots & \vdots \\
\vdots & & \ddots & e^{-\frac{t}{2}} & 0 & \vdots & & \ddots & 1 & 0\\
0 & \cdots & \cdots & 0 & e^{-\frac{t}{2}} & 0 & \cdots & \cdots & 0 & 1
\end{array}\right) \otimes I_r,
\end{equation*}
where $I_r$ is the identity matrix of size $r$. That is, if we denote by $X_{ij}(t)$ (resp.~$Y_{ij}(t)$) the coefficients of $X(t)$ (resp.~$Y(t)$), the couples $\left(X_{ij}(t),Y_{ij}(t)\right)$ with $1 \leq i \leq r$ and $1 \leq j \leq n$ are independent from one another and the variance matrix of $\left(X_{ij}(t),Y_{ij}(t)\right)$ is:
\begin{align*}
&\begin{pmatrix}
1-\frac{te^{-t}}{1-e^{-t}}& e^{-\frac{t}{2}}\left(1-\frac{t}{1-e^{-t}}\right)\\
e^{-\frac{t}{2}}\left(1-\frac{t}{1-e^{-t}}\right) & 1-\frac{te^{-t}}{1-e^{-t}}
\end{pmatrix} & &\text{if } j=1, \text{ and} & &\begin{pmatrix}
1& e^{-\frac{t}{2}}\\ e^{-\frac{t}{2}} & 1
\end{pmatrix} & &\text{otherwise.}
\end{align*}
\end{dfn}

\begin{ntn}
\label{ntn alpha0}
We set $\alpha_0 = \dfrac{n-r}{2(2r+1)(2n+1)}$.
\end{ntn}

We can now state our main result.
\begin{thm}
\label{thm asymptotics variance}
Let $\X$ be a complex projective manifold of dimension $n\geq 2$ defined over the reals, we assume that its real locus $M$ is non-empty. Let $\E \to \X$ be a rank $r \in \{1,\dots,n-1\}$ Hermitian vector bundle and let $\L \to \X$ be a positive Hermitian line bundle, both equipped with compatible real structures. For every $d \in \N$, let $s_d$ be a standard Gaussian vector in $\R \H$.

Let $\beta \in (0,\frac{1}{2})$, then there exists $C_\beta >0$ such that, for all $\alpha \in (0,\alpha_0)$, for all $\phi_1$ and $\phi_2 \in \mathcal{C}^0(M)$, the following holds as $d \to +\infty$:
\begin{multline}
\label{eq asymptotics variance}
\var{\rmes{d}}\left(\phi_1,\phi_2\right) = d^{r-\frac{n}{2}} \left(\int_M \phi_1 \phi_2 \rmes{M}\right) \frac{\vol{\S^{n-1}}}{(2\pi)^r} \mathcal{I}_{n,r}\\
+ \Norm{\phi_1}_{\infty}\Norm{\phi_2}_{\infty}  O\!\left(d^{r-\frac{n}{2}-\alpha}\right) + \Norm{\phi_1}_{\infty}\varpi_{\phi_2}\left(C_\beta d^{-\beta}\right) O\!\left(d^{r-\frac{n}{2}}\right),
\end{multline}
where
\begin{equation}
\label{eq def Inr}
\mathcal{I}_{n,r} = \frac{1}{2} \int_0^{+\infty} \left(\frac{\esp{\odet{X(t)}\odet{Y(t)}}}{\left(1-e^{-t}\right)^\frac{r}{2}} -(2\pi)^r\left(\frac{\vol{\S^{n-r}}}{\vol{\S^n}}\right)^2\right)t^{\frac{n-2}{2}}\dx t< + \infty.
\end{equation}
Moreover the error terms $O\!\left(d^{r-\frac{n}{2}-\alpha}\right)$ and $O\!\left(d^{r-\frac{n}{2}}\right)$ in~\eqref{eq asymptotics variance} do not depend on $\left(\phi_1,\phi_2\right)$.
\end{thm}

We obtain the variance of the volume of $Z_d$ by applying Thm.~\ref{thm asymptotics variance} to $\phi_1=\phi_2= \mathbf{1}$. When $\phi_1=\phi_2= \phi$ we get the following.
\begin{cor}[Variance of the linear statistics]
\label{cor variance of linear statistics}
In the same setting as Thm.~\ref{thm asymptotics variance}, let $\beta \in (0,\frac{1}{2})$, then there exists $C_\beta >0$ such that, for all $\alpha \in (0,\alpha_0)$ and all $\phi \in \mathcal{C}^0(M)$, the following holds as $d \to +\infty$:
\begin{multline}
\label{eq variance of linear statistics}
\var{\prsc{\rmes{d}}{\phi}} = d^{r-\frac{n}{2}} \left(\int_M \phi^2 \rmes{M}\right) \frac{\vol{\S^{n-1}}}{(2\pi)^r} \mathcal{I}_{n,r}\\
+ \Norm{\phi}_{\infty}^2 O\!\left(d^{r-\frac{n}{2}-\alpha}\right) + \Norm{\phi}_{\infty}\varpi_{\phi}\left(C_\beta d^{-\beta}\right) O\!\left(d^{r-\frac{n}{2}}\right).
\end{multline}
Moreover, the error terms do not depend on $\phi$.
\end{cor}

\begin{rems}
Some remarks are in order.
\begin{itemize}
\item The value of the constant $\alpha_0$ should not be taken too seriously. This constant appears for technical reasons and it is probably far from optimal.
\item If $\phi_2$ is Lipschitz continuous with Lipschitz constant $K$, then the error term in eq.~\eqref{eq asymptotics variance} can be replaced by:
\begin{equation*}
\Norm{\phi_1}_\infty \left(\Norm{\phi_2}_\infty + K\right) O\!\left(d^{r-\frac{n}{2}-\alpha}\right)
\end{equation*}
by fixing $\beta > \alpha_0$, which is possible since $\frac{1}{2} > \alpha_0$.
\item Thm.~\ref{thm asymptotics variance} shows that $\var{\rmes{d}}$ is a continuous bilinear form on $\left(\mathcal{C}^0(M),\Norm{\cdot}_\infty\right)$ for $d$ large enough. Moreover, denoting by $\prsc{\cdot}{\cdot}_M$ the $L^2$-inner product on $\mathcal{C}^0(M)$ defined by $\prsc{\phi_1}{\phi_2}_M = \displaystyle\int_M \phi_1\phi_2 \rmes{M}$, we have:
\begin{equation*}
d^{\frac{n}{2}-r} \var{\rmes{d}} \xrightarrow[d \to + \infty]{} \frac{\vol{\S^{n-1}}}{(2\pi)^r} \mathcal{I}_{n,r} \prsc{\cdot}{\cdot}_M
\end{equation*}
in the weak sense. A priori, there is no such convergence as continuous bilinear forms on $\left(\mathcal{C}^0(M),\Norm{\cdot}_\infty\right)$ since the estimate~\eqref{eq asymptotics variance} involves the continuity modulus of $\phi_2$.
\item The fact that the constant $\mathcal{I}_{n,r}$ is finite is part of the statement and is proved below (Lemma~\ref{lem integrability Dnr}). This constant is necessarily non-negative. Numerical evidence suggests that it is positive but we do not know how to prove it at this point.
\item Thm.~\ref{thm asymptotics variance} does not apply in the case of maximal codimension ($r=n$). This case presents an additional singularity which causes our proof to fail. However, we believe a similar result to be true for $r=n$, at least in the case of the Kostlan--Shub--Smale polynomials described below (compare \cite{Dal2015,Wsc2005}).
\end{itemize}
\end{rems}

\begin{cor}[Concentration in probability]
\label{cor concentration}
In the same setting as Thm.~\ref{thm asymptotics variance}, let $\alpha > \frac{r}{2}- \frac{n}{4}$ and let $\phi \in \mathcal{C}^0(M)$. Then, for every $\epsilon >0$, we have:
\begin{equation*}
\P\left( \norm{\prsc{\rmes{d}}{\phi}-\rule{0pt}{4mm}\esp{\prsc{\rmes{d}}{\phi}}} > d^{\alpha} \epsilon \right)  = \frac{1}{\epsilon^2} O\!\left( d^{r-\frac{n}{2}-2\alpha} \right),
\end{equation*}
where the error term is independent of $\epsilon$, but depends on $\phi$.
\end{cor}

\begin{cor}
\label{cor connected components}
In the same setting as Thm.~\ref{thm asymptotics variance}, let $U \subset M$ be an open subset, then as $d\to + \infty$ we have:
\begin{equation*}
\P\left( Z_d \cap U = \emptyset \right) = O\!\left( d^{-\frac{n}{2}} \right).
\end{equation*}
\end{cor}

Our last corollary is concerned with the convergence of a random sequence of sections of increasing degree. Let us denote by $\dx \nu_d$ the standard Gaussian measure on $\R\H$ (see~\eqref{eq Gaussian density}). Let $\dx \nu$ denote the product measure $\bigotimes_{d \in \N} \dx \nu_d$ on $\prod_{d \in \N} \R \H$. Then we have the following.

\begin{cor}[Almost sure convergence]
\label{cor as convergence}
In the same setting as Thm.~\ref{thm asymptotics variance}, let us assume that $n\geq 3$. Let $(s_d)_{d \in \N} \in \prod_{d \in \N} \R \H$ be a random sequence of sections. Then, $\dx \nu$-almost surely, we have:
\begin{equation*}
\forall \phi \in \mathcal{C}^0(M), \qquad d^{-\frac{r}{2}}\prsc{\rmes{s_d}}{\phi} \xrightarrow[d \to +\infty]{} \frac{\vol{\S^{n-r}}}{\vol{\S^n}} \left(\int_M \phi \rmes{M}\right).
\end{equation*}
That is, $\dx \nu$-almost surely,
\begin{equation*}
d^{-\frac{r}{2}}\rmes{s_d} \xrightarrow[d \to +\infty]{} \frac{\vol{\S^{n-r}}}{\vol{\S^n}}\rmes{M},
\end{equation*}
in the sense of the weak convergence of measures.
\end{cor}

\begin{rem}
We expect this result to hold for $n=2$ as well, but our proof fails in this case. 
\end{rem}

\paragraph*{The Kostlan--Shub--Smale polynomials}

Let us consider the simplest example of our framework. We choose $\X$ to be the complex projective space $\C\P^n$, with the real structure defined by the usual conjugation in $\C^{n+1}$. Then $M$ is the real projective space $\R \P^n$. Let $\L = \mathcal{O}(1)$ be the hyperplane line bundle, equipped with its natural real structure and the metric dual to the standard metric on the tautological line bundle over $\C\P^n$. Then the curvature form of $\L$ is the Fubini--Study form $\omega_{FS}$, normalized so that the induced Riemannian metric is the quotient of the Euclidean metric on the unit sphere of $\C^{n+1}$. Let $\E = \C^r \times \C\P^n \to \C\P^n$ be the rank $r$ trivial bundle with the trivial real structure and the trivial metric.

In this setting, the global holomorphic sections of $\L^d$ are the complex homogeneous polynomials of degree $d$ in $n+1$ variables and those of $\E \otimes \L^d$ are $r$-tuples of such polynomials, since $\E$ is trivial. Finally, the real structures being just the usual conjugations, we have:
\begin{equation*}
\R \H = \R_{\text{hom}}^d[X_0,\dots,X_n]^r,
\end{equation*}
where $\R_{\text{hom}}^d[X_0,\dots,X_n]$ is the space of real homogeneous polynomials of degree $d$ in $n+1$ variables. The $r$ copies of this space in $\R \H$ are pairwise orthogonal for the inner product~\eqref{eq definition inner product}. Hence a standard Gaussian in $\R\H$ is a $r$-tuple of independent standard Gaussian in $\R_{\text{hom}}^d[X_0,\dots,X_n] = \R H^0\left(\X,\L^d\right)$.

It is well-known (cf.~\cite{BSZ2000a,BBL1996,Kos1993}) that the monomials are pairwise orthogonal for the $L^2$-inner product~\eqref{eq definition inner product}, but not orthonormal. Let $\alpha =\left(\alpha_0,\dots,\alpha_n\right) \in \N^{n+1}$, we denote its length by $\norm{\alpha} = \alpha_0 + \dots + \alpha_n$. We also define $X^{\alpha} = X_0^{\alpha_0} \cdots X_n^{\alpha_n}$ and $\alpha ! = (\alpha_0!) \cdots (\alpha_n!)$. Finally, if $\norm{\alpha}=d$, we denote by $\binom{d}{\alpha}$ the multinomial coefficient $\frac{d!}{\alpha!}$. Then, an orthonormal basis of $\R_{\text{hom}}^d[X_0,\dots,X_n]$ for the inner product~\eqref{eq definition inner product} is given by the family:
\begin{equation*}
\left(\sqrt{\frac{(d+n)!}{\pi^n d!}} \sqrt{\binom{d}{\alpha}}X^\alpha\right)_{\norm{\alpha} = d}.
\end{equation*}
Thus a standard Gaussian vector in $\R_{\text{hom}}^d[X_0,\dots,X_n]$ is a random polynomial:
\begin{equation*}
\sqrt{\frac{(d+n)!}{\pi^n d!}} \sum_{\norm{\alpha} =d} a_\alpha \sqrt{\binom{d}{\alpha}}X^\alpha,
\end{equation*}
where the coefficients $(a_\alpha)_{\norm{\alpha}=d}$ are independent real standard Gaussian variables. Since we are only concerned with the zero set of this random polynomial, we can drop the factor $\sqrt{\frac{(d+n)!}{\pi^n d!}}$.

Finally, in this setting, $\rmes{d}$ is the common zero set of $r$ independent random polynomials in $\R_{\text{hom}}^d[X_0,\dots,X_n]$ of the form:
\begin{equation}
\label{KSS polynomials}
\sum_{\norm{\alpha} =d} a_\alpha \sqrt{\binom{d}{\alpha}}X^\alpha,
\end{equation}
with independent coefficients $(a_\alpha)_{\norm{\alpha}=d}$ distributed according to the real standard Gaussian distribution. Such polynomials are known as the Kostlan--Shub--Smale polynomials. They were introduced in \cite{Kos1993,SS1993} and were actively studied since (cf.~\cite{AW2005,Buer2007,Dal2015,Pod2001,Wsc2005}).

\paragraph*{Related works.}

As we just said, zero sets of systems of independent random polynomials distributed as~\eqref{KSS polynomials} were studied by Kostlan \cite{Kos1993} and Shub and Smale \cite{SS1993}. The expected volume of these random algebraic manifolds was computed by Kostlan \cite{Kos1993} and their expected Euler characteristic was computed by Podkorytov \cite{Pod2001} in codimension $1$, and by Bürgisser \cite{Buer2007} in higher codimension. Both these results were extended to the setting of the present paper in \cite{Let2016}.

In \cite{Wsc2005}, Wschebor obtained an asymptotic bound, as the dimension $n$ goes to infinity, for the variance of number of real roots of a system of $n$  independent Kostlan--Shub--Smale polynomials. Recently, Dalmao \cite{Dal2015} computed an asymptotic of order $\sqrt{d}$ for the variance of the number of real roots of one Kostlan--Shub--Smale polynomial in dimension $n=1$. His result is very similar to~\eqref{eq asymptotics variance}, which leads us to think that such a result should hold for $r=n$. He also proved a central limit theorem for this number of real roots, using Wiener chaos methods.

In \cite[thm.~3]{KL2001}, Kratz and Le\`{o}n considered the level curves of a centered stationary Gaussian field with unit variance on the plane $\R^2$. More precisely, they considered the length of a level curve intersected with some large square $[-T,T]\times [-T,T]$. As $T \to + \infty$, they proved asymptotics of order $T^2$ for both the expectation and the variance of this length. They also proved that it satisfies a central limit theorem as $T \to +\infty$. In particular, their result applies to the centered Gaussian field on $\R^2$ with correlation function $(x,y)\mapsto \exp\left(-\frac{1}{2}\Norm{x-y}^2\right)$. This field can be seen as the scaling limit, in the sense of \cite{NS2015}, of the centered Gaussian field $\left(s_d(x)\right)_{x \in M}$ defined by our random sections, when $n=2$ and $r=1$.

The study of more general random algebraic submanifolds, obtained as the zero sets of random sections, was pioneered by Shiffman and Zelditch \cite{SZ1999,SZ2008,SZ2010}. They considered the integration current over the common complex zero set $Z_d$ of $r$ independent random sections in $H^0(\X,\L^d)$, distributed as standard complex Gaussians. In \cite{SZ1999}, they computed the asymptotic, as $d$ goes to infinity, of the expectation of the associated smooth statistics when $r=1$. They also provided an upper bound for the variance of these quantities and proved the equivalent of Cor.~\ref{cor as convergence} in this complex algebraic setting. In \cite{SZ2008}, they gave an asymptotic of order $d^{2r-n-\frac{1}{2}}$ for the variance of the volume of $Z_d \cap U$, where $U \subset \X$ is a domain satisfying some regularity conditions. In \cite{SZ2010}, they proved a similar asymptotic for the variance of the smooth statistics associated with $Z_d$. When $r=1$, they deduced a central limit theorem from these estimates and an asymptotic normality result of Sodin and Tsirelson \cite{ST2004}. Finally, in \cite[thm.~1.4]{SZZ2008}, Shiffman, Zelditch and Zrebiec proved that the probability that $Z_d \cap U = \emptyset$, where $U$ is any open subset of $\X$, decreases exponentially fast as $d$ goes to infinity. 

Coming back to our real algebraic setting, one should be able to deduce from the general result of Nazarov and Sodin \cite[thm.~3]{NS2015} that, given an open set $U \subset M$, the probability that $Z_d \cap U = \emptyset$ goes to $0$ as $d$ goes to infinity. Corollary~\ref{cor connected components} gives an upper bound for the convergence rate. In particular, this bounds the probability for $Z_d$ to be empty. In the same spirit, Gayet and Welschinger \cite{GW2015} proved the following result. Let $\Sigma$ be a fixed diffeomorphism type of codimension $r$ submanifold of $\R^n$, let $x \in M$ and let $B_d(x)$ denote the geodesic ball of center $x$ and radius $\frac{1}{\sqrt{d}}$. Then, the probability that $Z_d \cap B_d(x)$ contains a submanifold diffeomorphic to $\Sigma$ is bounded from below. On the other hand, when $n=2$ and $r=1$, the Harnack--Klein inequality shows that the number of connected components of $Z_d$ is bounded by a polynomial in $d$. In \cite{GW2011}, Gayet and Welschinger proved that the probability for $Z_d$ to have the maximal number of connected components decreases exponentially fast with $d$.

Another well-studied model of random submanifolds is that of Riemannian random waves, i.e.~zero sets of random eigenfunctions of the Laplacian associated with some eigenvalue $\lambda$. In this setting, Rudnick and Wigman \cite{RW2008} computed an asymptotic bound, as $\lambda \to +\infty$, for the variance of the volume of a random hypersurface on the flat $n$-dimensional torus $\mathbb{T}^n$. On $\mathbb{T}^2$, this result was improved by Krishnapur, Kurlberg and Wigman \cite{KKW2013} who computed the precise asymptotic of the variance of the length of a random curve. In \cite{Wig2010}, Wigman computed the asymptotic variance of the linear statistics associated with a random curve on the Euclidean sphere $\S^2$. His result holds for a large class of test-function that contains the characteristic functions of open sets satisfying some regularity assumption. In relation with Cor.~\ref{cor connected components}, Nazarov and Sodin \cite{NS2009} proved that, on the Euclidean sphere $\S^2$, the number of connected components of a random curve times $\frac{1}{\lambda}$ converges exponentially fast in probability to a deterministic constant as $\lambda \to + \infty$.

\paragraph*{About the proof.}

The idea of the proof is the following. The random section $s_d$ defines a centered Gaussian field $(s_d(x))_{x \in \X}$. The correlation kernel of this field equals the Bergman kernel, that is the kernel of the orthogonal projection onto $\H$ for the inner product~\eqref{eq definition inner product} (compare \cite{BSZ2000a,Let2016,SZ1999,SZ2008,SZ2010}).

In order to compute the covariance of the smooth statistics $\prsc{\rmes{s}}{\phi_1}$ and $\prsc{\rmes{s}}{\phi_2}$, we apply a Kac--Rice formula (cf.~\cite{AW2009,BSZ2000a,Dal2015,TA2007,Wig2010}). This allows us to write $\var{\rmes{d}}(\phi_1,\phi_2)$ as the integral over $M\times M$ of some function $\D_d(x,y)$, defined by~\eqref{eq def density}. This density $\D_d(x,y)$ is the difference of two terms, coming respectively from
\begin{align*}
\esp{\prsc{\rmes{d}}{\phi_1}\prsc{\rmes{d}}{\phi_2}}& & &\text{and} & \esp{\prsc{\rmes{d}}{\phi_1}} \esp{\prsc{\rmes{d}}{\phi_2}}.
\end{align*}

Since the Bergman kernel decreases exponentially fast outside of the diagonal $\Delta$ in $M^2$ (see Section~\ref{subsec far off diagonal estimates}), the values of $s_d(x)$ and $s_d(y)$ are almost uncorrelated for $(x,y)$ far from~$\Delta$. As a consequence, when the distance between $x$ and $y$ is much larger than $\frac{1}{\sqrt{d}}$, the above two terms in the expression of $\D_d(x,y)$ are equal, up to a small error (see Sect.~\ref{subsubsec far off diagonal correlated} for a precise statement). Thus, $\D_d(x,y)$ is small far from $\Delta$, and its integral over this domain only contributes a remainder term to $\var{\rmes{d}}(\phi_1,\phi_2)$.

The main contribution to the value of $\var{\rmes{d}}(\phi_1,\phi_2)$ comes from the integration of $\D_d(x,y)$ over a neighborhood of $\Delta$ of size about~$\frac{1}{\sqrt{d}}$. We perform a change of variable in order to express this term as an integral over a domain of fixed size. This rescaling by $\frac{1}{\sqrt{d}}$ explains the factor $d^{-\frac{n}{2}}$ in~\eqref{eq asymptotics variance}. Besides, the order of growth of $\D_d(x,y)$ close to $\Delta$ is $d^r$, that is the order of growth of the square of $\esp{\rmes{d}}$ (see Thm.~\ref{thm reminder expectation test function}). Finally, we get an order of growth of $d^{r-\frac{n}{2}}$ for $\var{\rmes{d}}(\phi_1,\phi_2)$. The constant in~\eqref{eq asymptotics variance} appears as the scaling limit of the integral of $\D_d(x,y)$ over a neighborhood of $\Delta$ of typical size $\frac{1}{\sqrt{d}}$.

The difficulty in making this sketch of proof rigorous comes from the combination of two facts. First, we do not know exactly the value of the Bergman kernel (our correlation function) and its derivatives, but only asymptotics. In addition, the conditioning in the Kac--Rice formula is singular along $\Delta$, and so is $\D_d$. Because of this, we lose all uniformity in the control of the error terms close to the diagonal. Nonetheless, by careful bookkeeping of the error terms, we can make the above heuristic precise.

\paragraph*{Outline of the paper.}

In Section~\ref{sec random real algebraic submanifolds} we describe precisely our framework and the construction of the random measures $\rmes{s_d}$. We also introduce the Bergman kernel and explain how it is related to our random submanifolds.

In Section~\ref{sec estimates for the bergman kernel}, we recall various estimates for the Bergman kernel that we use in the proof of our main theorem. These estimates were established by Dai, Liu and Ma \cite{DLM2006}, and Ma and Marinescu \cite{MM2007,MM2013,MM2015} in a complex setting. Our main contribution in this section consists in checking that the preferred trivialization used by Ma and Marinescu to state their near-diagonal estimates is well-behaved with respect to the real structures on $\X$, $\E$ and $\L$ (see Section~\ref{subsec real normal trivialization}).

Section~\ref{sec proof of the main theorem} is concerned with the proof of Thm.~\ref{thm asymptotics variance}. In Sect.~\ref{subsec Kac--Rice formula}, we prove a Kac--Rice formula adapted to our problem, using Federer's coarea formula and Kodaira's embedding theorem. In Sect.~\ref{subsec an integral formula for the variance} we prove an integral formula for the variance, using the Kac--Rice formula (Thm.~\ref{thm Kac-Rice var}). The core of the proof is contained in Sect.~\ref{subsec asymptotic for the variance}.

Finally, we prove Corollaries~\ref{cor concentration}, \ref{cor connected components} and~\ref{cor as convergence} in Section~\ref{sec proofs of the corollaries}.

\paragraph*{Acknowledgments.}

I am thankful to Damien Gayet for his guidance in the course of this work and for countless mathematical discussions, on this topic and others.

\tableofcontents

\section{Random real algebraic submanifolds}
\label{sec random real algebraic submanifolds}

\subsection{General setting}
\label{subsec general setting}

In this section, we introduce our framework. It is the same as the algebraic setting of \cite{Let2016}, see also~\cite{GW2014b,GW2015}. Classical references for the material of this section are \cite[chap.~0]{GH1994} and \cite[chap.~1]{Sil1989}.

Let $\X$ be a smooth complex projective manifold of complex dimension $n\geq 2$. We assume that $\X$ is defined over the reals, that is $\X$ is equipped with an anti-holomorphic involution $c_\X$. The real locus of $(\X,c_\X)$ is the set of fixed points of $c_\X$. In the sequel, we assume that it is non-empty and we denote it by $M$. It is a classical fact that $M$ is a smooth closed (i.e.~compact without boundary) submanifold of $\X$ of real dimension $n$ (see \cite[chap.~1]{Sil1989}).

Let $\E \to \X$ be a holomorphic vector bundle of rank $r \in \{1,\dots,n-1\}$. Let $c_\E$ be a real structure on $\E$, compatible with $c_\X$ in the sense that the projection $\pi_\E : \E \to \X$ satisfies $c_\X \circ \pi_\E = \pi_\E \circ c_\E$ and $c_\E$ is fiberwise $\C$-anti-linear. Let $h_\E$ be a real Hermitian metric on $\E$, that is $c_\E^\star(h_\E)=\overline{h_\E}$.

Similarly, let $\L \to \X$ be an ample holomorphic line bundle equipped with a compatible real structure $c_\L$ and a real Hermitian metric $h_\L$. Moreover, we assume that the curvature form $\omega$ of $h_\L$ is a Kähler form. Recall that if $\zeta$ is any non-vanishing holomorphic section on the open set $\Omega \subset \X$, then the restriction of $\omega$ to $\Omega$ is given by:
\begin{equation*}
\omega_{/\Omega} = \frac{1}{2i}\partial \bar{\partial} \ln\left(h_\L(\zeta,\zeta)\right).
\end{equation*}
This Kähler form is associated with a Hermitian metric $g_\C$ on $\X$. The real part of $g_\C$ defines a Riemannian metric $g = \omega(\cdot,i\cdot)$ on $\X$, compatible with the complex structure. Note that, since $h_\L$ is compatible with the real structures on $\X$ and $\L$, we have $c_\L^\star(h_\L)=\overline{h_\L}$ and $c_\X^\star\omega = - \omega$. Then we have $c_\X^\star g_\C = \overline{g_\C}$, hence $c_\X^\star g = g$ and $c_\X$ is an isometry of $(\X,g)$.

Then $g$ induces a Riemannian measure on every smooth submanifold of $\X$. In the case of $\X$, this measure is given by the volume form $\dx V_\X=\frac{\omega^n}{n!}$. We denote by $\rmes{M}$ the Riemannian measure on $(M,g)$.

Let $d \in \N$, then the rank $r$ holomorphic vector bundle $\E \otimes \L^d$ can be endowed with a real structure $c_d = c_\E \otimes c_\L^d$, compatible with $c_\X$, and a real Hermitian metric $h_d = h_\E \otimes h_\L^d$. If $x \in M$, then $c_d$ induces a $\C$-anti-linear involution of the fiber $(\E \otimes \L^d)_x$. We denote by $\R(\E \otimes \L^d)_x$ the fixed points set of this involution, which is a dimension $r$ real vector space.

Let $\Gamma(\E \otimes \L^d)$ denote the space of smooth sections of $\E \otimes \L^d$. We can define a Hermitian inner product on $\Gamma(\E \otimes \L^d)$ by:
\begin{equation}
\label{eq definition inner product}
\forall s_1, s_2 \in \Gamma(\E \otimes \L^d), \qquad \prsc{s_1}{s_2} = \int_\X h_d(s_1(x),s_2(x)) \dx V_\X.
\end{equation}
We say that a section $s \in \Gamma(\E \otimes \L^d)$ is real if it is equivariant for the real structures, that is: $c_d \circ s = s \circ c_\X$. Let $\R \Gamma(\E \otimes \L^d)$ denote the real vector space of real smooth sections of $\E \otimes \L^d$. The restriction of $\prsc{\cdot}{\cdot}$ to $\R \Gamma(\E \otimes \L^d)$ is a Euclidean inner product.

\begin{ntn}
\label{ntn bracket}
In this paper, $\prsc{\cdot}{\cdot}$ will always denote either the inner product on the concerned Euclidean (or Hermitian) space or the duality pairing between a space and its topological dual. Which one will be clear from the context.
\end{ntn}

Let $\H$ denote the space of global holomorphic sections of $\E \otimes \L^d$. This space has finite complex dimension $N_d$ by Hodge's theory (compare \cite[thm.~1.4.1]{MM2007}). We denote by $\R \H$ the space of global real holomorphic sections of $\E \otimes\L^d$:
\begin{equation}
\label{eq dfn RH0}
\R \H = \left\{ s \in \H \mvert c_d \circ s = s \circ c_\X \right\}.
\end{equation}
The restriction of the inner product \eqref{eq definition inner product} to $\R \H$ makes it into a Euclidean space of real dimension $N_d$.

\begin{rem}
\label{rem inner product}
Note that, even when we consider real sections restricted to $M$, the inner product is defined by integrating on the whole complex manifold $\X$.
\end{rem}

\subsection{Random submanifolds}
\label{subsec random submanifolds}

This section is concerned with the definition of the random submanifolds we consider and the related random variables.

Let $d \in \N$ and $s \in \R \H$, we denote the real zero set of $s$ by $Z_s=s^{-1}(0) \cap M$. If the restriction of $s$ to $M$ vanishes transversally, then $Z_s$ is a smooth submanifold of codimension $r$ of $M$. In this case, we denote by $\rmes{s}$ the Riemannian measure on $Z_s$ induced by $g$, seen as a Radon measure on $M$. Note that this includes the case where $Z_s$ is empty.

Recall the following facts, that we already discussed in \cite{Let2016}.
\begin{dfn}[see \cite{Nic2015}]
\label{dfn 0 ample}
We say that $\R \H$ is \emph{$0$-ample} if, for any $x \in M$, the evaluation map
\begin{equation}
\label{eq dfn 0 ample}
\begin{array}{rccc}
\ev_x^d:& \R \H & \longrightarrow &\R \left(\E \otimes \L^d\right)_x\\
& s & \longmapsto & s(x)
\end{array}
\end{equation}
is surjective.
\end{dfn}

\begin{lem}[see \cite{Let2016}, cor.~3.10]
\label{lem dfn d1}
There exists $d_1 \in \N$, depending only on $\X$, $\E$ and $\L$, such that for all $d \geq d_1$, $\R \H$ is $0$-ample.
\end{lem}

\begin{lem}[see \cite{Let2016}, section~2.6]
\label{lem 0 ample implies nice}
If $\R \H$ is $0$-ample, then for almost every section $s \in \R \H$ (for the Lebesgue measure), the restriction of $s$ to $M$ vanishes transversally.
\end{lem}

From now on, we only consider the case $d \geq d_1$, so that $\rmes{s}$ is a well-defined measure on $M$ for almost every $s \in \R \H$. Let $s_d$ be a standard Gaussian vector in $\R \H$, that is $s_d$ is a random vector whose distribution admits the density:
\begin{equation}
\label{eq Gaussian density}
s \mapsto \frac{1}{\sqrt{2\pi}^{N_d}} \exp\left(-\frac{1}{2}\Norm{s}^2\right)
\end{equation}
with respect to the normalized Lebesgue measure on $\R \H$. Here $\Norm{\cdot}$ is the norm associated with the Euclidean inner product \eqref{eq definition inner product}. Then $Z_{s_d}$ is almost surely a submanifold of codimension $r$ of $M$ and $\rmes{s_d}$ is almost surely a random positive Radon measure on $M$. To simplify notations, we set $Z_d = Z_{s_d}$ and $\rmes{d}=\rmes{s_d}$. For more details concerning Gaussian vectors, we refer to \cite[appendix~A]{Let2016} and the references therein.

Let $\phi \in \mathcal{C}^0(M)$, for every $s \in \R \H$ vanishing transversally, we set
\begin{equation}
\label{eq def linear stats}
\prsc{\rmes{s}}{\phi} = \int_{x \in Z_s} \phi(x) \rmes{s}.
\end{equation}
Such a $\phi$ will be referred to as a \emph{test-function}. Following \cite{SZ2010}, we call \emph{linear statistic} of degree $d$ associated with $\phi$ the real random variable $\prsc{\rmes{d}}{\phi}$.

\subsection{The correlation kernel}
\label{subsec correlation kernel}

Let $d \in \N$, then $(s_d(x))_{x \in \X}$ is a smooth centered Gaussian field on $\X$. As such, it is characterized by its correlation kernel. In this section, we recall that the correlation kernel of $s_d$ equals the Bergman kernel of $\E \otimes \L^d$. This is now a well-known fact (see \cite{BSZ2000a,GW2014b,SZ1999,SZ2010}) and was already used by the author in \cite{Let2016}.

Let us first recall some facts about random vectors (see for example \cite[appendix~A]{Let2016}). In this paper, we only consider centered random vectors (that is their expectation vanishes), so we give the following definitions in this restricted setting. Let $X_1$ and $X_2$ be centered random vectors taking values in Euclidean (or Hermitian) vector spaces $V_1$ and $V_2$ respectively, then we define their \emph{covariance operator} as:
\begin{equation}
\label{eq def covariance operator}
\cov{X_1}{X_2} : v \longmapsto \esp{X_1 \prsc{v}{X_2}}
\end{equation}
from $V_2$ to $V_1$. For every $v \in V_2$, we set $v^*=\prsc{\cdot}{v} \in V_2^*$. Then $\cov{X_1}{X_2} = \esp{X_1 \otimes X_2^*}$ is an element of $V_1 \otimes V_2^*$. The \emph{variance operator} of a centered random vector $X \in V$ is defined as $\var{X} = \cov{X}{X} = \esp{X \otimes X^*} \in V \otimes V^*$. We denote by $X \sim \mathcal{N}(\Lambda)$ the fact that $X$ is a centered Gaussian vector with variance operator $\Lambda$. Finally, we say that $X \in V$ is a \emph{standard} Gaussian vector if $X \sim \mathcal{N}(\Id)$, where $\Id$ is the identity operator on $V$. A standard Gaussian vector admits the density~\eqref{eq Gaussian density} with respect to the normalized Lebesgue measure on $V$.

Recall that $(\mathcal{E}\otimes \mathcal{L}^d) \boxtimes (\mathcal{E}\otimes \mathcal{L}^d)^*$ stands for the bundle $P_1^\star\left(\mathcal{E}\otimes \mathcal{L}^d\right) \otimes P_2^\star \left(\left(\mathcal{E}\otimes \mathcal{L}^d\right)^*\right)$ over $\mathcal{X}\times \mathcal{X}$, where $P_1$ (resp.~$P_2$) denotes the projection from $\X \times \X$ onto the first (resp.~second) factor. The covariance kernel of $(s_d(x))_{x \in \X}$ is the section of $(\mathcal{E}\otimes \mathcal{L}^d) \boxtimes (\mathcal{E}\otimes \mathcal{L}^d)^*$ defined by:
\begin{equation}
\label{eq def covariance kernel}
(x,y) \mapsto \cov{s_d(x)}{s_d(y)} = \esp{s_d(x) \otimes s_d(y)^*}.
\end{equation}

The orthogonal projection from $\R \Gamma(\E \otimes \L^d)$ onto $\R \H$ admits a Schwartz kernel (see \cite[thm.~B.2.7]{MM2007}). That is, there exists a unique section $E_d$ of $(\mathcal{E}\otimes \mathcal{L}^d) \boxtimes (\mathcal{E}\otimes \mathcal{L}^d)^*$ such that, for any $s \in \R \Gamma(\E \otimes \L^d)$, the projection of $s$ onto $\R \H$ is given by:
\begin{equation}
\label{eq def Schwartz kernel}
x \longmapsto \int_{y \in \X} E_d(x,y)\left(s(y)\right) \dx V_\X.
\end{equation}
This section is called the \emph{Bergman kernel} of $\E \otimes \L^d$. Note that $E_d$ is also the Schwartz kernel of the orthogonal projection from $\Gamma(\E\otimes\L^d)$ onto $\H$, for the Hermitian inner product \eqref{eq definition inner product}.
\begin{prop}
\label{prop covariance equals Bergman}
Let $d \in \N$ and let $s_d$ be a standard Gaussian vector in $\R \H$. Then, for all $x$ and $y \in \X$, we have:
\begin{equation}
\label{eq covariance equals Bergman}
\cov{s_d(x)}{s_d(y)} = \esp{s_d(x) \otimes s_d(y)^*} = E_d(x,y).
\end{equation}
\end{prop}
\begin{proof}
We will prove that $(x,y)\mapsto \esp{s_d(x)\otimes s_d(y)^*}$ is the kernel of the orthogonal projection onto $\R \H$, i.e.~satisfies \eqref{eq def Schwartz kernel}. Let $s \in \R \Gamma(\E \otimes \L^d)$, then
\begin{equation*}
\int_{y \in \X} \esp{s_d(x)\otimes s_d(y)^*}\left(s(y)\right) \dx V_\X = \esp{s_d(x) \int_{y\in \X} s_d(y)^*(s(y))\dx V_\X} = \esp{s_d(x)\prsc{s}{s_d}}.
\end{equation*}
If $s$ is orthogonal to $\R \H$ this quantity equals $0$. If $s \in \R \H$ then
\begin{equation*}
\esp{s_d(x) \prsc{s}{s_d}} = \esp{\ev_x^d(s_d) s_d^*(s)} = \ev_x^d \left(\esp{s_d \otimes s_d^*}(s)\right) = \ev_x^d (\var{s_d} s) = \ev_x^d(s) = s(x)
\end{equation*}
since $\var{s_d} = \Id$. Thus, for any $s \in \R \Gamma(\E \otimes \L^d)$, $\esp{s_d(x)\prsc{s}{s_d}}$ is the value at $x$ of the orthogonal projection of $s$ on $\R \H$. Finally, the correlation kernel of $(s_d(x))_{x \in \X}$ satisfies \eqref{eq def Schwartz kernel} and equals $E_d$.
\end{proof}

\begin{rem}
\label{rem Bergman kernel basis}
If $(s_{1,d},\dots,s_{N_d,d})$ is any orthonormal basis of $\R \H$, we have:
\begin{equation}
\label{eq Bergman in coordinates}
E_d : (x,y) \longmapsto \sum_{i=1}^{N_d} s_{i,d}(x) \otimes s_{i,d}(y)^*.
\end{equation}
\end{rem}

\begin{rem}
\label{rem Bergman kernel splitting}
If $\E$ is the trivial bundle $\X \times \C^r \to \X$ then $E_d$ splits as $E_d = \Id \otimes\, e_d$, where $\Id$ is the identity of $\C^r$ and $e_d$ is the Bergman kernel of $\L^d$. There is no such splitting in general.
\end{rem}

\begin{rem}
\label{rem Bergman kernel cov complex}
In a complex setting, $E_d$ is also the covariance kernel of the centered Gaussian field associated with a standard complex Gaussian vector in $\H$.
\end{rem}

The Bergman kernel also describes the distribution of the derivatives of $s_d$. Let $\nabla^d$ denote any connection on $\E \otimes \L^d\to \X$. Then $\nabla^d$ induces a connection $(\nabla^d)^*$ on $(\E \otimes \L^d)^*\to \X$, which is defined for all $\eta \in \Gamma\left(\left(\E \otimes \L^d\right)^*\right)$ by:
\begin{equation}
\label{eq def dual connection}
\forall s \in \Gamma\left(\E \otimes \L^d\right),\ \forall x \in \X, \qquad d_x\prsc{s}{\eta} = \prsc{\nabla^d_x s}{\eta(x)}+\prsc{s(x)}{(\nabla^d)_x^*\eta},
\end{equation}
where $\prsc{\cdot}{\cdot}$ is the duality pairing. Let $s \in \Gamma\left(\E \otimes \L^d\right)$, then $s^\diamond: x \mapsto s(x)^* = \prsc{\cdot}{s(x)}$ defines a smooth section of $\left(\E \otimes \L^d\right)^*$. Note that we use the notation $s^\diamond$ because $s^*$ already denotes $\prsc{\cdot}{s}$ which is a linear form on $\Gamma\left(\E \otimes \L^d\right)$. We want to understand the relation between $(\nabla^d)_x^*s^\diamond : T_x\X \to \left(\E \otimes \L^d\right)_x^*$ and $\left(\nabla^d_xs\right)^*$. Recall that $\left(\nabla^d_xs\right)^* = \prsc{\cdot}{\nabla^d_xs}$, where the inner product is the one on $\left(\E \otimes \L^d\right)_x \otimes T^*_x\X$ induced by $h_d$ and $g_\C$. That is, $\left(\nabla^d_xs\right)^*$ is the adjoint operator of $\nabla^d_xs : T_x \X \to \left(\E \otimes \L^d\right)_x$. In order to get a nice relation, we have to assume that $\nabla^d$ is a metric connection, i.e.~that:
\begin{equation}
\label{eq def metric connection}
\forall s,t \in \Gamma\left(\E \otimes \L^d\right),\ \forall x \in \X, \qquad d_x\prsc{s}{t} = \prsc{\nabla^d_xs}{t(x)} + \prsc{s(x)}{\nabla^d_xt}.
\end{equation}

\begin{lem}
\label{lem dual connection}
Let $\nabla^d$ be a metric connection on $\E \otimes \L^d$, let $s \in \Gamma\left(\E \otimes \L^d\right)$ and let $x \in \X$. Then for all $v \in T_x\X$,
\begin{equation}
\label{eq lem dual connection}
(\nabla^d)^*_x s^\diamond \cdot v = \left(\nabla_x^ds\cdot v\right)^* = v^* \circ \left(\nabla_x^ds\right)^*.
\end{equation}
\end{lem}

\begin{proof}
First, for all $s,t \in \Gamma\left(\E \otimes \L^d\right)$ and all $x \in \X$,
\begin{equation}
\label{eq truc1}
\prsc{t(x)}{s(x)} = \prsc{t(x)}{s(x)^*} = \prsc{t(x)}{s^\diamond(x)}.
\end{equation}
Then, by taking the derivative of~\eqref{eq truc1}, we get that for all $s,t \in \Gamma\left(\E \otimes \L^d\right)$, for all $x \in \X$ and $v \in T_x\X$:
\begin{equation*}
\prsc{t(x)}{\nabla^d_xs\cdot v}+\prsc{\nabla^d_xt\cdot v}{s(x)} = d_x\left(\prsc{t}{s}\right)\cdot v = \prsc{t(x)}{(\nabla^d)^*_xs^\diamond\cdot v}+ \prsc{\nabla^d_xt\cdot v}{s^\diamond(x)}.
\end{equation*}
The first equality comes from the fact that $\nabla^d$ is metric (see~\eqref{eq def metric connection}) and the second from the definition of the dual connection~\eqref{eq def dual connection}. Besides $\prsc{\nabla^d_xt\cdot v}{s^\diamond(x)} =\prsc{\nabla^d_xt\cdot v}{s(x)}$, hence for all $s \in \Gamma\left(\E \otimes \L^d\right)$ and all $x \in \X$ we have:
\begin{equation*}
\forall v \in T_x\X, \qquad (\nabla^d)^*_xs^\diamond \cdot v = \left(\nabla^d_xs\cdot v\right)^*.
\end{equation*}
Recall that $\left(\nabla_x^ds\right)^*$ is the adjoint of $\nabla_x^ds$. Hence for all $v \in T_x\X$ and all $\zeta \in \left(\E \otimes \L^d\right)_x$,
\begin{equation*}
\prsc{\zeta}{\nabla^d_xs\cdot v} = \prsc{\left(\nabla^d_xs\right)^*\zeta}{v} = v^* \circ \left(\nabla^d_xs\right)^* (\zeta),
\end{equation*}
which proves the second equality in~\eqref{eq lem dual connection}.
\end{proof}

\begin{rem}
\label{rem metric connection}
Conversely, one can show that a connection satisfying the first equality in eq.~\eqref{eq lem dual connection} for every $s,x$ and $v$ is metric.
\end{rem}

From now on, we assume that $\nabla^d$ is metric. Then $\nabla^d$ induces a natural connection $\nabla_1^d$ on $P_1^\star(\E \otimes \L^d) \to \X \times \X$ whose partial derivatives are: $\nabla^d$ with respect to the first variable, and the trivial connection with respect to the second. Similarly, $(\nabla^d)^*$ induces a connection $\nabla_2^d$ on $P_2^\star\left((\E \otimes \L^d)^*\right)$ and $\nabla_1^d \otimes \Id + \Id \otimes \nabla_2^d$ is a connection on $(\E \otimes \L^d) \boxtimes (\E \otimes \L^d)^*$. We denote by $\partial_x$ (resp. $\partial_y$) its partial derivative with respect to the first (resp. second) variable. By taking partial derivatives in \eqref{eq covariance equals Bergman}, we get the following.

\begin{cor}
\label{cor variance of the 1-jet}
Let $d \in \N$, let $\nabla^d$ be a metric connection on $\E \otimes \L^d$ and let $s_d$ be a standard Gaussian vector in $\R \H$. Then, for all $x$ and $y \in \X$, for all $(v,w) \in T_x\X \times T_y\X$, we have:
\begin{align}
\label{cov bergman 10}
\cov{\nabla^d_xs \cdot v}{s(y)} &= \esp{\left(\nabla^d_xs\cdot v\right) \otimes s(y)^*} = \partial_x E_d(x,y)\cdot v,\\
\label{cov bergman 01}
\cov{s(x)}{\nabla^d_ys\cdot w} &= \esp{s(x) \otimes \left(\nabla^d_ys\cdot w\right)^*} = \partial_y E_d(x,y)\cdot w,\\
\label{cov bergman 11}
\cov{\nabla^d_xs\cdot v}{\nabla^d_ys\cdot w} &= \esp{\left(\nabla^d_xs\cdot v\right) \otimes \left(\nabla^d_ys\cdot w\right)^*} = \partial_x\partial_y E_d(x,y)\cdot (v,w).
\end{align}
\end{cor}

\begin{proof}
The first equality of each line is simply the definition of the covariance operator. By applying $\partial_x$ to~\eqref{eq covariance equals Bergman} we get:
\begin{equation*}
\esp{\left(\nabla^d_xs\right) \otimes s(y)^*} = \partial_x E_d(x,y),
\end{equation*}
which proves~\eqref{cov bergman 10}. We can rewrite~\eqref{eq covariance equals Bergman} as: $\forall x,y \in \X$, $E_d(x,y) = \esp{s(x) \otimes s^\diamond(y)}$. By applying $\partial_y$ to this equality, we get:
\begin{equation*}
\esp{s(x) \otimes \left(\nabla^d\right)^*_y s^\diamond} = \partial_y E_d(x,y).
\end{equation*}
Then we apply this operator to $w \in T_y\X$, and we obtain~\eqref{cov bergman 01} by Lemma~\ref{lem dual connection}. The proof of~\eqref{cov bergman 11} is similar.
\end{proof}

We would like to write that $\partial_y E_d(x,y)$ is $\cov{s(x)}{\nabla^d_ys}=\esp{s(x) \otimes \left(\nabla^d_ys\right)^*}$. Unfortunately, this can not be true since
\begin{align*}
\partial_y E_d(x,y) &\in T^*_y\X \otimes \left(\E \otimes \L^d\right)_x \otimes \left(\E \otimes \L^d\right)_y^*\\
\intertext{while} \esp{s(x) \otimes \left(\nabla^d_ys\right)^*} &\in T_y\X \otimes \left(\E \otimes \L^d\right)_x \otimes \left(\E \otimes \L^d\right)_y^*.
\end{align*}
Let $\partial_y^\sharp E_d(x,y) \in T_y\X \otimes \left(\E \otimes \L^d\right)_x \otimes \left(\E \otimes \L^d\right)_y^*$ be defined by:
\begin{equation}
\label{eq def sharp}
\forall w \in T_y\X, \qquad \partial_y^\sharp E_d(x,y) \cdot w^* = \partial_y E_d(x,y) \cdot w.
\end{equation}
Similarly, let $\partial_x\partial_y^\sharp E_d(x,y) \in T^*_x\X \otimes T_y\X \otimes \left(\E \otimes \L^d\right)_x \otimes \left(\E \otimes \L^d\right)_y^*$ be defined by:
\begin{equation}
\label{eq def sharp 2}
\forall (v,w) \in T_x\X \times T_y\X, \qquad \partial_x\partial_y^\sharp E_d(x,y) \cdot (v,w^*) = \partial_x\partial_y E_d(x,y)\cdot (v,w).
\end{equation} 
Then by Lemma~\ref{lem dual connection} and Corollary~\ref{cor variance of the 1-jet}, we have the following.
\begin{cor}
\label{cor variance of the 1-jet sharp}
Let $d \in \N$, let $\nabla^d$ be a metric connection on $\E \otimes \L^d$ and let $s_d$ be a standard Gaussian vector in $\R \H$. Then, for all $x$ and $y \in \X$, we have:
\begin{align}
\label{cov bergman 10 sharp}
\cov{\nabla^d_xs}{s(y)} &= \esp{\nabla^d_xs \otimes s(y)^*} = \partial_x E_d(x,y),\\
\label{cov bergman 01 sharp}
\cov{s(x)}{\nabla^d_ys} &= \esp{s(x) \otimes \left(\nabla^d_ys\right)^*} = \partial_y^\sharp E_d(x,y),\\
\label{cov bergman 11 sharp}
\cov{\nabla^d_xs}{\nabla^d_ys} &= \esp{\nabla^d_xs \otimes \left(\nabla^d_ys\right)^*} = \partial_x\partial_y^\sharp E_d(x,y).
\end{align}
\end{cor}

\section{Estimates for the Bergman kernel}
\label{sec estimates for the bergman kernel}

The goal of this section is to recall the estimates we need for the Bergman kernel. Most of what follows can be found in \cite{MM2007}, with small additions from \cite{MM2013} and \cite{MM2015}. The first to use this kind of estimates in a random geometry context were Shiffman and Zelditch \cite{SZ1999}. They used the estimates from \cite{Zel1998} for the related Szegö kernel (see also \cite{BSZ2000a,SZ2008}). Catlin \cite{Cat1999} proved similar estimates for the Bergman kernel independently.

In order to state the near-diagonal estimates for the Bergman kernel, we first need to choose preferred charts on $\X$, $\E$ and $\L$ around any point in $M$. This is done in Section~\ref{subsec real normal trivialization}. Unlike our main reference \cite{MM2007}, we are only concerned with a neighborhood of the real locus of $\X$, but we need to check that these charts are well-behaved with respect to the real structures. Sections~\ref{subsec near diagonal estimates}, \ref{subsec diagonal estimates} and \ref{subsec far off diagonal estimates} state respectively near-diagonal, diagonal and far off-diagonal estimates for $E_d$.

\subsection{Real normal trivialization}
\label{subsec real normal trivialization}

In this section, we define preferred local trivializations for $\E$ and $\L$ around any point in $M$. We also prove that these trivializations are compatible with the real and metric structures.

Let $R >0$ be such that the injectivity radius of $\X$ is larger than $2 R$. Let $x_0 \in M$, then the exponential map $\exp_{x_0} : T_{x_0}\X \to \X$ at $x_0$ is a diffeomorphism from the ball $B_{T_{x_0}\X}(0,2R) \subset T_{x_0}\X$ to the geodesic ball $B_\X(x_0,2R) \subset \X$. Note that this diffeomorphism is not biholomorphic in general.

\begin{ntn}
\label{ntn open ball}
Here and in the sequel, we always denote by $B_A(a,R)$ the open ball of center $a$ and radius $R>0$ in the metric space $A$.
\end{ntn}

Since $c_\X$ is an isometry (see Sect.~\ref{subsec general setting}), we have that $c_\X \circ \exp_{x_0} = \exp_{x_0} \circ \, d_{x_0}c_\X$. Then $\exp_{x_0}$ sends $T_{x_0}M = \ker\left(d_{x_0}c_\X - \Id \right)$ to $M$ and agrees on $T_{x_0}M$ with the exponential map at $x_0$ in $(M,g)$. By restriction, we get a diffeomorphism from $B_{T_{x_0}M}(0,2R) \subset T_{x_0}M$ to the geodesic ball $B_M(x_0,2R) \subset M$. Moreover, on $B_{T_{x_0}\X}(0,2R)$ we have:
\begin{equation}
\label{eq real equivariance exp}
d_{x_0}c_\X = (\exp_{x_0})^{-1} \circ c_\X \circ \exp_{x_0}.
\end{equation}
We say that $\exp_{x_0}$ defines a \emph{real normal chart} about $x_0$.

Since $i\cdot T_{x_0}M = \ker\left(d_{x_0}c_\X + \Id \right)$, we have $T_{x_0}\X = T_{x_0}M \oplus i\cdot T_{x_0}M$. Note that $T_{x_0}M$ and $i\cdot T_{x_0}M$ are orthogonal for $g_{x_0}$, since these are distinct eigenspaces of an isometric involution. Moreover, we know from Sect.~\ref{subsec general setting} that $c_\X^\star g_\C = \overline{g_\C}$. This implies that $(g_\C)_{x_0}$ takes real values on $T_{x_0}M\times T_{x_0}M$, i.e.~the restrictions to $T_{x_0}M$ of $(g_\C)_{x_0}$ and $g_{x_0}$ are equal. Thus, $(g_\C)_{x_0}$ is the sesquilinear extension of $g_{x_0}$ restricted to $T_{x_0}M$. Let $\mathcal{I}$ be an isometry from $T_{x_0}M$ to $\R^n$ with its standard Euclidean structure, $\mathcal{I}$ extends as a $\C$-linear isometry $\mathcal{I}_\C : T_{x_0}\X \to \C^n$, such that $\mathcal{I}_\C \circ d_{x_0}c_\X \circ \mathcal{I}_\C^{-1}$ is the complex conjugation in $\C^n$. Thus, $\exp_{x_0} \circ\, \mathcal{I}_\C^{-1}:B_{\C^n}(0,2R) \to B_{\X}(x_0,2R)$ defines normal coordinates that induce normal coordinates $B_{\R^n}(0,2R) \to B_M(x_0,2R)$ and such that $\mathcal{I}_\C \circ (\exp_{x_0})^{-1}\circ c_\X \circ \exp_{x_0} \circ\, \mathcal{I}_\C^{-1}$ is the complex conjugation in $\C^n$. Such coordinates are called \emph{real normal coordinates} about $x_0$.

We can now trivialize $\E$ over $B_\X(x_0,2R)$. Let $\nabla^\E$ denote the Chern connection of $\E$. We identify the fiber at $\exp_{x_0}(z) \in B_{\X}(x_0,2R)$ with $\E_{x_0}$, by parallel transport with respect to $\nabla^\E$ along the geodesic from $x_0$ to $\exp_{x_0}(z)$, defined by $t \mapsto \exp_{x_0}(tz)$ from $[0,1]$ to $\X$ (cf.~\cite[sect.~1.6]{MM2007} and~\cite{MM2013}). This defines a bundle map $\varphi_{x_0} : B_{T_{x_0}\X}(0,2R) \times \E_{x_0} \to \E_{/ B_\X(x_0,2R)}$ that covers $\exp_{x_0}$. We say that $\varphi_{x_0}$ is the \emph{real normal trivialization} of $\E$ over $B_\X(x_0,2R)$.

Since $x_0 \in M$, $c_\E(\E_{x_0})=\E_{x_0}$ and we denote by $c_{\E,x_0}$ the restriction of $c_\E$ to $\E_{x_0}$. Then $(d_{x_0}c_\X,c_{\E,x_0})$ is a real structure on $B_{T_{x_0}\X}(0,2R) \times \E_{x_0}$ compatible with the real structure on $B_{T_{x_0}\X}(0,2R)$. We want to check that $\varphi_{x_0}$ is well-behaved with respect to the real structures, i.e.~that for all $z \in B_{T_{x_0}\X}(0,2R)$ and $\zeta^0 \in \E_{x_0}$,
\begin{equation}
\label{eq equivariance trivialization real structures}
c_\E(\varphi_{x_0}(z,\zeta^0))= \varphi_{x_0}\left(d_{x_0}c_\X \cdot z, c_{\E,x_0}(\zeta^0)\right).
\end{equation}
This will be a consequence of Lemma~\ref{lem Chern connection is real} below.

\begin{dfn}
\label{dfn real connection}
Let $\E \to \X$ be a holomorphic vector bundle equipped with compatible real structures $c_\E$ and $c_\X$ and let $\nabla$ be a connection on $\E$, we say that $\nabla$ is a \emph{real connection} if for every section $s \in \Gamma(\E)$ we have:
\begin{equation*}
\forall x \in \X, \qquad \nabla_x \left( c_\E \circ s \circ c_\X\right) = c_\E \circ \nabla_{c_\X(x)} s \circ d_xc_\X.
\end{equation*}
\end{dfn}

\begin{rem}
\label{rem real connection in a real direction}
Let $x \in M$, $v \in T_xM$ and $s \in \R \Gamma(\E)$. If $\nabla$ is a real connection on $\E$, then $\nabla_xs\cdot v \in \R \E_x$. Indeed,
\begin{equation*}
\nabla_xs\cdot v = \nabla_{c_\X(x)}s \circ d_xc_\X\cdot v = c_\E \left(\nabla_x \left( c_\E \circ s \circ c_\X\right)\cdot v\right)= c_\E\left(\nabla_xs\cdot v\right).
\end{equation*}
\end{rem}

\begin{lem}
\label{lem Chern connection is real}
Let $\E \to \X$ be a holomorphic vector bundle equipped with compatible real structures $c_\E$ and $c_\X$ and a real Hermitian metric $h_\E$. Then, the Chern connection $\nabla^\E$ of $\E$ is real.
\end{lem}

\begin{proof}
Since $c_\E$ and $c_\X$ are involutions and $(d_xc_\X)^{-1} = d_{c_\X(x)}c_\X$, we need to check that
\begin{equation}
\label{eq Chern connection is real}
\forall s \in \Gamma(\E), \forall x \in \X \qquad \nabla^\E_x s = c_\E \circ \nabla^\E_{c_\X(x)}\left( c_\E \circ s \circ c_\X \right) \circ d_xc_\X.
\end{equation}
Let $\tilde{\nabla}$ be defined by $\tilde{\nabla}_x s = c_\E \circ \nabla^\E_{c_\X(x)}\left( c_\E \circ s \circ c_\X \right) \circ d_xc_\X$, for all $s \in \Gamma(\E)$ and $x \in \X$. Then $\tilde{\nabla}$ is a connection on $\E$ and it is enough to check that it is compatible with both the metric and the complex structure. Indeed, in this  case $\tilde{\nabla} = \nabla^\E$ by unicity of the Chern connection, which proves \eqref{eq Chern connection is real}.

Let us check that $\tilde{\nabla}$ satisfies Leibniz' rule. Let $s \in \Gamma(\E)$ and $f:\X \to \C$. We have:
\begin{align*}
\tilde{\nabla}_x(fs)&= c_\E \circ \nabla^\E_{c_\X(x)} \left( (\overline{f\circ c_\X}) (c_\E \circ s \circ c_\X) \right) \circ d_xc_\X\\
&= c_\E \circ \left(\overline{f(x)}\nabla^\E_{c_\X(x)}\left( c_\E \circ s \circ c_\X \right) + d_{c_\X(x)}(\overline{f\circ c_\X})\otimes c_\E(s(x)) \right) \circ d_xc_\X\\
&= f(x) \tilde{\nabla}_xs + d_xf \otimes s(x).
\end{align*}
Since $\nabla^\E$ is the Chern connection, its anti-holomorphic part is $\debar^\E$. Then, $d_xc_\X$ and $c_\E$ being anti-linear (resp. fiberwise), the anti-linear part of $\tilde{\nabla}_xs$ equals $c_\E \circ \debar^\E_{c_\X(x)}\left(c_\E \circ s \circ c_\X \right) \circ d_xc_\X$. By computing in a local holomorphic frame, one can check that:
\begin{equation*}
\forall s \in \Gamma(\E), \forall x \in \X, \qquad c_\E \circ \debar^\E_{c_\X(x)}\left(c_\E \circ s \circ c_\X \right) \circ d_xc_\X=\debar^\E_xs.
\end{equation*}
Thus, $\tilde{\nabla}$ is compatible with the complex structure. Finally, we check the compatibility with the metric structure. Let $s,t \in \Gamma(\E)$ and $x \in \X$, since $h_\E=c_\E^\star(\overline{h_\E})$ we have:
\begin{align*}
d_x(h_\E(s,t)) &= d_x\left(\overline{h_\E}(c_\E \circ s, c_\E \circ t)\right) = d_{c_\X(x)}\left(\overline{h_\E}(c_\E \circ s \circ c_\X, c_\E \circ t \circ c_\X)\right) \circ d_x c_\X\\
&\begin{aligned}=\overline{h_\E}\left(\nabla^\E_{c_\X(x)}(c_\E \circ s \circ c_\X),c_\E(t(x))\right)&\circ d_xc_\X\\+&\overline{h_\E}\left(c_\E(s(x)),\nabla^\E_{c_\X(x)}(c_\E \circ t \circ c_\X)\right)\circ d_x c_\X\end{aligned}\\
&\begin{aligned}= h_\E\left(c_\E \circ \nabla^\E_{c_\X(x)}(c_\E \circ s \circ c_\X),t(x)\right) &\circ d_xc_\X\\ + &h_\E\left(s(x),c_\E \circ\nabla^\E_{c_\X(x)}(c_\E \circ s \circ c_\X)\right)\circ d_x c_\X\end{aligned}\\
&= h_\E\left(\tilde{\nabla}_xs,t(x)\right)+ h_\E\left(s(x),\tilde{\nabla}_xt\right).\qedhere
\end{align*}
\end{proof}

Let us now prove \eqref{eq equivariance trivialization real structures}. Let $z \in B_{T_{x_0}\X}(0,2R)$, let $\zeta^0 \in \E_{x_0}$ and let $\zeta : B_\X(x_0,2R) \to \E$ be the section defined by $\zeta : x\mapsto \varphi_{x_0}\left((\exp_{x_0})^{-1}(x),\zeta^0\right)$. We denote by $\gamma : [0,1] \mapsto \X$ the geodesic $t \mapsto \exp_{x_0}(tz)$. We have for all $t \in [0,1]$, $\zeta(\gamma(t))=\varphi_{x_0}(tz,\zeta^0)$ and, by the definition of $\varphi_{x_0}$, we have:
\begin{equation}
\label{eq parallel transport}
\forall t \in [0,1],\qquad \nabla^\E_{\gamma(t)}\zeta\cdot \gamma'(t)= 0.
\end{equation}
Let us denote $\tilde{\zeta}=c_\E \circ \zeta \circ c_\X$ and $\overline{\gamma}=c_\X \circ \gamma$. Since $\nabla^\E$ is real, \eqref{eq parallel transport} implies that for all $t \in [0,1]$,
\begin{equation}
\label{eq parallel transport 2}
\nabla^\E_{\overline{\gamma}(t)}\tilde{\zeta} \cdot \overline{\gamma}'(t) = \nabla^\E_{c_\X(\gamma(t))}\tilde{\zeta} \circ d_{\gamma(t)}(c_\X) \cdot \gamma'(t) = c_\E \circ \nabla^\E_{\gamma(t)} \zeta \cdot \gamma'(t)= 0.
\end{equation}
Since $c_\X$ is an isometry, $\overline{\gamma}$ is a geodesic. More precisely, $\overline{\gamma}:t \mapsto \exp_{x_0}(td_{x_0}c_\X \cdot z)$. Besides, $\tilde{\zeta}(x_0)=c_\E(\zeta(x_0))=c_{\E,x_0}(\zeta^0)$. Then by \eqref{eq parallel transport 2}, for all $t\in [0,1]$,
\begin{equation*}
\varphi_{x_0}\left(td_{x_0}c_\X \cdot z,c_{\E,x_0}(\zeta^0)\right)=\varphi_{x_0}\left(td_{x_0}c_\X \cdot z,\tilde{\zeta}(x_0)\right)=\tilde{\zeta}(\overline{\gamma}(t)).
\end{equation*}
Finally, we get \eqref{eq equivariance trivialization real structures} for $t=1$:
\begin{equation*}
\varphi_{x_0}\left(d_{x_0}c_\X \cdot z,c_{\E,x_0}(\zeta^0)\right)= \tilde{\zeta}(\overline{\gamma}(1))=c_\E(\zeta(\gamma(1)))=c_\E\left(\varphi_{x_0}(z,\zeta^0)\right).
\end{equation*}

Recall that $\R \E$ is the set of fixed points of $c_\E$. Then $\R \E$ is naturally a rank $r$ real vector bundle over $M$, as a subbundle of $\E_{/M}$. Let $\zeta^0 \in \R \E_{x_0}$, and $\zeta : x\mapsto \varphi_{x_0}\left((\exp_{x_0})^{-1}(x),\zeta^0\right)$ then, for all $x \in B_\X(x_0,2R)$,
\begin{align*}
c_\E \circ \zeta \circ c_\X (x) &= c_\E \circ \varphi_{x_0}\left((\exp_{x_0})^{-1}(c_\X(x)),\zeta^0\right)\\
&= c_\E \circ \varphi_{x_0}\left(d_{x_0}c_\X \circ(\exp_{x_0})^{-1}(x),\zeta^0\right)\\
&= \varphi_{x_0}\left((\exp_{x_0})^{-1}(x),c_{\E,x_0}(\zeta^0)\right)\\
&= \zeta(x). 
\end{align*}
Hence, $\zeta$ is a real local section of $\E$ and in particular, $\forall x \in M$, $\zeta(x) \in \R\E_x$. This shows that $\varphi_{x_0}$ induces, by restriction, a bundle map $B_{T_{x_0}M}(0,2R) \times \R \E_{x_0} \to \R\E_{/B_M(x_0,2R)}$ that covers the restriction of $\exp_{x_0}$ to $B_{T_{x_0}M}(0,2R)$.

Let $(\zeta_1^0,\dots,\zeta_r^0)$ be an orthonormal basis of $\R \E_{x_0}$. Since $\R \E_{x_0} = \ker\left(c_{\E,x_0}-\Id\right)$ and $c_{\E,x_0}$ is $\C$-anti-linear, we have $\E_{x_0} = \R \E_{x_0}\oplus i \cdot \R \E_{x_0}$. Moreover, since $h_{\E,x_0}$ and $c_{\E,x_0}$ are compatible, $(\zeta_1^0,\dots,\zeta_r^0)$ is also an orthonormal basis of $\E_{x_0}$. Let $i\in \left\{1,\dots,r\right\}$, we denote by $\zeta_i : B_\X(x_0,2R) \to \E$ the real local section defined by:
\begin{equation*}
\label{eq zeta i}
\zeta_i : x \mapsto \varphi_{x_0}\left((\exp_{x_0})^{-1}(x),\zeta_i^0\right).
\end{equation*}
Then, for every $x \in B_\X(0,2R)$, $(\zeta_1(x),\dots,\zeta_r(x))$ is an orthonormal basis of $\E_x$. Indeed, the sections $\zeta_i$ are obtained by parallel transport for $\nabla^\E$ along geodesics starting at $x_0$, and $\nabla^\E$ is compatible with $h_\E$. Hence, for all $i$ and $j\in \left\{1,\dots,r\right\}$, for all $z \in B_\X(x_0,2R)$,
\begin{multline*}
\label{eq constant along geodesic}
\frac{\dx}{\dx t}\left(h_\E(\zeta_i(\exp_{x_0}(tz)),\zeta_j(\exp_{x_0}(tz)))\right) = h_\E\left(\nabla^\E_{\exp_{x_0}(tz)}\zeta_i\circ d_{tz}\exp_{x_0}\cdot z,\zeta_j(\exp_{x_0}(tz)))\right)\\
+h_\E\left(\zeta_i(\exp_{x_0}(tz))),\nabla^\E_{\exp_{x_0}(tz)}\zeta_j\circ d_{tz}\exp_{x_0}\cdot z\right) = 0.
\end{multline*}
The function $x \mapsto h_\E(\zeta_i(x),\zeta_j(x))$ is then constant along geodesics starting at $x_0$, hence on $B_\X(x_0,2R)$. Since $\left(h_\E(\zeta_i(x),\zeta_j(x))\right)_{1\leq i,j\leq r}$ is the identity matrix of size $r$ at $x_0$, $(\zeta_1,\dots,\zeta_r)$ is a smooth unitary frame for $\E$ over $B_\X(0,2R)$. In particular, this shows that the real normal trivialization $\varphi_{x_0}$ is unitary. Since the $\zeta_i$ are real, $(\zeta_1(x),\dots,\zeta_r(x))$ is an orthonormal basis of $\R \E_x$ for all $x \in M$. Hence $(\zeta_1,\dots,\zeta_r)$ is also a smooth orthogonal frame for $\R \E$ over $B_M(0,2R)$. We say that $(\zeta_1,\dots,\zeta_r)$ is a local \emph{real unitary frame}.

Similarly, let $\varphi'_{x_0}$ denote the real normal trivialization of $\L$ over $B_\X(x_0,2R)$. Then any unit vector $\zeta^0_0\in\R \L_{x_0}$ defines a local real unitary frame $\zeta_0$ for $\L$:
\begin{equation*}
\zeta_0 : x \mapsto \varphi'_{x_0}\left((\exp_{x_0})^{-1}(x),\zeta^0_0\right).
\end{equation*}
Then, for any $d \in \N$, $\varphi_{x_0}$ and $\varphi'_{x_0}$ induce a trivialization $\varphi_{x_0} \otimes (\varphi'_{x_0})^d$ of $\E \otimes \L^d$. This trivialization is the real normal trivialization of $\E \otimes \L^d$ over $B_\X(x_0,2R)$, i.e.~it is obtained by parallel transport along geodesics starting at $x_0$ for the Chern connection of $\E \otimes \L^d$. Moreover, a local real unitary frame for $\E \otimes \L^d$ is given by $(\zeta_1 \otimes \zeta_0^d,\dots,\zeta_r \otimes \zeta_0^d)$.

\subsection{Near-diagonal estimates}
\label{subsec near diagonal estimates}

We can now state the near-diagonal estimates of Ma and Marinescu for the Bergman kernel. In the sequel, we fix some $R >0$ such that $2R$ is smaller than the injectivity radius of $\X$. Let $x \in M$, we have a natural real normal chart
\begin{equation*}
\label{eq normal chart 2}
\exp_x\times \exp_x : B_{T_x\X}(0,2R)\times B_{T_x\X}(0,2R) \to B_\X(x,2R) \times B_\X(x,2R).
\end{equation*}
Moreover, the real normal trivialization of $\E \otimes \L^d$ over $B_\X(x,2R)$ (see Section~\ref{subsec real normal trivialization}) induces a trivialization
\begin{equation*}
\label{eq normal trivialization 2}
B_{T_x\X}(0,2R)\times B_{T_x\X}(0,2R) \times \End\left(\left(\E \otimes \L^d\right)_x\right) \simeq \left(\E \otimes \L^d\right) \boxtimes \left(\E \otimes\L^d\right)^*_{/B_\X(x,2R) \times B_\X(x,2R)}
\end{equation*}
that covers $\exp_x\times \exp_x$. This trivialization coincides with the real normal trivialization of $\left(\E \otimes \L^d\right) \boxtimes \left(\E \otimes\L^d\right)^*$ over $B_\X(x,2R) \times B_\X(x,2R)$.

Recall that $\dx V_\X$ denotes the Riemannian measure on $\X$. When we read this measure in the real normal chart $\exp_x$, it admits a density $\kappa : B_{T_x\X}(0,2R)\to \R_+$ with respect to the normalized Lebesgue measure of $(T_x\X,g_x)$. More precisely, we have $\kappa(z) = \sqrt{\det(g_{ij}(z))}$ where $(g_{ij}(z))$ is the matrix of $\left((\exp_x)^\star g\right)_z$, read in any real orthonormal basis of $(T_x\X,g_x)$. Since we use normal coordinates and $\X$ is compact, we have 
\begin{equation}
\label{eq estimate kappa}
\kappa(z) = 1 + O\!\left(\Norm{z}^2\right)
\end{equation}
where $\Norm{\cdot}$ is induced by $g_x$ and the estimate $O\!\left(\Norm{z}^2\right)$ does not depend on $x$.

Similarly, on the real locus $(M,g)$, $\rmes{M}$ admits a density, in the real normal chart $\exp_x$, with respect to the normalized Lebesgue measure on $(T_xM,g_x)$. This density is:
\begin{equation}
\label{eq density real}
z \longmapsto \det \left(\left(\left(\exp_x^\star g\right)_z\right)_{/T_xM}\right)^\frac{1}{2},
\end{equation}
from $B_{T_xM}(0,2R)$ to $\R_+$. As we already explained in Sect.~\ref{subsec real normal trivialization}, on the real locus, $g_\C$ is the sesquilinear extension of the restriction of $g$ to $TM$. Hence, for all $z \in B_{T_xM}(0,2R)$ we have:
\begin{equation*}
\det \left(\left(\left(\exp_x^\star g\right)_z\right)_{/T_xM}\right)^2 = \det \left(\left(\exp_x^\star g\right)_z\right),
\end{equation*}
which means that the density of $\rmes{M}$ in the chart $\exp_x$ is $\sqrt{\kappa}:B_{T_xM}(0,2R) \to \R_+$.

The following result gives the asymptotic of the Bergman kernel $E_d$ (see Sect.~\ref{subsec correlation kernel}) and its derivatives, read in the real normal trivialization about $x$ of $\left(\E \otimes \L^d\right) \boxtimes \left(\E \otimes\L^d\right)^*$. It was first established by Dai, Liu and Ma in \cite[thm.~4.18']{DLM2006}.
\begin{thm}[Dai--Liu--Ma]
\label{thm Ma Marinescu}
There exists $C'>0$ such that, for any $p \in \N$, there exists $C_p$ such that $\forall k \in \{0,\dots,p\}$, $\forall d \in \N^*$, $\forall z,w \in B_{T_x\X}(0,R)$,
\begin{multline*}
\Norm{D^k_{(z,w)}\left(E_d(z,w) - \left(\frac{d}{\pi}\right)^{\hspace{-1mm}n}\frac{\exp\left(-\frac{d}{2}\left(\Norm{z}^2+\Norm{w}^2-2 \prsc{z}{w}\right)\right)}{\sqrt{\kappa(z)}\sqrt{\kappa(w)}}\Id_{(\E\otimes\L^d)_x}\right)}\\
\leq C_p d^{n+\frac{p}{2}-1}\left(1+\sqrt{d}(\Norm{z}+\Norm{w})\right)^{2n+6+p} \exp\left(-C'\sqrt{d}\Norm{z-w}\right)+O\!\left(d^{-\infty}\right),
\end{multline*}
where:
\begin{itemize}
\item $D^k_{(z,w)}$ is the $k$-th differential at $(z,w)$ for a map $T_x\X \times T_x\X \to\End\left(\left(\E \otimes \L^d\right)_x\right)$,
\item the Hermitian inner product $\prsc{\cdot}{\cdot}$ comes from the Hermitian metric $\left(g_\C\right)_x$,
\item the norm $\Norm{\cdot}$ on $T_x\X$ is induced by $g_x$ (or equivalently $\prsc{\cdot}{\cdot}$),
\item the norm $\Norm{\cdot}$ on $\left(T^*_x\X\right)^{\otimes q}\otimes \End\left(\left(\E \otimes \L^d\right)_x\right)$ is induced by $g_x$ and $(h_d)_x$.
\end{itemize}
Moreover, the constants $C_p$ and $C'$ do not depend on $x$. The notation $O\!\left(d^{-\infty}\right)$ means that, for any $l \in \N$, this term is $O\!\left(d^{-l}\right)$ with a constant that does not depend on $x$, $z$, $w$ or $d$.
\end{thm}

\begin{proof}
This is a weak version of \cite[thm.~4.2.1]{MM2007}, with $k=1$ and $m'=0$ in the notations of~\cite{MM2007}. We used the fact that $\mathcal{F}_0$ in \cite{MM2007} is given by:
\begin{equation*}
\mathcal{F}_0(z,w) = \frac{1}{\pi^n}\exp\left(-\frac{1}{2}\left(\Norm{z}^2+\Norm{w}^2-2 \prsc{z}{w}\right)\right)\Id_{(\E\otimes\L^d)_x},
\end{equation*}
(compare (4.1.84), (4.1.85) and (4.1.92) pp.~191--192 and (5.1.18) p.~46 in \cite{MM2007}) and $\mathcal{F}_1 = 0$. See~\cite[Rem.~1.4.26]{MM2007} and~\cite{MM2013}.
\end{proof}

\begin{rem}
\label{rem normalization Ma Marinescu}
Note that our formula differs from the one in \cite{MM2007,MM2013} by a factor $\pi$ in the exponential. This comes from different normalizations of the Kähler form $\omega$.
\end{rem}

We are only interested in the behavior of $E_d$ at points of the real locus, hence we restrict our focus to points in $M$ and derivatives in real directions. Similarly, for $x,y \in M$, $E_d(x,y)$ restricts to an element of $\R\!\left(\E \otimes \L^d\right)_x \otimes \R\!\left(\E \otimes \L^d\right)^*_y$, still denoted by $E_d(x,y)$. Note that we can recover the original $E_d(x,y) : \left(\E \otimes \L^d\right)_y \to \left(\E \otimes \L^d\right)_x$ from its restriction by $\C$-linear extension.

First, we need to know the behavior of $E_d$ and its derivatives up to order $1$ in each variable in a neighborhood of the diagonal in $M \times M$.
\begin{cor}
\label{cor near diag estimates}
There exist $C$ and $C' >0$, not depending on $x$, such that $\forall k \in \{0,1,2\}$, $\forall d \in \N^*$, $\forall z,w \in B_{T_xM}(0,R)$,
\begin{multline*}
\Norm{D^k_{(z,w)}\left(E_d(z,w) - \left(\frac{d}{\pi}\right)^n\frac{\exp\left(-\frac{d}{2}\Norm{z-w}^2\right)}{\sqrt{\kappa(z)}\sqrt{\kappa(w)}}\Id_{\R(\E\otimes\L^d)_x}\right)}\\
\leq C d^{n+\frac{k}{2}-1}\left(1+\sqrt{d}(\Norm{z}+\Norm{w})\right)^{2n+8} \exp\left(-C'\sqrt{d}\Norm{z-w}\right)+O\!\left(d^{-\infty}\right),
\end{multline*}
where $D^k$ is the $k$-th differential for a map from $T_xM \times T_xM$ to $\End\left(\R\left(\E \otimes \L^d\right)_x\right)$, the norm on $T_xM$ is induced by $g_x$ and the norm on $\left(T^*_xM\right)^{\otimes q}\otimes \End\left(\left(\E \otimes \L^d\right)_x\right)$ is induced by $g_x$ and $(h_d)_x$.
\end{cor}

\begin{proof}
We apply Theorem.~\ref{thm Ma Marinescu} for $p=k \in \{0,1,2\}$ and set $C=\max(C_0,C_1,C_2)$. Then we restrict everything to the real locus.
\end{proof}

\subsection{Diagonal estimates}
\label{subsec diagonal estimates}

In this section, we deduce diagonal estimates for $E_d$ and its derivatives from Thm.~\ref{thm Ma Marinescu}. Let $x \in M$, then the usual differential for maps from $T_x\X$ to $(\E \otimes \L^d)_x$ defines a local trivial connection $\tilde{\nabla}^d$ on $(\E \otimes \L^d)_{/B_\X(0,2R)}$, via the real normal trivialization. Since this trivialization is well-behaved with respect to both the metric and the real structure (cf.~Sect.~\ref{subsec real normal trivialization}), $\tilde{\nabla}^d$ is metric and real. By a partition of unity argument, there exists a real metric connection $\nabla^d$ on $\E \otimes \L^d$ such that $\nabla^d$ agrees with $\tilde{\nabla}^d$ on $B_\X(0,R)$. In the remainder of this section, we use this connection~$\nabla^d$, and the induced connection on $(\E \otimes \L^d) \boxtimes (\E \otimes \L^d)^*$, so that the connection is trivial in the real normal trivialization about $x$.

Recall that $\partial_y^\sharp E_d$ and $\partial_x\partial_y^\sharp E_d$ are defined by~\eqref{eq def sharp} and~\eqref{eq def sharp 2} respectively.

\begin{cor}
\label{cor diag estimates}
Let $x \in M$, let $\nabla^d$ be a real metric connection that is trivial over $B_{T_x\X}(0,R)$ in the real normal trivialization about $x$. Let $\partial_y^\sharp$ and $\partial_x$ denote the associated partial derivatives for sections of $(\E \otimes \L^d) \boxtimes (\E \otimes \L^d)^*$, then we have the following estimates as $d \to +\infty$. 
\begin{align}
\label{eq value Bergman diag 00}
E_d(x,x) &= \frac{d^n}{\pi^n} \Id_{\R(\E\otimes\L^d)_x} + O\!\left(d^{n-1}\right),\\
\label{eq value Bergman diag 10}
\partial_xE_d(x,x) &= O\!\left(d^{n-\frac{1}{2}}\right),\\
\label{eq value Bergman diag 01}
\partial_y^\sharp E_d(x,x) &= O\!\left(d^{n-\frac{1}{2}}\right),\\
\label{eq value Bergman diag 11}
\partial_x\partial_y^\sharp E_d(x,x) &= \frac{d^{n+1}}{\pi^n} \Id_{\R(\E\otimes\L^d)_x} \otimes \Id_{T^*_xM} + O\!\left(d^n\right).
\end{align}
Moreover the error terms do not depend on $x$.
\end{cor}

\begin{proof}
Let $x \in M$ and let us choose an orthonormal basis of $T_xM$. We denote the corresponding coordinates on $T_xM \times T_xM$ by $(z_1,\dots,z_n,w_1,\dots,w_n)$ and by $\partial_{z_i}$ and $\partial_{w_j}$ the associated partial derivatives. Let us compute the partial derivatives of $E_d$ read in the real normal trivialization of $\left(\E \otimes \L^d\right) \boxtimes \left(\E \otimes\L^d\right)^*$ about $(x,x)$. By Cor.~\ref{cor near diag estimates}, we only need to compute the partial derivatives at $(0,0)$ of
\begin{equation}
\label{eq def xi d}
\xi_d : (z,w) \mapsto \frac{\exp\left(-\frac{d}{2}\Norm{z-w}^2\right)}{\sqrt{\kappa(z)}\sqrt{\kappa(w)}}
\end{equation}
for any $d \in \N$. For all $i$ and $j \in \{1,\dots,n\}$ and all $(z,w) \in B_{T_xM}(0,R)$ we have:
\begin{align}
\label{eq derivative xi d z}
\partial_{z_i}\xi_d(z,w) &= \left(-d(z_i-w_i)-\frac{1}{2}\frac{\partial_{z_i}\kappa(z)}{\kappa(z)}\right)\frac{\exp\left(-\frac{d}{2}\Norm{z-w}^2\right)}{\sqrt{\kappa(z)}\sqrt{\kappa(w)}},\\
\label{eq derivative xi d w}
\partial_{w_j}\xi_d(z,w) &= \left(d(z_j-w_j)-\frac{1}{2}\frac{\partial_{w_j}\kappa(w)}{\kappa(w)}\right)\frac{\exp\left(-\frac{d}{2}\Norm{z-w}^2\right)}{\sqrt{\kappa(z)}\sqrt{\kappa(w)}}
\end{align}
and
\begin{multline}
\label{eq derivative xi d zw}
\partial_{z_i}\partial_{w_j}\xi_d(z,w) = \frac{\exp\left(-\frac{d}{2}\Norm{z-w}^2\right)}{\sqrt{\kappa(z)}\sqrt{\kappa(w)}} \times\\ \left(d\delta_{ij}-d^2(z_i-w_i)(z_j-w_j)-\frac{d(z_j-w_j)}{2}\frac{\partial_{z_i}\kappa(z)}{\kappa(z)}+\frac{d(z_i-w_i)}{2}\frac{\partial_{w_j}\kappa(w)}{\kappa(w)}\right),
\end{multline}
where $\delta_{ij}$ equals $1$ if $i=j$ and $0$ otherwise. Recall that, by \eqref{eq estimate kappa}, $\kappa(0)=1$ and the partial derivatives of $\kappa$ vanish at the origin. Then evaluating the above expressions at $(0,0)$ gives:
\begin{align*}
\xi_d(0,0) &= 1, & \partial_{z_i}\xi_d(0,0)&=0=\partial_{w_j}\xi_d(0,0) & &\text{and} & \partial_{z_i}\partial_{w_j}\xi_d(0,0)&=\delta_{ij}d.
\end{align*}
By Cor.~\ref{cor near diag estimates}, we have the following estimates for the partial derivatives of $E_d$ read in the real normal trivialization about $x$: for all $i,j \in \{1,\dots,n\}$,
\begin{equation}
\label{eq diag values in chart}
\begin{aligned}
E_d(0,0) &= \frac{d^n}{\pi^n} \Id_{\R(\E\otimes\L^d)_x} + O\!\left(d^{n-1}\right), & \partial_{w_j}E_d(0,0) &= O\!\left(d^{n-\frac{1}{2}}\right),\\
\partial_{z_i}\partial_{w_j}E_d(0,0) &= \delta_{ij}\frac{d^{n+1}}{\pi^n} \Id_{\R(\E\otimes\L^d)_x} + O\!\left(d^{n}\right), & \partial_{z_i}E_d(0,0) &= O\!\left(d^{n-\frac{1}{2}}\right).
\end{aligned}
\end{equation}
Moreover these estimates are uniform in $x \in M$. Equations~\eqref{eq value Bergman diag 00}, \eqref{eq value Bergman diag 10}, \eqref{eq value Bergman diag 01} and \eqref{eq value Bergman diag 11} are coordinate-free versions of these statements.
\end{proof}

\subsection{Far off-diagonal estimates}
\label{subsec far off diagonal estimates}

Finally, we will use the fact that the Bergman kernel and its derivatives decrease fast enough outside of the diagonal. In this section we recall the far off-diagonal estimates of \cite[thm.~5]{MM2015}, see also \cite[prop.~4.1.5]{MM2007}.

Let $d \in \N$ and let $S$ be a smooth section of $\R\left(\E \otimes \L^d\right) \boxtimes \R\left(\E \otimes \L^d\right)^*$. Let $x,y \in M$, we denote by $\Norm{S(x,y)}_{\mathcal{C}^k}$ the maximum of the norms of $S$ and its derivatives of order at most $k$ at the point $(x,y)$. The derivatives of $S$ are computed with respect to the connection induced by the Chern connection of $\E\otimes\L^d$ and the Levi--Civita connection on $M$. The norms of the derivatives are the ones induced by $h_d$ and $g$.

\begin{thm}[Ma--Marinescu]
\label{thm off diag estimates}
There exist $C'>0$ and $d_0 \in \N^*$ such that, for all $k \in \N$, there exists $C_k>0$ such that $\forall d \geq d_0$, $\forall x,y \in M$
\begin{equation*}
\Norm{E_d(x,y)}_{\mathcal{C}^k} \leq C_k d^{n+\frac{k}{2}} \exp \left(-C' \sqrt{d}\, \rho_g(x,y)\right),
\end{equation*}
where $\rho_g(\cdot,\cdot)$ denotes the geodesic distance in $(M,g)$.
\end{thm}

\begin{proof}
This is the first part of \cite[thm.~5]{MM2015}, where we only considered the restriction of $E_d$ and its derivatives to $M$. Note that the Levi--Civita connection on $M$ is the restriction of the Levi--Civita connection on $\X$. Hence the norm $\Norm{\cdot}_{\mathcal{C}^k}$, such as we defined it, is smaller than the one used in \cite{MM2015}.
\end{proof}

\section{Proof of Theorem~\ref{thm asymptotics variance}}
\label{sec proof of the main theorem}

In this section, we prove Theorem~\ref{thm asymptotics variance}. Recall that $\X$ is a compact Kähler manifold of dimension $n\geq 2$ defined over the reals and that $M$ denotes its real locus, assumed to be non-empty. Let $\E\to \X$ be a rank $r \in\{1,\dots,n-1\}$ real Hermitian vector bundle and $\L\to \X$ be a real Hermitian line bundle whose curvature form is $\omega$, the Kähler form of~$\X$. We assume that $\E$ and $\L$ are endowed with compatible real structures. For all $d \in \N$, $E_d$ denotes the Bergman kernel of $\E \otimes \L^d$. Finally, $s_d$ denotes a standard Gaussian vector in $\R \H$, whose real zero set is denoted by $Z_d$, and $\rmes{d}$ is the measure of integration over $Z_d$.

\subsection{The Kac--Rice formula}
\label{subsec Kac--Rice formula}

The first step in our proof of Thm.~\ref{thm asymptotics variance} is to prove a version of the Kac--Rice formula adapted to our problem. This is the goal of this section. First, we recall the Kac--Rice formula we used in \cite{Let2016} to compute the expectation of $\vol{Z_d}$ (Thm.~\ref{thm Kac-Rice exp}). Then we prove a Kac--Rice formula adapted to the computation of the covariance (Thm.~\ref{thm Kac-Rice var}), compare \cite[thm.~6.3]{AW2009} and \cite[chap.~11.5]{TA2007}.

Let $L : V \to V'$ be a linear map between two Euclidean spaces, recall that we denote by $\odet{L}$ its Jacobian (cf.~Def.~\ref{def Jacobian}). Since $LL^*$ is a semi-positive symmetric endomorphism of $V'$, $\det(LL^*) \geq 0$ and $\odet{L}$ is well-defined. The range of $L^*$ is $\ker(L)^\perp$, hence $\ker(LL^*)=\ker(L^*)=L(V)^\perp$. Thus $\odet{L} >0$ if and only if $LL^*$ is injective, that is if and only if $L$ is surjective. In fact, if $L$ is surjective, let $A$ be the matrix of the restriction of $L$ to $\ker(L)^\perp$ in any orthonormal basis of $\ker(L)^\perp$ and $V'$, then we have:
\begin{equation*}
\odet{L}=\sqrt{\det\left(A A^{\text{t}}\right)}=\norm{\det(A)}.
\end{equation*}

\begin{thm}[Kac--Rice formula]
\label{thm Kac-Rice exp}
Let $d \geq d_1$, where $d_1$ is defined by Lem.~\ref{lem dfn d1} and let $\nabla^d$ be any real connection on $\E \otimes \L^d$. Let $s_d$ be a standard Gaussian vector in $\R \H$. Then for any Borel measurable function $\phi : M \to \R$ we have:
\begin{equation}
\label{eq thm Kac-Rice exp}
\esp{\int_{x \in Z_d} \phi(x) \rmes{d}}= (2\pi)^{-\frac{r}{2}} \int_{x \in M} \frac{\phi(x)}{\odet{\ev_x^d}}\espcond{\odet{\nabla^d_{x}s_d}}{s_d(x)=0} \rmes{M}
\end{equation}
whenever one of these integrals is well-defined.
\end{thm}
The expectation on the right-hand side of~\eqref{eq thm Kac-Rice exp} is to be understood as the conditional expectation of $\odet{\nabla^d_{x}s_d}$ given that $s_d(x)=0$. This result is a consequence of~\cite[thm.~5.3]{Let2016}. See also Section~5.3 of \cite{Let2016}, where we applied this result with $\phi = \mathbf{1}$, in order to compute the expected volume of $Z_d$.

Let us denote by $\Delta = \{(x,y) \in M^2 \mid x = y\}$ the diagonal in $M^2$. Let $d \in \N$ and let $(x,y) \in M^2 \setminus \Delta$ we denote by $\ev^d_{x,y}$ the evaluation map:
\begin{equation}
\label{eq dfn ev map}
\begin{array}{rccc}
\ev_{x,y}^d:& \R \H & \longrightarrow & \R\!\left(\E \otimes \L^d\right)_x \oplus \R\!\left(\E \otimes \L^d\right)_y.\\
& s & \longmapsto & \left(s(x),s(y)\right)
\end{array}
\end{equation}
The following proposition is the equivalent of Lemma~\ref{lem dfn d1} for two points $(x,y) \notin \Delta$. One could prove this result using only the estimates of Section~\ref{sec estimates for the bergman kernel}. We give instead a less technical proof, using the Kodaira embedding theorem. See \cite[sect.~5.1]{MM2007} for a discussion of the relations between these approaches.

\begin{prop}
\label{prop amplitude 2 points}
There exists $d_2 \in \N$, depending only on $\X$, $\E$ and $\L$, such that for every $d \geq d_2$ and every $(x,y) \in M^2 \setminus \Delta$, the evaluation map $\ev^d_{x,y}$ is surjective.
\end{prop}

\begin{proof}
Recall that there exists $d_1 \in \N$ such that, for all $d\geq d_1$, the map $\ev_x^d$ defined by~\eqref{eq dfn 0 ample} is surjective for any $x \in M$ (see Lem.~\ref{lem dfn d1}). Then, for all $d \geq d_1$ and all $x \in M$, the complexified map $\tilde{\ev}_x^d:\H \to \left(\E \otimes \L^d\right)_x$ defined by $\tilde{\ev}_x^d(s)=s(x)$ is also surjective.

For any $l \in \N$, we denote by $\Psi_l: \X \to \P\left(H^0(\X,\L^l)^*\right)$ the Kodaira map, defined by $\Psi_l(x) = \left\{s \in H^0(\X,\L^l) \mvert s(x) = 0\right\}$. By the Kodaira embedding theorem (see \cite[chap.~1.4]{GH1994}), there exists $l_0 \in \N$ such that $\Psi_{l_0}$ is well-defined and is an embedding.

We set $d_2 = l_0+d_1$. Let $d \geq d_2$ and let $(x,y) \in M^2 \setminus \Delta$. Since $\Psi_{l_0}(x)$ and $\Psi_{l_0}(y)$ are distinct hyperplanes in $H^0(\X,\L^{l_0})$, there exist $\sigma_x$ and $\sigma_y \in H^0(\X,\L^{l_0})$ such that:
\begin{equation*}
\label{eq Kodaira}
\left\{\begin{aligned}
\sigma_x(x) &\neq 0,\\
\sigma_x(y) & = 0
\end{aligned}\right.
\qquad \text{and} \qquad
\left\{\begin{aligned}
\sigma_y(x) & = 0,\\
\sigma_y(y) & \neq 0.
\end{aligned}\right.
\end{equation*}
Since $d-l_0 \geq d_1$, $\tilde{\ev}_x^d$ is onto and there exist $\sigma_{1,x},\dots,\sigma_{r,x} \in H^0(\X,\E \otimes \L^{d-l_0})$ such that $\left(\sigma_{1,x}(x),\dots,\sigma_{r,x}(x)\right)$ is a basis of $\left(\E \otimes \L^{d-l_0}\right)_x$. Similarly there exist $\sigma_{1,y},\dots,\sigma_{r,y}$ such that $\left(\sigma_{1,y}(y),\dots,\sigma_{r,y}(y)\right)$ is a basis of $\left(\E \otimes \L^{d-l_0}\right)_y$. We define global holomorphic sections of $\E \otimes \L^d$ by $s_{k,x} = \sigma_{k,x}\otimes \sigma_x$ and $s_{k,y} = \sigma_{k,y}\otimes \sigma_y$ for all $k \in \{1,\dots,r\}$. These sections are such that $(s_{k,x}(x))_{1\leq k \leq r}$ is a basis of $\left(\E \otimes \L^d\right)_x$, $(s_{k,y}(y))_{1\leq k \leq r}$ is a basis of $\left(\E \otimes \L^d\right)_y$ and for all $k \in \{1,\dots,r\}$, $s_{k,x}(y) = 0 = s_{k,y}(x)$. This proves that the map
\begin{equation*}
\label{eq dfn ev map complex}
\begin{array}{rccc}
\tilde{\ev}_{x,y}^d:& \H & \longrightarrow & \left(\E \otimes \L^d\right)_x \oplus \left(\E \otimes \L^d\right)_y.\\
& s & \longmapsto & \left(s(x),s(y)\right)
\end{array}
\end{equation*}
has rank at least $2r$ (as a $\C$-linear map). Since $\tilde{\ev}_{x,y}^d$ is the complexified map of $\ev_{x,y}^d$, the latter must have rank at least $2r$ (as a $\R$-linear map), hence it is onto.
\end{proof}

\begin{rem}
\label{rem non degeneracy and surjectivity}
In Prop.~\ref{prop amplitude 2 points}, $\ev_{x,y}^d$ is surjective if and only if $\odet{\ev_{x,y}^d}>0$, that is if and only if $\ev_{x,y}^d \left(\ev_{x,y}^d\right)^*$ is non-singular. Since the latter is the variance operator of $\ev_{x,y}^d(s_d)$, where $s_d \sim \mathcal{N}(\Id)$ in $\R \H$, we see that the surjectivity of $\ev_{x,y}^d$ is equivalent to the non-degeneracy of the distribution of $(s_d(x),s_d(y))$.
\end{rem}

We can now deduce a Kac--Rice type formula from Prop.~\ref{prop amplitude 2 points}. For any $d \in \N$, we define $F_d$ to be the following bundle map over $M^2$:
\begin{equation*}
\label{eq def bundle map Fd}
\begin{array}{rccc}
F_d:& \R \H \times M^2 & \longrightarrow & \displaystyle \R\!\left(\E \otimes \L^d\right)\times \R\!\left(\E \otimes \L^d\right).\\
& (s,x,y) & \longmapsto & \left(s(x),s(y)\right)
\end{array}
\end{equation*}
Let $\nabla^d$ be any real connection on $\E \otimes \L^d\to \X$ (see Def.~\ref{dfn real connection}). Then by Rem.~\ref{rem real connection in a real direction}, the restriction of $\nabla^d$ defines a connection on $\R (\E \otimes \L^d) \to M$. Let $\nabla^d F_d$ denote the vertical component of the differential of $F_d$. Then, for all $(s_0,x,y) \in \R\H \times M^2$, we have:
\begin{equation*}
\label{eq differential Fdp}
\begin{array}{rccc}
\nabla^d_{(s_0,x,y)}F_d:& \R \H \times T_xM \times T_yM & \longrightarrow & \R\!\left(\E \otimes \L^d\right)_x \oplus \R\!\left(\E \otimes \L^d\right)_y.\\
& (s,v,w) & \longmapsto & \left(s(x)+\nabla^d_xs_0\cdot v,s(y)+\nabla^d_ys_0\cdot w\right)
\end{array}
\end{equation*}
We denote by $\partial_1^dF_d$ the partial derivative of $F_d$ with respect to the first variable (meaning~$s$), and by $\partial_2^dF_d$ its partial derivative with respect to the second variable (meaning $(x,y)$). Then for all $(s_0,x,y) \in \R\H \times M^2$ we have:
\begin{align}
\label{eq partial derivatives Fdp}
\partial_1^dF_d(s_0,x,y) &= \ev_{x,y}^d & &\text{and} & \partial_2^dF_d(s_0,x,y): (v,w) \mapsto \left(\nabla_x^ds_0 \cdot v,\nabla_y^ds_0 \cdot w\right).
\end{align}
From now on, we assume that $d\geq d_2$, where $d_2$ is given by Prop.~\ref{prop amplitude 2 points}. We define an incidence manifold $\Sigma_d$ by:
\begin{equation*}
\label{eq def incidence manifold}
\Sigma_d = \left(F_d\right)^{-1}(0) \cap \left(\R \H \times \left(M^2 \setminus \Delta\right)\right).
\end{equation*}
By Prop.~\ref{prop amplitude 2 points} and eq.~\ref{eq partial derivatives Fdp}, for all $(s,x,y)\in \R \H \times \left(M^2 \setminus \Delta\right)$, $\partial_1^dF_{d,p}(s,x,y)$ is surjective. Thus, the restriction of $F_d$ to $\R \H \times \left(M^2 \setminus \Delta\right)$ is a submersion and $\Sigma_d$ is a submanifold of $\R\H \times M^2$ of codimension $2r$. Note that we are only concerned with the zero set of $F_d$, hence none of this depends on the choice of $\nabla^d$. We can now state the Kac--Rice formula in this context.

\begin{thm}[Kac--Rice formula]
\label{thm Kac-Rice var}
Let $d \geq d_2$, where $d_2$ is given by Prop.~\ref{prop amplitude 2 points}, and let $\nabla^d$ be any real connection on $\E \otimes \L^d$. Let $s_d$ be a standard Gaussian vector in $\R \H$. Then for any Borel measurable function $\Phi : \Sigma_d \to \R$ we have:
\begin{multline}
\label{eq thm Kac-Rice var}
\esp{\int_{(x,y) \in (Z_d)^2 \setminus \Delta} \Phi(s_d,x,y) \rmes{d}^2}= \frac{1}{(2\pi)^r} \int_{(x,y) \in M^2 \setminus \Delta} \frac{1}{\odet{\ev_{x,y}^d}}\times\\
\espcond{\Phi(s_d,x,y) \odet{\nabla^d_{x}s_d}\odet{\nabla^d_{y}s_d}}{s_d(x)=0=s_d(y)} \rmes{M}^2
\end{multline}
whenever one of these integrals is well-defined. Here, $\rmes{M}^2$ stands for the product measure on $M^2$ induced by $\rmes{M}$. Similarly, $\rmes{d}^2$ is the product measure on $(Z_d)^2$.
\end{thm}

The expectation on the right-hand side of~\eqref{eq thm Kac-Rice var} is to be understood as the conditional expectation of $\Phi(s_d,x,y) \odet{\nabla^d_{x}s_d}\odet{\nabla^d_{y}s_d}$ given that $s_d(x)=0=s_d(y)$.

\begin{proof}
The proof of Thm.~\ref{thm Kac-Rice var} uses the double fibration trick, that is apply Federer's coarea formula twice. See for example \cite[App.~C]{Let2016} and the reference therein.

The Euclidean inner product on $\R \H$ defined by eq.~\eqref{eq definition inner product} and the Riemannian metric $g$ induce a Riemannian metric on $\R \H \times M^2$, and on $\Sigma_d$ by restriction. Let $\pi_1:\Sigma_d\to \R \H$ and $\pi_2:\Sigma_d\to M^2 \setminus \Delta$ denote the projections from $\Sigma_d$ to the first and second factors, respectively. For all $s \in \R \H$, $\pi_1^{-1}(s)$ is isometric to $Z_s$ and we identify these spaces. Similarly, for all $(x,y) \in M^2 \setminus \Delta$ we identify $\pi_2^{-1}(x,y)$ with the isometric space $\ker(\ev_{x,y}^d)$.

We denote by $\dx s$ the Lebesgue measure on $\R \H$ or any of its subspaces, normalized so that a unit cube has volume $1$. Let $\Phi:\Sigma_d \to \R$ be a Borel measurable function. Then
\begin{equation*}
\esp{\int_{(Z_d)^2 \setminus \Delta} \Phi \rmes{d}^2}= \int_{s \in \R \H} \left(\int_{(x,y) \in \pi_1^{-1}(s)} \Phi(s,x,y)\frac{e^{-\frac{1}{2}\Norm{s}^2}}{(2\pi)^{\frac{N_d}{2}}} \rmes{d}^2\right) \dx s,
\end{equation*}
where $N_d$ is the dimension of $\R \H$. Then, by the double fibration trick \cite[Prop.~C.3]{Let2016} this quantity equals:
\begin{equation}
\label{eq double fibration}
\int_{(x,y)\in M^2 \setminus \Delta}\left( \int_{s \in \ker(\ev_{x,y}^d)} \Phi(s,x,y)\frac{e^{-\frac{1}{2}\Norm{s}^2}}{(2\pi)^{\frac{N_d}{2}}} \frac{\odet{\partial_2^dF_d(s,x,y)}}{\odet{\partial_1^dF_d(s,x,y)}} \dx s\right) \rmes{M}^2.
\end{equation}
Then eq.~\eqref{eq partial derivatives Fdp} shows that $\partial_2^dF_{d,p}(s,x,y)= \nabla^d_xs \oplus \nabla^d_ys$. Moreover, by definition of the metrics, $T_xM$ is orthogonal to $T_yM$ and $\R(\E\otimes\L^d)_x$ is orthogonal to $\R(\E\otimes\L^d)_y$. Thus
\begin{align*}
\odet{\partial_2^dF_d(s,x,y)} &= \det\left(\partial_2^dF_d(s,x,y)\left(\partial_2^dF_d(s,x,y)\right)^*\right)^\frac{1}{2}\\
&= \det \left(\begin{pmatrix}
\nabla_x^ds & 0 \\ 0 & \nabla^d_ys
\end{pmatrix}
\begin{pmatrix}
(\nabla^d_xs)^* & 0 \\ 0 & (\nabla^d_ys)^*
\end{pmatrix}\right)^\frac{1}{2}\\
&= \det\begin{pmatrix}
\nabla_x^ds (\nabla_x^ds)^* & 0 \\ 0 & \nabla_y^ds (\nabla_y^ds)^*
\end{pmatrix}^\frac{1}{2}\\
&= \odet{\nabla^d_xs}\odet{\nabla^d_ys}.
\end{align*}
Besides, eq.\eqref{eq partial derivatives Fdp} also shows that $\odet{\partial_1^dF_d(s,x,y)}= \odet{\ev_{x,y}^d}$, which does not depend on $s$, so that \eqref{eq double fibration} equals:
\begin{equation*}
\int_{(x,y) \in M^2 \setminus \Delta} \frac{1}{\odet{\ev_{x,y}^d}}\left(\int_{s \in \ker\left(\ev_{x,y}^d\right)}\Phi \odet{\nabla^d_xs}\odet{\nabla^d_ys}\frac{e^{-\frac{1}{2}\Norm{s}^2}}{(2\pi)^{\frac{N_d}{2}}} \dx s \right) \rmes{M}^2.
\end{equation*}
Finally, by Prop.~\ref{prop amplitude 2 points}, $\ker(\ev_x^d)$ is a subspace of codimension $2r$ of $\R \H$. Hence, the inner integral in~\eqref{eq double fibration} can be expressed as a conditional expectation given that $\ev_{x,y}^d(s_d)=0$, up to a factor $(2\pi)^r$. This concludes the proof of Thm.~\ref{thm Kac-Rice var}.
\end{proof}

\subsection{An integral formula for the variance}
\label{subsec an integral formula for the variance}

In this section, we fix some $d \geq \max(d_0,d_1,d_2)$ where $d_0$, $d_1$ and $d_2$ are defined by Thm.~\ref{thm off diag estimates}, Lem.~\ref{lem dfn d1} and Prop.~\ref{prop amplitude 2 points} respectively. We denote by $\nabla^d$ a real connection on $\E \otimes \L^d$. Let $\phi_1, \phi_2 \in \mathcal{C}^0(M)$, we want to compute:
\begin{equation}
\label{eq variance}
\begin{aligned}
\var{\rmes{d}}\left(\phi_1,\phi_2\right) &= \cov{\prsc{\rmes{d}}{\phi_1}}{\prsc{\rmes{d}}{\phi_2}}\\
&= \esp{\prsc{\rmes{d}}{\phi_1}\prsc{\rmes{d}}{\phi_2}} - \esp{\prsc{\rmes{d}}{\phi_1}}\esp{\prsc{\rmes{d}}{\phi_2}}.
\end{aligned}
\end{equation}

First, by Thm.~\ref{thm Kac-Rice exp}, we have:
\begin{multline}
\label{eq exp squared}
\esp{\prsc{\rmes{d}}{\phi_1}}\esp{\prsc{\rmes{d}}{\phi_2}}=\frac{1}{(2\pi)^r} \times\\
\int_{M^2}\phi_1(x)\phi_2(y)\frac{\espcond{\odet{\nabla^d_{x}s_d}}{s_d(x)=0}}{\odet{\ev_x^d}}\frac{\espcond{\odet{\nabla^d_{y}s_d}}{s_d(y)=0}}{\odet{\ev_y^d}}\rmes{M}^2.
\end{multline}
On the other hand,
\begin{align*}
\esp{\prsc{\rmes{d}}{\phi_1}\prsc{\rmes{d}}{\phi_2}} &=\esp{\left(\int_{x \in Z_d}\phi_1(x) \rmes{d}\right)\left(\int_{y \in Z_d}\phi_2(y) \rmes{d}\right)}\\
&= \esp{\int_{(x,y) \in (Z_d)^2\setminus \Delta}\phi_1(x)\phi_2(y) \rmes{d}^2}.
\end{align*}
Indeed, $Z_d$ is almost surely of dimension $n-r > 0$, so that $(Z_d)^2 \cap \Delta$ (that is the diagonal in $(Z_d)^2$) has measure~$0$ for $\rmes{d}^2$. We compute this integral by Thm.~\ref{thm Kac-Rice var}:
\begin{multline}
\label{eq case n>r}
\esp{\int_{(x,y) \in (Z_d)^2\setminus \Delta}\phi_1(x)\phi_2(y) \rmes{d}^2} = \frac{1}{(2\pi)^r}\int_{(x,y) \in M^2 \setminus \Delta} \frac{\phi_1(x)\phi_2(y)}{\odet{\ev_{x,y}^d}}\times\\
\espcond{\odet{\nabla^d_{x}s_d}\odet{\nabla^d_{y}s_d}}{s_d(x)=0=s_d(y)} \rmes{M}^2.
\end{multline}
Let $\D_d$ be the function defined by: $\forall (x,y)\in M^2 \setminus \Delta$,
\begin{multline}
\label{eq def density}
\D_d(x,y) = \left(\frac{\espcond{\odet{\nabla^d_{x}s_d}\!\odet{\nabla^d_{y}s_d}}{\ev_{x,y}^d(s_d)=0}}{\odet{\ev_{x,y}^d}}\right.\\
\left.-\frac{\espcond{\odet{\nabla^d_{x}s_d}}{s_d(x)=0}\espcond{\odet{\nabla^d_{y}s_d}}{s_d(y)=0}}{\odet{\ev_x^d}\odet{\ev_y^d}}\right)
\end{multline}
Since $\dim M =n >0$, $\Delta$ has measure $0$ in $M^2$ for $\rmes{M}^2$. Thus, by~\eqref{eq variance}, \eqref{eq exp squared}, \eqref{eq case n>r} and~\eqref{eq def density}, we have:
\begin{equation}
\label{eq variance2}
\var{\rmes{d}}\left(\phi_1,\phi_2\right) = \frac{1}{(2\pi)^r} \int_{M^2}\phi_1(x)\phi_2(y) \D_d(x,y) \rmes{M}^2.
\end{equation}

\begin{rem}
\label{rem choice of nablad}
At this stage, it is worth noticing that the values of the conditional expectations appearing in the definition of $\D_d$ (see eq.~\eqref{eq def density}) do not depend on the choice of $\nabla^d$. In fact, the whole conditional distribution of $\nabla^d_xs_d$ given that $s_d(x)=0$ (resp.~of $\nabla^d_ys_d$ given that $s_d(y)=0$, resp.~of $(\nabla^d_xs_d,\nabla^d_ys_d)$ given that $s_d(x)=0=s_d(y)$) is independent of the choice of $\nabla^d$. Indeed, if $s_d(x)=0$ then $\nabla^d_x s_d$ does not depend on $\nabla^d$, and we conditioned on the vanishing of $s_d(x)$ (resp.~$s_d(y)$, resp.~$s_d(x)$ and $s_d(y)$). Thus, in the sequel, we can use any real connection we like, even one that depends on $(x,y) \in M^2 \setminus \Delta$.
\end{rem}

\subsection{Asymptotic for the variance}
\label{subsec asymptotic for the variance}

In this section we compute the asymptotic of the integral in eq.~\eqref{eq variance2}. The main point is to write $M^2$ as the disjoint union of a neighborhood of $\Delta$, of size about $\frac{\ln d}{\sqrt{d}}$, and its complement. In~\eqref{eq variance2}, the set of points that are far from the diagonal will contribute a term of smaller order than the neighborhood of $\Delta$. This is a consequence of the fast decrease of the Bergman kernel outside of the diagonal. The values of $s_d$ at $x$ and $y$ are not correlated, up to some small error, outside of a neighborhood of $\Delta$.

We still assume that $d \geq \max(d_0,d_1,d_2)$ and we denote by $s_d$ a standard Gaussian vector in $\R \H$.

\subsubsection{Asymptotics for the uncorrelated terms}
\label{subsubsec uncorrelated terms}

Let us first compute asymptotics for the terms in the expression of $\D_d$ (see eq.~\eqref{eq def density}) that only depend on one point, say $x \in M$. For all $x \in M$, $\ev_x^d$ is linear. Hence $s_d(x)=\ev_x^d(s_d)$ is a centered Gaussian vector in $\R \left(\E \otimes \L^d\right)_x$ with variance operator:
\begin{equation}
\label{eq variance evx}
\ev_x^d \left(\ev_x^d\right)^* = \esp{s_d(x) \otimes \left(s_d(x)\right)^*} = E_d(x,x),
\end{equation}
where $E_d$ is the Bergman kernel of $\E\otimes \L^d$ and the last equality is given by Prop.~\ref{prop covariance equals Bergman}.

\begin{lem}
\label{lem estimates odet evx}
For every $x \in M$, we have:
\begin{equation*}
\left(\frac{\pi}{d}\right)^\frac{nr}{2}\odet{\ev_x^d} = 1 + O\!\left(d^{-1}\right),
\end{equation*}
where the error term $O\!\left(d^{-1}\right)$ does not depend on $x$.
\end{lem}

\begin{proof}
Let $x \in M$, then $\odet{\ev_x^d}^2 = \det E_d(x,x)$ by~\eqref{eq variance evx}. By~\eqref{eq value Bergman diag 00},  we have:
\begin{equation*}
\left(\frac{\pi}{d}\right)^{nr}\odet{\ev_x^d}^2 = \det\left(\Id_{\R\left(\E\otimes \L^d\right)_x}+O\!\left(d^{-1}\right)\right) = 1 + O\!\left(d^{-1}\right).
\end{equation*}
The error term in~\eqref{eq value Bergman diag 00} is independent of $x$, therefore the same is true here.
\end{proof}

Let $\nabla^d$ be a real connection on $\E \otimes \L^d$. We assume that $\nabla^d$ is a metric connection, so that Lem.~\ref{lem dual connection} and Cor.~\ref{cor variance of the 1-jet} are valid in this context. Recall that the Chern connection is an example of real metric connection.

For all $x \in M$, let $j_x^d : s \mapsto \left(s(x),\nabla_x^ds\right)$ denote the evaluation of the $1$-jet at $x$, from $\R \H$ to $\R\!\left(\E \otimes \L^d\right)_x \otimes \left(\R \oplus T^*_xM\right)$. Since $j_x^d$ is linear, $\left(s_d(x),\nabla_x^ds_d\right)$ is a centered Gaussian vector with variance operator $j_x^d \left(j_x^d\right)^*$. This operator splits according to the direct sum $\R\!\left(\E \otimes \L^d\right)_x\oplus \R\!\left(\E \otimes \L^d\right)_x \otimes T^*_xM$:
\begin{equation}
\label{eq variance jx}
\begin{aligned}
j_x^d \left(j_x^d\right)^* &= \esp{j_x^d(s_d) \otimes \left(j_x^d(s_d)\right)^*}\\
&= \begin{pmatrix}
\esp{s_d(x)\otimes s_d(x)^*} & \esp{s_d(x)\otimes (\nabla^d_xs_d)^*}\\
\esp{\nabla^d_xs_d\otimes s_d(x)^*} & \esp{\nabla^d_xs_d\otimes (\nabla^d_xs_d)^*}
\end{pmatrix}\\
&= \begin{pmatrix}
E_d(x,x) & \partial_y^\sharp E_d(x,x)\\
\partial_xE_d(x,x) & \partial_x\partial_y^\sharp E_d(x,x)
\end{pmatrix},
\end{aligned}
\end{equation}
where the last equality comes from Cor.~\ref{cor variance of the 1-jet sharp}. We chose $d \geq d_1$, so that $\ev_x^d$ is surjective (see Lem.~\ref{lem dfn d1}), i.e.~$\det\left(\ev_x^d\left(\ev_x^d\right)^*\right)>0$. Hence, the distribution of $s_d(x)$ is non-degenerate. Then (see~\cite[prop.~1.2]{AW2009}), the distribution of $\nabla^d_xs_d$ given that $s_d(x)=0$ is a centered Gaussian in $\R\!\left(\E \otimes \L^d\right)_x \otimes T^*_xM$ with variance operator:
\begin{equation}
\label{eq conditional exp}
\partial_x\partial_y^\sharp E_d(x,x)- \partial_xE_d(x,x) \left(E_d(x,x)\right)^{-1}\partial_y^\sharp E_d(x,x).
\end{equation}

\begin{lem}
\label{lem estimates cond exp x}
For every $x \in M$, we have:
\begin{equation*}
\left(\frac{\pi^n}{d^{n+1}}\right)^\frac{r}{2} \espcond{\odet{\nabla^d_xs_d}}{s_d(x)=0} = (2\pi)^\frac{r}{2} \frac{\vol{\S^{n-r}}}{\vol{\S^n}} \left(1 + O\!\left(d^{-1}\right)\right),
\end{equation*}
where the error term is independent of $x$.
\end{lem}

\begin{proof}
Let $x \in M$, and let $L_d(x)$ be a centered Gaussian vector in $\R\!\left(\E \otimes \L^d\right)_x \otimes T^*_xM$ with variance operator:
\begin{equation}
\label{eq def Lambda d x}
\Lambda_d(x) = \frac{\pi^n}{d^{n+1}}\left(\partial_x\partial_y^\sharp E_d(x,x)- \partial_xE_d(x,x) \left(E_d(x,x)\right)^{-1}\partial_y^\sharp E_d(x,x)\right).
\end{equation}
By~\eqref{eq conditional exp} and the above discussion, the distribution of $\nabla^d_xs_d$ given that $s_d(x)=0$ equals that of $\left(\frac{d^{n+1}}{\pi^n}\right)^\frac{1}{2}L_d(x)$. Then,
\begin{equation}
\label{eq cond exp equals}
\espcond{\odet{\nabla^d_xs_d}}{s_d(x)=0} = \esp{\odet{\hspace{-1mm}\left(\frac{d^{n+1}}{\pi^n}\right)^{\hspace{-1mm}\frac{1}{2}} \hspace{-1mm} L_d(x)}}=\hspace{-1mm} \left(\frac{d^{n+1}}{\pi^n}\right)^{\hspace{-1mm}\frac{r}{2}} \esp{\odet{L_d(x)}}.
\end{equation}

Recall that the distribution of $\nabla^d_xs_d$ given that $s_d(x)=0$ does not depend on the choice of $\nabla^d$ (Rem.~\ref{rem choice of nablad}). Hence $\Lambda_d(x)$ does not depend on the choice of $\nabla^d$. For the following computation, we choose $\nabla^d$ to be trivial over $B_{T_xM}(0,R)$ in the real normal trivialization about $x$. Then we can use the diagonal estimates of Cor.~\ref{cor diag estimates} for the Bergman kernel and its derivatives. We have: $\Lambda_d(x) = \Id_{\R(\E\otimes\L^d)_x} \otimes \Id_{T^*_xM} + O\!\left(d^{-1}\right)$, where the error does not depend on $x$. Hence, 
\begin{equation}
\label{eq det Lambda dx}
\det\left(\Lambda_d(x)\right) = 1 + O\!\left(d^{-1}\right).
\end{equation}
Besides, there exists some $K>0$ such that $\Norm{\Lambda_d(x)^{-1}-\Id} \leq Kd^{-1}$ for all $d$ large enough. Then, by the mean value inequality, for all $L \in \R(\E\otimes\L^d)_x \otimes T^*_xM$
\begin{equation*}
\norm{\exp\left(-\frac{1}{2}\prsc{\left(\Lambda_d(x)^{-1}- \Id\right) L}{L}\right)-1} \leq \frac{K}{2d}\Norm{L}^2\exp\left(\frac{K}{2d}\Norm{L}^2\right).
\end{equation*}

Let $L^0_d(x) \sim \mathcal{N}(\Id)$ in $\R(\E\otimes\L^d)_x \otimes T^*_xM$ and let $\dx L$ denote the normalized Lebesgue measure on $\R(\E\otimes\L^d)_x \otimes T^*_xM$. Then we have:
\begin{multline*}
(2\pi)^{\frac{nr}{2}}\norm{\det\left(\Lambda_d(x)\right)^\frac{1}{2}\esp{\odet{L_d(x)}}-\esp{\odet{L^0_d(x)}}}\\
\begin{aligned}
&\leq\int \odet{L} e^{-\frac{1}{2}\Norm{L}^2}\norm{\exp\left(-\frac{1}{2}\prsc{\left(\Lambda_d(x)^{-1}- \Id\right) L}{L}\right)-1} \dx L\\
&\leq \frac{K}{2d} \int \odet{L} \exp\left(-\frac{1}{2}\left(1-\frac{K}{d}\right)\Norm{L}^2\right) \dx L.
\end{aligned}
\end{multline*}
The integral on the last line converges to some finite limit as $d \to +\infty$. Thus, by~\eqref{eq det Lambda dx},
\begin{equation}
\label{eq last lemma}
\begin{aligned}
\esp{\odet{L_d(x)}}&= \det\left(\Lambda_d(x)\right)^{-\frac{1}{2}} \left(\esp{\odet{L^0_d(x)}}+ O\!\left(d^{-1}\right)\right)\\
&=\esp{\odet{L^0_d(x)}}+ O\!\left(d^{-1}\right),
\end{aligned}
\end{equation}
uniformly in $x \in M$. Lemma~\ref{lem estimates cond exp x} follows from~\eqref{eq  cond exp equals}, \eqref{eq last lemma} and the following equality, that was proved in~\cite[lem.~A.14]{Let2016}:
\begin{equation}
\label{eq value exp odet lim}
\esp{\odet{L^0_d(x)}} = (2\pi)^\frac{r}{2}\frac{\vol{\S^{n-r}}}{\vol{\S^n}}.\qedhere
\end{equation}
\end{proof}

\subsubsection{Far off-diagonal asymptotics for the correlated terms}
\label{subsubsec far off diagonal correlated}

We can now focus on computing terms in the expression of $\D_d$ that depend on both $x$ and~$y$. For all $(x,y)\in M^2 \setminus \Delta$, $\ev^d_{x,y}(s_d)=\left(s_d(x),s_d(y)\right)$ is a centered Gaussian vector with variance operator:
\begin{equation}
\label{eq variance evxy}
\begin{aligned}
\ev^d_{x,y}(\ev^d_{x,y})^* &= \esp{\ev^d_{x,y}(s_d) \otimes \ev^d_{x,y}(s_d)^*}\\
&= \begin{pmatrix}
\esp{s_d(x) \otimes s_d(x)^*} & \esp{s_d(x) \otimes s_d(y)^*}\\ \esp{s_d(y) \otimes s_d(x)^*} & \esp{s_d(y) \otimes s_d(y)^*}
\end{pmatrix}\\
&= \begin{pmatrix}
E_d(x,x) & E_d(x,y)\\ E_d(y,x) & E_d(y,y)
\end{pmatrix},
\end{aligned}
\end{equation}
where we decomposed this operator according to the direct sum $\R\!\left(\E \otimes \L^d\right)_x \oplus \R\!\left(\E \otimes \L^d\right)_y$. Since we assumed $d \geq d_2$, $\odet{\ev_{x,y}^d}>0$ (see Prop.~\ref{prop amplitude 2 points}) and the distribution of $\left(s_d(x),s_d(y)\right)$ is non-degenerate.

We denote by $j_{x,y}^d:s \mapsto \left(s(x),s(y),\nabla^d_xs,\nabla^d_ys\right)$ the evaluation of the $1$-jet at $(x,y)$. Then $j_{x,y}^d$ is a linear map from $\R \H$ to
\begin{equation}
\label{eq target space 1jet}
\R\!\left(\E \otimes \L^d\right)_x \oplus \R\!\left(\E \otimes \L^d\right)_y \oplus \left(\R\!\left(\E \otimes \L^d\right)_x \otimes T^*_xM\right) \oplus \left(\R\!\left(\E \otimes \L^d\right)_y \otimes T^*_yM\right),
\end{equation}
and $j_{x,y}^d(s_d)$ is a centered Gaussian vector, with variance operator $j_{x,y}^d \left(j_{x,y}^d\right)^*$. We can split this variance operator according to the direct sum~\eqref{eq target space 1jet}. Then by Cor.~\ref{cor variance of the 1-jet sharp}, we have:
\begin{equation}
\label{eq variance jxy}
\begin{aligned}
&j_{x,y}^d \left(j_{x,y}^d\right)^* = \esp{j_{x,y}^d(s_d) \otimes \left(j_{x,y}^d(s_d)\right)^*}\\
&=\begin{pmatrix}
\esp{s_d(x)\otimes s_d(x)^*} & \esp{s_d(x)\otimes s_d(y)^*} & \esp{s_d(x)\otimes (\nabla^d_xs_d)^*} & \esp{s_d(x)\otimes (\nabla^d_ys_d)^*}\\
\esp{s_d(y)\otimes s_d(x)^*} & \esp{s_d(y)\otimes s_d(y)^*} & \esp{s_d(y)\otimes (\nabla^d_xs_d)^*} & \esp{s_d(y)\otimes (\nabla^d_ys_d)^*}\\
\esp{\nabla^d_xs_d\otimes s_d(x)^*} & \esp{\nabla^d_xs_d\otimes s_d(y)^*} & \esp{\nabla^d_xs_d\otimes (\nabla^d_xs_d)^*} & \esp{\nabla^d_xs_d\otimes (\nabla^d_ys_d)^*}\\
\esp{\nabla^d_ys_d\otimes s_d(x)^*} & \esp{\nabla^d_ys_d\otimes s_d(y)^*} & \esp{\nabla^d_ys_d\otimes (\nabla^d_xs_d)^*} & \esp{\nabla^d_ys_d\otimes (\nabla^d_ys_d)^*}
\end{pmatrix}\\
&= \begin{pmatrix}
E_d(x,x) & E_d(x,y) & \partial_y^\sharp E_d(x,x) & \partial_y^\sharp E_d(x,y)\\
E_d(y,x) & E_d(y,y) & \partial_y^\sharp E_d(y,x) & \partial_y^\sharp E_d(y,y)\\
\partial_xE_d(x,x) & \partial_xE_d(x,y) & \partial_x\partial_y^\sharp E_d(x,x) & \partial_x\partial_y^\sharp E_d(x,y)\\
\partial_xE_d(y,x) & \partial_xE_d(y,y) & \partial_x\partial_y^\sharp E_d(y,x) & \partial_x\partial_y^\sharp E_d(y,y)
\end{pmatrix}.
\end{aligned}
\end{equation}
Since the distribution of $\left(s_d(x),s_d(y)\right)$ is non-degenerate, the distribution of $\left(\nabla_x^ds,\nabla_y^ds\right)$ given that $\ev_{x,y}(s_d)=0$ is a centered Gaussian with variance operator:
\begin{equation}
\label{eq conditional variance full}
\begin{aligned}
\begin{pmatrix}
\partial_x\partial_y^\sharp E_d(x,x) & \partial_x\partial_y^\sharp E_d(x,y)\\
\partial_x\partial_y^\sharp E_d(y,x) & \partial_x\partial_y^\sharp E_d(y,y)
\end{pmatrix}&-\\
\begin{pmatrix}
\partial_xE_d(x,x) & \partial_xE_d(x,y)\\
\partial_xE_d(y,x) & \partial_xE_d(y,y)
\end{pmatrix}&
\begin{pmatrix}
E_d(x,x) & E_d(x,y)\\
E_d(y,x) & E_d(y,y)
\end{pmatrix}^{-1}
\begin{pmatrix}
\partial_y^\sharp E_d(x,x) & \partial_y^\sharp E_d(x,y)\\
\partial_y^\sharp E_d(y,x) & \partial_y^\sharp E_d(y,y)
\end{pmatrix}.
\end{aligned}
\end{equation}

\begin{dfn}
\label{def Lambda dxy}
For every $(x,y) \in M^2 \setminus \Delta$ and every $d$ large enough, we define $\Lambda_d(x,y)$ to be the operator such that $\frac{d^{n+1}}{\pi^n}\Lambda_d(x,y)$ equals~\eqref{eq conditional variance full}. That is, $\Lambda_d(x,y)$ is the conditional variance of $\left(\frac{\pi^n}{d^{n+1}}\right)^\frac{1}{2}\left(\nabla_x^ds,\nabla_y^ds\right)$ given that $\ev_{x,y}(s_d)=0$.
\end{dfn}

Let $C'>0$ be the constant appearing in the exponential in Thm.~\ref{thm off diag estimates}. We denote 
\begin{align}
\label{eq def bn}
&b_n = \frac{1}{C'}\left(\frac{n}{2}+1\right)\\
\intertext{and}
\label{eq def Delta d}
\Delta_d = &\left\{(x,y) \in M^2 \mvert \rho_g(x,y) < b_n\frac{\ln d}{\sqrt{d}} \right\},
\end{align}
where, as before, $\rho_g$ is the geodesic distance in $(M,g)$.

\begin{lem}
\label{lem off diagonal evxy}
For every $(x,y) \in M^2 \setminus \Delta_d$, we have:
\begin{equation*}
\label{eq lem off diagonal evxy}
\odet{\ev^d_{x,y}} = \odet{\ev_x^d}\odet{\ev_y^d} \left(1 + O\!\left(d^{-\frac{n}{2}-1}\right)\right),
\end{equation*}
where the error term is uniform in $(x,y)\in M^2 \setminus \Delta_d$
\end{lem}

\begin{proof}
For all $(x,y) \in M^2 \setminus \Delta_d$, we have $\rho_g(x,y) \geq b_n \frac{\ln d}{\sqrt{d}}$. Then, by Thm.~\ref{thm off diag estimates},
\begin{equation*}
\Norm{E_d(x,y)} \leq C_0 d^n \exp\left(-C' b_n \ln d \right) \leq C_0 d^{\frac{n}{2}-1}.
\end{equation*}
Then, by~\eqref{eq variance evxy} we have:
\begin{align*}
\ev^d_{x,y}\left(\ev_{x,y}^d\right)^* &= \begin{pmatrix}
E_d(x,x) & E_d(x,y) \\ E_d(y,x) & E_d(y,y)
\end{pmatrix}\\
&= \begin{pmatrix}
E_d(x,x) & 0 \\ 0 & E_d(y,y)
\end{pmatrix} + O\!\left(d^{\frac{n}{2}-1}\right).
\end{align*}
Besides, by~\eqref{eq value Bergman diag 00},
\begin{equation}
\label{eq estimate condition inverse}
\begin{pmatrix}
E_d(x,x) & 0 \\ 0 & E_d(y,y)
\end{pmatrix}^{-1} = \left(\frac{\pi}{d}\right)^n \left(\Id + O\!\left(d^{-1}\right)\right) = O\!\left(d^{-n}\right),
\end{equation}
so that
\begin{equation}
\label{eq easier condition inverse}
\begin{pmatrix}
E_d(x,x) & E_d(x,y) \\ E_d(y,x) & E_d(y,y)
\end{pmatrix} = \begin{pmatrix}
E_d(x,x) & 0 \\ 0 & E_d(y,y)
\end{pmatrix} \left(\Id + O\!\left(d^{-\frac{n}{2}-1}\right)\right).
\end{equation}
We conclude the proof by taking the square root of the determinant of this last equality (recall~\eqref{eq variance evx}).
\end{proof}

\begin{lem}
\label{lem off diagonal cond exp}
For every $(x,y) \in M^2 \setminus \Delta_d$, we have:
\begin{multline*}
\label{eq lem off cond exp}
\espcond{\odet{\nabla^d_{x}s_d}\odet{\nabla^d_{y}s_d}}{\ev_{x,y}^d(s_d)=0} =\\ \espcond{\odet{\nabla^d_{x}s_d}}{s_d(x)=0}\espcond{\odet{\nabla^d_{y}s_d}}{s_d(y)=0}\left(1 + O\!\left(d^{-\frac{n}{2}-1}\right)\right),
\end{multline*}
where the error term is uniform in $(x,y)\in M^2 \setminus \Delta_d$
\end{lem}

This lemma is a consequence of the following technical result.
\begin{lem}
\label{lem almost uncorrelated}
For every $(x,y) \in M^2 \setminus \Delta_d$, we have:
\begin{equation*}
\Lambda_d(x,y) = \begin{pmatrix}
\Lambda_d(x) & 0 \\ 0 & \Lambda_d(y)
\end{pmatrix} \left(\Id+ O\!\left(d^{-\frac{n}{2}-1}\right)\right),
\end{equation*}
uniformly in $(x,y) \in M^2 \setminus \Delta_d$.
\end{lem}

\begin{proof}[Proof of Lemma~\ref{lem off diagonal cond exp}]
Let $\left(L_d(x),L_d(y)\right)$ and $\left(L'_d(x),L'_d(y)\right)$ be centered Gaussian vectors in
\begin{equation*}
\left(\R\!\left(\E \otimes \L^d\right)_x \otimes T^*_xM\right) \oplus \left(\R\!\left(\E \otimes \L^d\right)_y \otimes T^*_yM\right)
\end{equation*}
such that: the variance of $\left(L'_d(x),L'_d(y)\right)$ is $\Lambda_d(x,y)$ (recall Def.~\ref{def Lambda dxy}), and $L_d(x)$ and $L_d(y)$ are independent with variances $\Lambda_d(x)$ and $\Lambda_d(y)$ respectively (see~\eqref{eq def Lambda d x}). Then, the distribution of $\left(L_d(x),L_d(y)\right)$ is a centered Gaussian with variance $\left(\begin{smallmatrix} \Lambda_d(x) & 0 \\ 0 & \Lambda_d(y) \end{smallmatrix}\right)$. From the definitions of $\Lambda_d(x)$, $\Lambda_d(y)$ and $\Lambda_d(x,y)$, we have:
\begin{align*}
\espcond{\odet{\nabla^d_{x}s_d}}{s_d(x)=0} &= \left(\frac{d^{n+1}}{\pi^n}\right)^{\hspace{-1mm}\frac{r}{2}}\esp{\odet{L_d(x)}},\\
\espcond{\odet{\nabla^d_{y}s_d}}{s_d(y)=0} &= \left(\frac{d^{n+1}}{\pi^n}\right)^{\hspace{-1mm}\frac{r}{2}}\esp{\odet{L_d(y)}},\\
\espcond{\odet{\nabla^d_{x}s_d}\!\odet{\nabla^d_{y}s_d}}{\ev_{x,y}^d(s_d)=0} &= \left(\frac{d^{n+1}}{\pi^n}\right)^{\hspace{-1mm}r} \esp{\odet{L'_d(x)}\!\odet{L'_d(y)}}.
\end{align*}
Since $L_d(x)$ and $L_d(y)$ are independent, we only need to prove that:
\begin{equation}
\label{eq easier}
\esp{\odet{L'_d(x)}\odet{L'_d(y)}} =\esp{\odet{L_d(x)}\odet{L_d(y)}}\left(1 + O\!\left(d^{-\frac{n}{2}-1}\right)\right).
\end{equation}

By Lemma~\ref{lem almost uncorrelated},
\begin{equation}
\label{eq estimates det Lambda dxy}
\begin{aligned}
\det\left(\Lambda_d(x,y)\right) &= \det\left(\begin{pmatrix}
\Lambda_d(x) & 0 \\ 0 & \Lambda_d(y)
\end{pmatrix}\left(\Id+ O\!\left(d^{-\frac{n}{2}-1}\right)\right) \right)\\
&= \det\left(\Lambda_d(x)\right)\det\left(\Lambda_d(x)\right) \left(1+ O\!\left(d^{-\frac{n}{2}-1}\right)\right).
\end{aligned}
\end{equation}
Besides Lem.~\ref{lem almost uncorrelated} shows that:
\begin{equation*}
\label{eq inverse Lambda dxy}
\Lambda_d(x,y)^{-1} = \begin{pmatrix}
\Lambda_d(x) & 0 \\ 0 & \Lambda_d(y)
\end{pmatrix}^{-1} \left(\Id+ O\!\left(d^{-\frac{n}{2}-1}\right)\right).
\end{equation*}
By Cor.~\ref{cor diag estimates} and eq.~\eqref{eq def Lambda d x}, we have: $\left(\begin{smallmatrix}\Lambda_d(x) & 0 \\ 0 & \Lambda_d(y)
\end{smallmatrix}\right) = \Id + O\!\left(d^{-1}\right)$. Hence,
\begin{equation}
\label{eq Lambda d x Lambda d y -1}
\begin{pmatrix}\Lambda_d(x) & 0 \\ 0 & \Lambda_d(y)
\end{pmatrix}^{-1} = \Id + O\!\left(d^{-1}\right)
\end{equation}
uniformly in $(x,y)$, and 
\begin{equation*}
\Lambda_d(x,y)^{-1} - \begin{pmatrix}
\Lambda_d(x) & 0 \\ 0 & \Lambda_d(y)
\end{pmatrix}^{-1} =  O\!\left(d^{-\frac{n}{2}-1}\right).
\end{equation*}
Thus there exists $K>0$ such that, for all $d$ large enough,
\begin{equation*}
\Norm{\Lambda_d(x,y)^{-1} - \begin{pmatrix}
\Lambda_d(x) & 0 \\ 0 & \Lambda_d(y)
\end{pmatrix}^{-1}} \leq \frac{K}{d^{\frac{n}{2}+1}}.
\end{equation*}
Then, for every $L=(L_1,L_2) \in \left(\R\!\left(\E \otimes \L^d\right)_x \otimes T^*_xM\right) \oplus \left(\R\!\left(\E \otimes \L^d\right)_y \otimes T^*_yM\right)$, we have:
\begin{equation*}
\norm{\exp\left(-\frac{1}{2}\prsc{\left(\Lambda_d(x,y)^{-1}-\left(\begin{smallmatrix}
\Lambda_d(x) & 0 \\ 0 & \Lambda_d(y)
\end{smallmatrix}\right)^{-1}\right)L}{L}\right)-1} \leq \frac{K\Norm{L}^2}{2d^{\frac{n}{2}+1}} \exp\left(\frac{K\Norm{L}^2}{2d^{\frac{n}{2}+1}}\right).
\end{equation*}
Let $\dx L$ denote the normalized Lebesgue measure on this vector space. We have:
\begin{align*}
&\begin{aligned}
(2\pi)^{nr} \left\lvert \det\left(\Lambda_d(x,y)\right)^\frac{1}{2}\right. &\esp{\odet{L'_d(x)}\odet{L'_d(y)}}\\
&-\left. \det\left(\Lambda_d(x)\right)^\frac{1}{2}\det\left(\Lambda_d(y)\right)^\frac{1}{2} \esp{\odet{L_d(x)}\odet{L_d(y)}}\right\rvert
\end{aligned}\\
&\begin{aligned}
\leq \int \odet{L_1} & \odet{L_2} \exp\left(-\frac{1}{2}\prsc{\left(\begin{smallmatrix}
\Lambda_d(x) & 0 \\ 0 & \Lambda_d(y)
\end{smallmatrix}\right)^{-1}L}{L}\right)\times\\
&\norm{\exp\left(-\frac{1}{2}\prsc{\left(\Lambda_d(x,y)^{-1}-\left(\begin{smallmatrix}
\Lambda_d(x) & 0 \\ 0 & \Lambda_d(y)
\end{smallmatrix}\right)^{-1}\right)L}{L}\right)-1} \dx L
\end{aligned}\\
&\begin{aligned}
\leq \frac{K}{2d^{\frac{n}{2}+1}}\int \odet{L_1}& \odet{L_2} \Norm{L}^2\times\\
&\exp\left(-\frac{1}{2}\prsc{\left(\left(\begin{smallmatrix}
\Lambda_d(x) & 0 \\ 0 & \Lambda_d(y)
\end{smallmatrix}\right)^{-1}-\frac{K}{2d^{\frac{n}{2}+1}}\Id\right)L}{L}\right)\dx L
\end{aligned}\\
& = O\!\left(d^{-\frac{n}{2}-1}\right).
\end{align*}

Let us prove the last equality. By eq.~\eqref{eq Lambda d x Lambda d y -1}, for every $d$ large enough (uniform in $(x,y)$),
\begin{equation*}
\Norm{\begin{pmatrix}
\Lambda_d(x) & 0 \\ 0 & \Lambda_d(y)
\end{pmatrix}^{-1} -\left(1+\frac{K}{2d^{\frac{n}{2}+1}}\right)\Id} \leq \frac{1}{2},
\end{equation*}
so that:
\begin{multline*}
\int \odet{L_1} \odet{L_2} \Norm{L}^2 \exp\left(-\frac{1}{2}\prsc{\left(\left(\begin{smallmatrix}
\Lambda_d(x) & 0 \\ 0 & \Lambda_d(y)
\end{smallmatrix}\right)^{-1}-\frac{K}{2d^{\frac{n}{2}+1}}\Id\right)L}{L}\right)\dx L\\
\leq \int \odet{L_1} \odet{L_2} \Norm{L}^2 \exp\left(-\frac{1}{4}\Norm{L}^2\right) \dx L.
\end{multline*}
And the last integral is finite since $\odet{L_1} \odet{L_2} \Norm{L}^2$ is the norm of a polynomial in $L$.

Eq.~\eqref{eq det Lambda dx} and~\eqref{eq estimates det Lambda dxy} show that $\det(\Lambda_d(x,y)) = 1 +O\!\left(d^{-1}\right)$. Then, by the previous computations and eq.~\eqref{eq estimates det Lambda dxy}, we have:
\begin{multline*}
\esp{\odet{L'_d(x)}\odet{L'_d(y)}}\\
\begin{aligned}
&= \left(\frac{\det\left(\Lambda_d(x)\right)\det\left(\Lambda_d(y)\right)}{\det\left(\Lambda_d(x,y)\right)}\right)^\frac{1}{2} \esp{\odet{L_d(x)}\odet{L_d(y)}} +  O\!\left(d^{-\frac{n}{2}-1}\right)\\
&= \esp{\odet{L_d(x)}\odet{L_d(y)}}\left(1 + O\!\left(d^{-\frac{n}{2}-1}\right)\right)+  O\!\left(d^{-\frac{n}{2}-1}\right).
\end{aligned}
\end{multline*}
Equations~\eqref{eq last lemma} and \eqref{eq value exp odet lim} prove that
\begin{equation*}
\esp{\odet{L_d(x)}\odet{L_d(y)}} = \esp{\odet{L_d(x)}}\esp{\odet{L_d(y)}}
\end{equation*}
converges to some positive constant. This proves~\eqref{eq easier} and establishes Lemma~\ref{lem off diagonal cond exp}.
\end{proof}

\begin{proof}[Proof of Lemma~\ref{lem almost uncorrelated}]
First, recall that $\Lambda_d(x,y)$, $\Lambda_d(x)$ and $\Lambda_d(y)$ do not depend on the choice of $\nabla^d$ (see Rem.~\ref{rem choice of nablad}). In this proof, we use the Chern connection which is both real and metric. Let $(x,y)\in M^2 \setminus \Delta_d$, then $\rho_g(x,y) \geq b_n\frac{\ln d}{\sqrt{d}}$. By Thm.~\ref{thm off diag estimates}, we have:
\begin{equation*}
\Norm{\partial_xE_d(x,y)} \leq C_1 d^{n+\frac{1}{2}} \exp\left(-C' b_n \ln d \right) \leq C_1 d^{\frac{n}{2}-\frac{1}{2}}.
\end{equation*}
Similarly, $\Norm{\partial_xE_d(y,x)}$, $\Norm{\partial_y^\sharp E_d(x,y)}$ and $\Norm{\partial_y^\sharp E_d(y,x)}$ are smaller than $C_1 d^\frac{n-1}{2}$. Then
\begin{align}
\begin{pmatrix}
\partial_xE_d(x,x) & \partial_xE_d(x,y)\\
\partial_xE_d(y,x) & \partial_xE_d(y,y)
\end{pmatrix} &= \begin{pmatrix}
\partial_xE_d(x,x) & 0\\
0 & \partial_xE_d(y,y)
\end{pmatrix} + O\!\left(d^\frac{n-1}{2}\right)\\
\begin{pmatrix}
\partial_y^\sharp E_d(x,x) & \partial_y^\sharp E_d(x,y)\\
\partial_y^\sharp E_d(y,x) & \partial_y^\sharp E_d(y,y)
\end{pmatrix} &= \begin{pmatrix}
\partial_y^\sharp E_d(x,x) & 0\\
0 & \partial_y^\sharp E_d(y,y)
\end{pmatrix}+ O\!\left(d^\frac{n-1}{2}\right)
\intertext{and, by eq.~\eqref{eq easier condition inverse},}
\begin{pmatrix}
E_d(x,x) & E_d(x,y) \\ E_d(y,x) & E_d(y,y)
\end{pmatrix}^{-1} &= \begin{pmatrix}
E_d(x,x) & 0 \\ 0 & E_d(y,y)
\end{pmatrix}^{-1} \left(\Id + O\!\left(d^{-\frac{n}{2}-1}\right)\right).
\end{align}
Using eq.~\eqref{eq value Bergman diag 10}, \eqref{eq value Bergman diag 01} and~\eqref{eq estimate condition inverse}, we get:
\begin{multline}
\label{eq conditioning easier}
\begin{pmatrix}
\partial_xE_d(x,x) & \partial_xE_d(x,y)\\
\partial_xE_d(y,x) & \partial_xE_d(y,y)
\end{pmatrix}
\begin{pmatrix}
E_d(x,x) & E_d(x,y) \\ E_d(y,x) & E_d(y,y)
\end{pmatrix}^{-1}
\begin{pmatrix}
\partial_y^\sharp E_d(x,x) & \partial_y^\sharp E_d(x,y)\\
\partial_y^\sharp E_d(y,x) & \partial_y^\sharp E_d(y,y)
\end{pmatrix}\\
=\begin{pmatrix}
\partial_xE_d(x,x) & 0\\
0 & \partial_xE_d(y,y)
\end{pmatrix}
\begin{pmatrix}
E_d(x,x) & 0 \\ 0 & E_d(y,y)
\end{pmatrix}^{-1}
\begin{pmatrix}
\partial_y^\sharp E_d(x,x) & 0\\
0 & \partial_y^\sharp E_d(y,y)
\end{pmatrix}\\ +O\!\left(d^{\frac{n}{2}-1}\right)
\end{multline}

Using Thm.~\ref{thm off diag estimates} once more, we know that $\Norm{\partial_x\partial_y^\sharp E_d(x,y)}$ and $\Norm{\partial_x\partial_y^\sharp E_d(y,x)}$ are smaller than $C_2 d^\frac{n}{2}$. Then we have:
\begin{equation}
\label{eq conditioning easier 2}
\begin{pmatrix}
\partial_x\partial_y^\sharp E_d(x,x) & \partial_x\partial_y^\sharp E_d(x,y)\\
\partial_x\partial_y^\sharp E_d(y,x) & \partial_x\partial_y^\sharp E_d(y,y)
\end{pmatrix} = \begin{pmatrix}
\partial_x\partial_y^\sharp E_d(x,x) & 0\\
0 & \partial_x\partial_y^\sharp E_d(y,y)
\end{pmatrix} + O\!\left(d^\frac{n}{2}\right).
\end{equation}
We substract eq.~\eqref{eq conditioning easier} to eq.~\eqref{eq conditioning easier 2} and divide by $\frac{d^{n+1}}{\pi^n}$. By definition of $\Lambda_d(x,y)$, $\Lambda_d(x)$ and $\Lambda_d(y)$ (see Def.~\ref{def Lambda dxy} and eq.~\eqref{eq def Lambda d x}),
\begin{equation*}
\Lambda_d(x,y) = \begin{pmatrix}
\Lambda_d(x) & 0\\ 0 & \Lambda_d(y)
\end{pmatrix} + O\!\left(d^{-\frac{n}{2}-1}\right) = \begin{pmatrix}
\Lambda_d(x) & 0\\ 0 & \Lambda_d(y)
\end{pmatrix} \left(\Id + O\!\left(d^{-\frac{n}{2}-1}\right)\right),
\end{equation*}
where we used the fact that $\Lambda_d(x) = \Id + O\!\left(d^{-1}\right) = \Lambda_d(y)$ to obtain the last equality.
\end{proof}

\begin{prop}
\label{prop off diagonal is small}
Let $\phi_1,\phi_2 \in \mathcal{C}^0(M)$, then we have the following as $d \to +\infty$:
\begin{equation*}
\int_{M^2 \setminus \Delta_d} \phi_1(x)\phi_2(y) \D_d(x,y) \rmes{M}^2 = \Norm{\phi_1}_\infty\Norm{\phi_2}_\infty O\!\left(d^{r-\frac{n}{2}-1}\right),
\end{equation*}
where the error term is independent of $(\phi_1,\phi_2)$.
\end{prop}

\begin{proof}
We combine Lemmas~\ref{lem off diagonal evxy} and~\ref{lem off diagonal cond exp}, which gives:
\begin{multline*}
\frac{\espcond{\odet{\nabla^d_{x}s_d}\odet{\nabla^d_{y}s_d}}{\ev_{x,y}^d(s_d)=0}}{\odet{\ev_{x,y}^d}}=\\
\frac{\espcond{\odet{\nabla^d_{x}s_d}}{s_d(x)=0}\espcond{\odet{\nabla^d_{y}s_d}}{s_d(y)=0}}{\odet{\ev_x^d}\odet{\ev_y^d}} \left(1+ O\!\left(d^{-\frac{n}{2}-1}\right)\right)
\end{multline*}
for all $(x,y) \in M^2 \setminus \Delta_d$. Besides, by Lemmas~\ref{lem estimates odet evx} and~\ref{lem estimates cond exp x},
\begin{equation*}
\frac{\espcond{\odet{\nabla^d_{x}s_d}}{s_d(x)=0}}{\odet{\ev_x^d}} = O\!\left(d^\frac{r}{2}\right) = \frac{\espcond{\odet{\nabla^d_{y}s_d}}{s_d(y)=0}}{\odet{\ev_y^d}}.
\end{equation*}
Recalling the definition of $\D_d$ (eq.~\eqref{eq def density}), we obtain that:
\begin{equation*}
\forall (x,y) \in M^2 \setminus \Delta_d, \qquad \D_d(x,y) = O\!\left(d^{r-\frac{n}{2}-1}\right),
\end{equation*}
uniformly in $(x,y) \notin \Delta_d$. Then, for any continuous $\phi_1$ and $\phi_2 \in \mathcal{C}^0(M)$, we have:
\begin{align*}
\norm{\int_{M^2 \setminus \Delta_d} \phi_1(x)\phi_2(y) \D_d(x,y) \rmes{M}^2} &\leq \Norm{\phi_1}_\infty\Norm{\phi_2}_\infty \vol{M^2} \left(\sup_{M^2 \setminus \Delta_d} \norm{\D_d}\right)\\
&= \Norm{\phi_1}_\infty\Norm{\phi_2}_\infty O\!\left(d^{r-\frac{n}{2}-1}\right),
\end{align*}
and the error term does not depend on $(\phi_1,\phi_2)$.
\end{proof}

\subsubsection{Properties of the limit distribution}
\label{subsubsec properties of the limit distribution}

Before we tackle the computation of the dominant term in~\eqref{eq variance2}, that is the integral over $\Delta_d$, we introduce the random variables that will turn out to be the scaling limits of $\left(\nabla^d_x s_d, \nabla^d_y s_d\right)$ given that $\ev_{x,y}^d(s_d) =0$. We also establish some of their properties.

\begin{ntn}
\label{ntn z* tens z}
Let $x \in M$ and $z \in T_xM$, we denote by $z^* \otimes z \in T^*_xM \otimes T_xM$ the linear map:
\begin{equation*}
\begin{array}{rccc}
z^* \otimes z : & T^*_xM & \longrightarrow & T^*_xM.\\
 & \eta & \longmapsto & \eta(z)z^*
\end{array}
\end{equation*}
Let $\left(\deron{}{x_1},\dots,\deron{}{x_n}\right)$ be an orthonormal basis of $T_xM$ and let $\left( dx_1,\dots,dx_n\right)$ denote its dual basis. If $z = \sum z_i \deron{}{x_i}$ then $z^* \otimes z = \sum z_i z_j dx_i \otimes \deron{}{x_j}$, i.e.~the matrix of $z^* \otimes z$ in $\left(dx_1,\dots,dx_n\right)$ is $\left(z_i z_j\right)_{1\leq i,j \leq n}$.
\end{ntn}

\begin{dfn}
\label{def Lambda z inf}
For all $x \in M$ and $z \in T_xM \setminus\{0\}$, we define
\begin{equation*}
\Lambda_x(z) \in \End\left(\R\left(\E \otimes \L^d \right)_x \otimes T^*_xM \otimes \R^2\right)
\end{equation*}
by:
\begin{align*}
\Lambda_x(z) = \begin{pmatrix}
\Id_{T^*_xM} - \frac{e^{-\Norm{z}^2}}{1-e^{-\Norm{z}^2}} z^* \otimes z & e^{-\frac{1}{2}\Norm{z}^2}\left(\Id_{T^*_xM} - \frac{z^*\otimes z}{1-e^{-\Norm{z}^2}}\right)\\
e^{-\frac{1}{2}\Norm{z}^2}\left(\Id_{T^*_xM} - \frac{z^*\otimes z}{1-e^{-\Norm{z}^2}}\right) & \Id_{T^*_xM} - \frac{e^{-\Norm{z}^2}}{1-e^{-\Norm{z}^2}} z^* \otimes z
\end{pmatrix}\otimes \Id_{\R\left(\E \otimes \L^d \right)_x}.
\end{align*}
\end{dfn}

We need information about $\Lambda_x(z)$, especially concerning the vanishing of its eigenvalues. This will be useful in the estimates involving $\Lambda_x(z)$ below.

\begin{lem}
\label{lem eigenvalues Lambda z}
For all $x \in M$ and $z \in T_xM \setminus\{0\}$, the eigenvalues of $\Lambda_x(z)$ are:
\begin{itemize}
\item $1 - e^{-\frac{1}{2}\Norm{z}^2}$ and $1 + e^{-\frac{1}{2}\Norm{z}^2}$, each with multiplicity $(n-1)r$,
\item $\dfrac{1-e^{-\Norm{z}^2}+\Norm{z}^2e^{-\frac{1}{2}\Norm{z}^2}}{1+e^{-\frac{1}{2}\Norm{z}^2}}$ and $\dfrac{1-e^{-\Norm{z}^2}-\Norm{z}^2e^{-\frac{1}{2}\Norm{z}^2}}{1-e^{-\frac{1}{2}\Norm{z}^2}}$, each with multiplicity $r$.
\end{itemize}
\end{lem}

\begin{proof}
Let  $x \in M$ and $z \in B_{T_xM}(0,b_n \ln d)\setminus\{0\}$. By definition of $\Lambda_x(z)$, its eigenvalues are the same as that of
\begin{equation}
\label{eq big operator}
\begin{pmatrix}
\Id_{T^*_xM} - \frac{e^{-\Norm{z}^2}}{1-e^{-\Norm{z}^2}} z^* \otimes z & e^{-\frac{1}{2}\Norm{z}^2}\left(\Id_{T^*_xM} - \frac{z^*\otimes z}{1-e^{-\Norm{z}^2}}\right)\\
e^{-\frac{1}{2}\Norm{z}^2}\left(\Id_{T^*_xM} - \frac{z^*\otimes z}{1-e^{-\Norm{z}^2}}\right) & \Id_{T^*_xM} - \frac{e^{-\Norm{z}^2}}{1-e^{-\Norm{z}^2}} z^* \otimes z
\end{pmatrix},
\end{equation}
with multiplicities multiplied by $r$. Hence, it is enough to compute the eigenvalues of the operator~\eqref{eq big operator}.

Let us choose an orthonormal basis $\left(\deron{}{x_1},\dots, \deron{}{x_n}\right)$ of $T_xM$ such that $z = \Norm{z}\deron{}{x_1}$, and let us denote by $\left(dx_1,\dots,dx_n\right)$ the dual basis. Then, $z^* \otimes z = \Norm{z}^2 dx_1 \otimes \deron{}{x_1}$. Let $(e_1,e_2)$ denote the canonical basis of $\R^2$, then the matrix of the operator \eqref{eq big operator} in the orthonormal basis $\left(e_1 \otimes dx_1,e_2 \otimes dx_1,\dots,e_1 \otimes dx_n,\dots,e_2 \otimes dx_n \right)$ is:
\begin{equation}
\label{eq big matrix}
\left(\begin{array}{cc|c}
1 - \frac{\Norm{z}^2 e^{-\Norm{z}^2}}{1-e^{-\Norm{z}^2}} & e^{-\frac{1}{2}\Norm{z}^2}\left(1 - \frac{\Norm{z}^2}{1-e^{-\Norm{z}^2}}\right) & 0 \\
e^{-\frac{1}{2}\Norm{z}^2}\left(1 - \frac{\Norm{z}^2}{1-e^{-\Norm{z}^2}}\right) & 1 - \frac{\Norm{z}^2 e^{-\Norm{z}^2}}{1-e^{-\Norm{z}^2}} & 0 \\
\hline
0 & 0 & \begin{pmatrix}
1 & e^{-\frac{1}{2}\Norm{z}^2} \\ e^{-\frac{1}{2}\Norm{z}^2} & 1
\end{pmatrix} \otimes
I_{n-1}\\
\end{array}\right),
\end{equation}
where $I_{n-1}$ stands for the identity matrix of size $n-1$.

The bottom-right block has eigenvalues $1-e^{-\frac{1}{2}\Norm{z}^2}$ and $1+e^{-\frac{1}{2}\Norm{z}^2}$, each with multiplicity $n-1$. To conclude the proof of Lemma~\ref{lem eigenvalues Lambda z}, we only need to observe that, for all $t>0$, the eigenvalues of
\begin{equation*}
\begin{pmatrix}
1 - \frac{t e^{-t}}{1-e^{-t}} & e^{-\frac{1}{2}t}\left(1 - \frac{t}{1-e^{-t}}\right)\\
e^{-\frac{1}{2}t}\left(1 - \frac{t}{1-e^{-t}}\right) & 1 - \frac{t e^{-t}}{1-e^{-t}}
\end{pmatrix}
\end{equation*}
are:
\begin{align*}
1 - \frac{t e^{-t}}{1-e^{-t}} + e^{-\frac{1}{2}t}\left(1 - \frac{t}{1-e^{-t}}\right) &= \frac{1-e^{-t}-te^{-\frac{1}{2}t}}{1-e^{-\frac{1}{2}t}}\\
\intertext{and}
1 - \frac{t e^{-t}}{1-e^{-t}} - e^{-\frac{1}{2}t}\left(1 - \frac{t}{1-e^{-t}}\right) &= \frac{1-e^{-t}+te^{-\frac{1}{2}t}}{1+e^{-\frac{1}{2}t}}.
\end{align*}
Note that the latter one is the largest.
\end{proof}

\begin{dfn}
\label{def f}
We define the function $f : (0,+\infty) \to \R$ by:
\begin{equation*}
\forall t >0, \qquad f(t) = \frac{1-e^{-\frac{1}{2}t}}{1-e^{-t}-te^{-\frac{1}{2}t}}.
\end{equation*}
\end{dfn}

\begin{cor}
\label{cor Lambda z det and inverse}
Let  $x \in M$ and $z \in T_xM\setminus\{0\}$, then we have:
\begin{multline}
\label{eq Lambda z det}
\det\left(\Lambda_x(z)\right) =\\
\left(1-e^{-\Norm{z}^2}\right)^{r(n-2)}\left(1-e^{-\Norm{z}^2}+\Norm{z}^2 e^{-\frac{1}{2}\Norm{z}^2} \right)^r\left(1-e^{-\Norm{z}^2}-\Norm{z}^2 e^{-\frac{1}{2}\Norm{z}^2}\right)^r>0.
\end{multline}
Moreover,
\begin{align}
\label{eq Lambda z inverse}
\Norm{\Lambda_x(z)} &< 2 & &\text{and} & \Norm{\Lambda_x(z)^{-1}} &= f\left(\Norm{z}^2\right),
\end{align}
where $\Norm{\cdot}$ denote the operator norm on $\End\left(\R^2 \otimes \R\left(\E \otimes \L^d \right)_x \otimes T^*_xM\right)$.
\end{cor}

\begin{proof}
First, the formula for $\det\left(\Lambda_x(z)\right)$ is a direct consequence of Lem.~\ref{lem eigenvalues Lambda z}, and we only need to check that the eigenvalues of $\Lambda_x(z)$ are positive. Clearly, $1\pm e^{-\frac{1}{2}t} >0$ when $t>0$.

Then, for all positive $t$, we have:
\begin{equation*}
\frac{1-e^{-t}-te^{-\frac{1}{2}t}}{1-e^{-\frac{1}{2}t}} = \frac{e^{-\frac{1}{2}t}}{1-e^{-\frac{1}{2}t}}\left(e^{\frac{1}{2}t}-e^{-\frac{1}{2}t}-t\right) = \frac{e^{-\frac{1}{2}t}}{1-e^{-\frac{1}{2}t}}\left( 2 \sinh\left(\frac{t}{2}\right) - t\right),
\end{equation*}
and $2 \sinh\left(\frac{t}{2}\right) > t$. Besides,
\begin{equation*}
\frac{1-e^{-t}+te^{-\frac{1}{2}t}}{1+e^{-\frac{1}{2}t}} = \frac{e^{-\frac{1}{2}t}}{1+e^{-\frac{1}{2}t}}\left( 2 \sinh\left(\frac{t}{2}\right) + t\right) > 0.
\end{equation*}

Recall that $\Norm{\Lambda_x(z)}$ is the larger eigenvalue of $\Lambda_x(z)$, and $\Norm{\Lambda_x(z)^{-1}}$ is the inverse of the smallest eigenvalue of $\Lambda_x(z)$. For all $t>0$ we have
\begin{equation*}
0 < 1-e^{-\frac{t}{2}} < 1 + e^{-\frac{t}{2}} < 2.
\end{equation*}
Besides,
\begin{equation*}
\frac{1-e^{-t}-te^{-\frac{1}{2}t}}{1-e^{-\frac{1}{2}t}} + \frac{1-e^{-t}+te^{-\frac{1}{2}t}}{1+e^{-\frac{1}{2}t}} = 2\left(1 - \frac{t e^{-t}}{1-e^{-t}}\right) < 2,
\end{equation*}
and we just proved that both these terms are positive. Hence, each of them is smaller than~$2$. Thus, all the eigenvalues of $\Lambda_x(z)$ are smaller than $2$ and $\Norm{\Lambda_x(z)}< 2$.

For all $t>0$,
\begin{equation*}
\frac{1-e^{-t}-te^{-\frac{1}{2}t}}{1-e^{-\frac{1}{2}t}} < 1-e^{-\frac{t}{2}} \iff 1-e^{-t}-te^{-\frac{1}{2}t} < 1 - 2e^{-\frac{t}{2}} + e^{-t} \iff 1 - \frac{t}{2} < e^{-\frac{t}{2}},
\end{equation*}
and this is always true by convexity of the exponential. Thus, the smallest eigenvalue of $\Lambda_x(z)$ is $\dfrac{1-e^{-\Norm{z}^2}-\Norm{z}^2e^{-\frac{1}{2}\Norm{z}^2}}{1-e^{-\frac{1}{2}\Norm{z}^2}}=\dfrac{1}{f(\Norm{z}^2)}$, which proves our last claim.
\end{proof}

\begin{rem}
\label{rem norm Lambda z inverse}
To better understand the estimate~\eqref{eq Lambda z inverse}, note that $f$ is a decreasing function on $(0,+\infty)$. Moreover,
\begin{align*}
f(t) & \xrightarrow[t \to + \infty]{} 1, & &\text{and} & f(t) & \sim \frac{12}{t^2}
\end{align*}
when $t$ goes to $0$.
\end{rem}

\begin{dfn}
\label{def Lx 0 z}
For every $x \in M$, and $z \in T_xM \setminus \{0\}$, let $\left(L_x(0),L_x(z)\right)$ be a centered Gaussian vector in $\R^2 \otimes \R\left(\E \otimes \L^d \right)_x \otimes T^*_xM$ with variance operator $\Lambda_x(z)$.
\end{dfn}

Recall that we defined the random vector $\left(X(t),Y(t)\right)$ for all $t>0$ in the introduction (see Def.~\ref{def XYt}). Then $\left(X(t),Y(t)\right)$ and $\left(L_x(0),L_x(z)\right)$ are related as follows.

\begin{lem}
\label{lem same distribution}
Let $x \in M$ and $z \in T_xM \setminus \{0\}$, then there exists an orthonormal basis of $T_xM$ such that, for every orthonormal basis of $\R\left(\E \otimes \L^d\right)_x$, the couple of $r\times n$ matrices associated with $\left(L_x(0),L_x(z)\right)$ in these bases is distributed as $(X(\Norm{z}^2),Y(\Norm{z}^2))$.
\end{lem}

\begin{proof}
As in the proof of Lem.~\ref{lem eigenvalues Lambda z}, let us choose an orthonormal basis $\left(\deron{}{x_1},\dots,\deron{}{x_n}\right)$ of $T_xM$ such that $z = \Norm{z}\deron{}{x_1}$. Let $\left(dx_1,\dots,dx_n\right)$ denote its dual basis. Let $\left(\zeta_1,\dots,\zeta_r\right)$ be any orthonormal basis of $\R \left(\E \otimes \L^d \right)_x$, and let $(e_1,e_2)$ denote the canonical basis of $\R^2$.

Then $z^* \otimes z = \Norm{z}^2 dx_1 \otimes \deron{}{x_1}$ and the matrix of the operator~\eqref{eq big operator} in the orthonormal basis $\left(e_1\otimes dx_1,\dots, e_1 \otimes dx_n,e_2 \otimes dx_1,\dots, e_2 \otimes dx_n\right)$ is:
\begin{equation}
\label{eq big matrix 0}
\left(\begin{array}{cc|cc}
1 - \frac{\Norm{z}^2 e^{-\Norm{z}^2}}{1-e^{-\Norm{z}^2}} & 0 & e^{-\frac{1}{2}\Norm{z}^2}\left(1 - \frac{\Norm{z}^2}{1-e^{-\Norm{z}^2}}\right) & 0 \\
0 & I_{n-1} & 0 & e^{-\frac{1}{2}\Norm{z}^2} I_{n-1}\\
\hline
e^{-\frac{1}{2}\Norm{z}^2}\left(1 - \frac{\Norm{z}^2}{1-e^{-\Norm{z}^2}}\right) & 0 & 1 - \frac{\Norm{z}^2 e^{-\Norm{z}^2}}{1-e^{-\Norm{z}^2}} & 0 \\
0 & e^{-\frac{1}{2}\Norm{z}^2} I_{n-1} & 0 & I_{n-1}
\end{array}\right),
\end{equation}
where $I_{n-1}$ stands for the identity matrix of size $n-1$. Since $\Lambda_x(z)$ equals this operator tensored by $\Id_{\R \left(\E \otimes \L^d \right)_x}$, the matrix of $\Lambda_x(z)$ in the orthonormal basis:
\begin{multline*}
\left(e_1\otimes dx_1\otimes \zeta_1,\dots, e_1 \otimes dx_n\otimes \zeta_1,e_2 \otimes dx_1\otimes \zeta_1,\dots, e_2 \otimes dx_n\otimes \zeta_1,\right. \\
\left. e_1\otimes dx_1\otimes \zeta_2,\dots,e_2\otimes dx_n\otimes \zeta_2, \dots e_1\otimes dx_1\otimes \zeta_r,\dots,e_2\otimes dx_n\otimes \zeta_r\right)
\end{multline*}
is exactly the variance matrix of $(X(\Norm{z}^2),Y(\Norm{z}^2))$ (cf.~Def.~\ref{def XYt}).

Let us denote by $M_x(0)$ and $M_x(z)$ the matrices of $L_x(0)$ and $L_x(d)$ in the bases $\left(\deron{}{x_1},\dots,\deron{}{x_n}\right)$ and $\left(\zeta_1,\dots,\zeta_r\right)$. Then $\left(M_x(0),M_x(z)\right)$ is a centered Gaussian vector in $\mathcal{M}_{rn}(\R)^2$. Moreover, we have just seen that the variance matrix of this random vector is the same as that of $(X(\Norm{z}^2),Y(\Norm{z}^2))$. This concludes the proof.
\end{proof}

\begin{cor}
\label{cor same distribution}
Let $x \in M$ and $z \in T_xM \setminus \{0\}$, then we have:
\begin{equation*}
\esp{\odet{L_x(0)}\odet{L_x(z)}} = \esp{\odet{X(\Norm{z}^2)}\odet{Y(\Norm{z}^2)}}.
\end{equation*}
\end{cor}

\begin{proof}
With the same notations as in the proof of Lemma~\ref{lem same distribution} above, we have:
\begin{equation*}
\esp{\odet{M_x(0)}\odet{M_x(z)}} = \esp{\odet{X(\Norm{z}^2)}\odet{Y(\Norm{z}^2)}},
\end{equation*}
since $\left(M_x(0),M_x(z)\right)$ and $(X(\Norm{z}^2),Y(\Norm{z}^2))$ have the same distribution. Besides, 
\begin{align*}
\odet{L_x(0)}& =\odet{M_x(0)} & &\text{and} & \odet{L_x(z)}&=\odet{M_x(0)}.\qedhere
\end{align*}
\end{proof}

Let us now establish some facts about the distribution of $\left(X(t),Y(t)\right)$ for $t>0$.

\begin{lem}
\label{lem cond exp bounded}
For all $t>0$, we have:
\begin{equation*}
\esp{\odet{X(t)}\odet{Y(t)}} \leq n^r.
\end{equation*}
\end{lem}

\begin{proof}
First, by the Cauchy--Schwarz inequality, we have:
\begin{equation*}
\esp{\odet{X(t)}\odet{Y(t)}} \leq \esp{\odet{X(t)}^2}^\frac{1}{2} \esp{\odet{Y(t)}^2}^\frac{1}{2}.
\end{equation*}
Then, the definition of $\left(X(t),Y(t)\right)$ (Def.~\ref{def XYt}) shows that both $X(t)$ and $Y(t)$ are centered Gaussian vectors in $\mathcal{M}_{rn}(\R)$ with variance matrix:
\begin{equation}
\label{eq variance L}
\begin{pmatrix}
1 - \frac{t e^{-t}}{1-e^{-t}} & 0 \\ 0 & I_{n-1}
\end{pmatrix}\otimes I_r.
\end{equation}
in the canonical bases of $\R^n$ and $\R^r$. Here $I_r$ and $I_{n-1}$ stand for the identity matrices of size $r$ and $n-1$ respectively. Hence,
\begin{equation*}
\esp{\odet{X(t)}\odet{Y(t)}} \leq \esp{\odet{X(t)}^2} = \esp{\det\left(X(t)X(t)^\text{t}\right)}.
\end{equation*}

We denote by $X_1(t),\dots,X_r(t)$ the rows of $X(t)$. Then
\begin{equation*}
X(t)X(t)^\text{t} = \left(\prsc{X_i(t)}{X_j(t)}\right)_{1\leq i,j\leq r},
\end{equation*}
where we see $X_i(t)$ as an element of $\R^n$ and $\prsc{\cdot}{\cdot}$ is the usual inner product on $\R^n$. Hence, $\det\left(X(t)X(t)^\text{t}\right)$ is the Gram determinant of the family $\left(X_1(t),\dots,X_r(t)\right)$, which is known to be the square of the $r$-dimensional volume of the parallelepiped spanned by these vectors. In particular,
\begin{equation*}
\det\left(X(t)X(t)^\text{t}\right) \leq \Norm{X_1(t)}^2 \cdots \Norm{X_r(t)}^2.
\end{equation*}
By~\eqref{eq variance L}, the $X_i(t)$ are independent identically distributed centered Gaussian vectors with variance matrix:
\begin{equation*}
\begin{pmatrix}
1 - \frac{t e^{-t}}{1-e^{-t}} & 0 \\ 0 & I_{n-1}
\end{pmatrix},
\end{equation*}
so that:
\begin{equation*}
\det\left(X(t)X(t)^\text{t}\right)\leq \esp{\Norm{X_1(t)}^2 \cdots \Norm{X_r(t)}^2} \leq \esp{\Norm{X_1(t)}^2}^r = \left(n - \frac{t e^{-t}}{1-e^{-t}}\right)^r\leq n^r.\qedhere
\end{equation*}
\end{proof}

\begin{lem}
\label{lem integrability D infinity}
We have the following estimate as $t \to + \infty$:
\begin{equation*}
\esp{\odet{X\left(t\right)}\odet{Y\left(t\right)}} = (2\pi)^r \left(\frac{\vol{\S^{n-r}}}{\vol{\S^n}}\right)^2 + O\!\left(te^{-\frac{t}{2}}\right).
\end{equation*}
\end{lem}

\begin{proof}
Let $\left(X(\infty),Y(\infty)\right)$ be a standard Gaussian vector in $\mathcal{M}_{rn}(\R)^2 \simeq \R^{2nr}$, i.e.~$X(\infty)$ and $Y(\infty)$ are independent standard Gaussian vectors in $\mathcal{M}_{rn}(\R)$. Then,
\begin{align*}
\esp{\odet{X(\infty)}\odet{Y(\infty)}} &= \esp{\odet{X(\infty)}}\esp{\odet{Y(\infty)}}\\
&=\esp{\odet{X(\infty)}}^2\\
&= (2\pi)^r \left(\frac{\vol{\S^{n-r}}}{\vol{\S^n}}\right)^2,
\end{align*}
where we used~\eqref{eq value exp odet lim} to get the last equality.

Then the proof is basically the same as that of Lemma~\ref{lem estimates cond exp x}. From Definition~\ref{def XYt}, we see that the variance operator $\Lambda(t)$ of $\left(X(t),Y(t)\right)$ equals $\Id + O\!\left(te^{-\frac{t}{2}}\right)$ as $t\to +\infty$. Hence:
\begin{align*}
\det\left( \Lambda(t)\right) &= 1 + O\!\left(te^{-\frac{t}{2}}\right) & &\text{and} & \Lambda(t)^{-1} =& \Id + O\!\left(te^{-\frac{t}{2}}\right).
\end{align*}
Let $C>0$ be such that $\Norm{\Lambda(t)^{-1}- \Id} \leq C te^{-\frac{t}{2}}$. We denote by $L=\left(L_1,L_2\right)$ a generic element of $\mathcal{M}_{rn}(\R)^2$ and by $\dx L$ the normalized Lebesgue measure on this space. Then,
\begin{multline*}
\left(2\pi\right)^{rn} \norm{\det\left(\Lambda(t)\right)^\frac{1}{2} \esp{\odet{X(t)}\odet{Y(t)}} - \esp{\odet{X(\infty)}\odet{Y(\infty)}}}\\
\begin{aligned}
&\leq \int \odet{L_1}\odet{L_2} \norm{\exp\left(-\frac{1}{2}\prsc{\left(\Lambda(t)^{-1}-\Id\right)L}{L}\right)-1}e^{-\frac{1}{2}\Norm{L}^2} \dx L\\
&\leq \frac{C}{2}te^{-\frac{t}{2}} \int \odet{L_1}\odet{L_2}\Norm{L}^2\exp\left(-\frac{1}{2}\left(1-\frac{C}{2}te^{-\frac{t}{2}}\right)\Norm{L}^2\right) \dx L\\
&=  O\!\left(te^{-\frac{t}{2}}\right).
\end{aligned}
\end{multline*}
Thus
\begin{align*}
\esp{\odet{X(t)}\odet{Y(t)}}&\\
=\det&\left( \Lambda(t)\right)^{-\frac{1}{2}}\left(\esp{\odet{X(\infty)}\odet{Y(\infty)}}+ O\!\left(te^{-\frac{t}{2}}\right)\right)\\
= (2\pi&)^r \left(\frac{\vol{\S^{n-r}}}{\vol{\S^n}}\right)^2 + O\!\left(te^{-\frac{t}{2}}\right).\qedhere
\end{align*}
\end{proof}

\begin{dfn}
\label{def D}
Let $D_{n,r} : (0,+\infty) \to \R$ be the function defined by:
\begin{equation*}
\forall t \in (0,+\infty), \qquad D_{n,r}(t) = \frac{\esp{\odet{X(t)}\odet{Y(t)}}}{\left(1-e^{-t}\right)^\frac{r}{2}} - (2\pi)^r \left(\frac{\vol{\S^{n-r}}}{\vol{\S^n}}\right)^2.
\end{equation*}
\end{dfn}

\begin{lem}
\label{lem integrability Dnr}
We have:
\begin{equation*}
\int_0^{+\infty} \norm{D_{n,r}(t)}t^{\frac{n-2}{2}} \dx t < +\infty .
\end{equation*}
\end{lem}

\begin{proof}
We first check the integrability of $\norm{D_{n,r}(t)}t^{\frac{n-2}{2}}$ at $t=0$. By Lemma~\ref{lem cond exp bounded}, about $t=0$ we have:
\begin{align*}
\norm{D_{n,r}(t)} t^{\frac{n-2}{2}} &\leq t^{\frac{n-2}{2}}\frac{\esp{\odet{X(t)}\odet{Y(t)}}}{\left(1-e^{-t}\right)^\frac{r}{2}} + t^{\frac{n-2}{2}}(2\pi)^r \left(\frac{\vol{\S^{n-r}}}{\vol{\S^n}}\right)^2\\
&\leq t^{\frac{n-2}{2}}\frac{n^r}{\left(1-e^{-t}\right)^\frac{r}{2}} + O(t^\frac{n-2}{2}) = O\!\left(t^\frac{n-2-r}{2}\right).
\end{align*}
And this is integrable at $t=0$ since $n-r \geq 1$.

Then, by Lemma~\ref{lem integrability D infinity}, we have: $\norm{D_{n,r}(t)}t^{\frac{n-2}{2}} = O\!\left(t^\frac{n}{2}e^{-\frac{t}{2}}\right)$ when $t$ goes to infinity. This proves the integrability at infinity.
\end{proof}

\subsubsection{Near-diagonal asymptotics for the correlated terms}
\label{subsub near diag correlated}

The next step of the proof is to compute the contribution of the integral~\eqref{eq variance2} on $\Delta_d$.  Let $R >0$ be such that $2R$ is smaller than the injectivity radius of $\X$, as in Section~\ref{sec estimates for the bergman kernel}. Let $d_3 \in \N$ be such that $\forall d \geq d_3$, $b_n\frac{\ln d}{\sqrt{d}}\leq R$. In the sequel we consider $d \geq \max(d_0,d_1,d_2,d_3)$.

Since we chose $d$ large enough that $b_n \frac{\ln d}{\sqrt{d}} \leq R$ we can compute everything in the exponential chart about $x$. Let $\phi_1,\phi_2 \in \mathcal{C}^0(M)$, we have:
\begin{multline}
\label{eq double integral}
\int_{\Delta_d} \phi_1(x)\phi_2(y) \D_d(x,y) \rmes{M}^2\\
\begin{aligned}
&= \int_{x \in M} \left(\int_{y \in B_M\left(x,b_n \frac{\ln d}{\sqrt{d}}\right)} \phi_1(x)\phi_2(y) \D_d(x,y) \rmes{M}\right) \rmes{M}\\
&= \int_{x \in M} \left(\int_{z \in B_{T_xM}\left(0,b_n \frac{\ln d}{\sqrt{d}}\right)} \phi_1(x)\phi_2(\exp_x(z)) \D_d(x,\exp_x(z)) \sqrt{\kappa(z)}\dx z \right) \rmes{M},
\end{aligned}
\end{multline}
where $\sqrt{\kappa}$ is the density of $(\exp_x)^\star\rmes{M}$ with respect to the normalized Lebesgue measure on $T_xM$ (see Sect.~\ref{subsec near diagonal estimates}). Let $x \in M$, for all $z \in B_{T_xM}\left(0,b_n \ln d\right)$ we define
\begin{equation}
\label{def Dd xz}
D_d(x,z) = \D_d\left(x, \exp_x\left(\frac{z}{\sqrt{d}}\right)\right),
\end{equation}
where $\D_d$ is defined by~\eqref{eq def density}. Then, by a change of variable in~\eqref{eq double integral},
\begin{multline}
\label{eq double int rescaled}
\int_{\Delta_d} \phi_1(x)\phi_2(y) \D_d(x,y) \rmes{M}^2 =\\
d^{-\frac{n}{2}}\int_{x \in M} \left(\int_{z \in B_{T_xM}(0,b_n \ln d)} \phi_1(x)\phi_2\left(\exp_x\left(\frac{z}{\sqrt{d}}\right)\right) D_d(x,z) \left(\kappa\left(\frac{z}{\sqrt{d}}\right)\right)^\frac{1}{2}\dx z \right) \rmes{M},
\end{multline}
and we need to compute the asymptotic of $D_d(x,z)$ as $d$ goes to infinity. We start by computing $\odet{\ev_{x,y}^d}$ when $(x,y) \in \Delta_d$.

\begin{prop}
\label{prop ultim estimate evxy}
Let $\alpha \in \left(0,\frac{1}{2r+1}\right)$, let $x \in M$ and $z \in B_{T_xM}(0,b_n\ln d)$. We denote $y=\exp_x\left(\frac{z}{\sqrt{d}}\right)$. Then we have:
\begin{equation}
\label{eq prop ultim estimate evxy}
\left(\frac{\pi}{d}\right)^{2nr} \det\left(\ev^d_{x,y}\left(\ev^d_{x,y}\right)^*\right)  = \left(1-e^{-\Norm{z}^2}\right)^r \left(1+ O\!\left(d^{-\alpha}\right)\right),
\end{equation}
where the error term does not depend on $(x,z)$.
\end{prop}

We will deduce Proposition~\ref{prop ultim estimate evxy} from the following two lemmas.

\begin{lem}
\label{lem near diag estimate evxy}
Let $\beta \in (0,1)$ and $d \geq d_3$, then for every $x \in M$ and $z \in B_{T_xM}(0,b_n\ln d)$, we have:
\begin{equation*}
\left(\frac{\pi}{d}\right)^{2nr} \det\left(\ev^d_{x,y}\left(\ev^d_{x,y}\right)\right) = \left(1-e^{-\Norm{z}^2}\right)^r + O\!\left(d^{\beta-1}\right),
\end{equation*}
where $y$ stands for $\exp_x\left(\frac{z}{\sqrt{d}}\right)$. Moreover the error term depends on $\beta$ but not on $(x,z)$.
\end{lem}

\begin{lem}
\label{lem near diag estimate evxy integrable}
There exists $\tilde{C}>0$ such that, for all $\beta \in [0,1)$, there exists $d_\beta \in \N$ such that: $\forall d \geq d_\beta$, $\forall x \in M$, $\forall z \in B_{T_xM}\left(0,d^{\beta-1}\right)\setminus \{0\}$,
\begin{equation*}
\norm{\left(\frac{\pi}{d}\right)^{2nr} \det\left(\ev^d_{x,y}\left(\ev^d_{x,y}\right)^*\right) \left(1-e^{-\Norm{z}^2}\right)^{-r} - 1} \leq \tilde{C} d^{\beta-1},
\end{equation*}
where $y$ stands for $\exp_x\left(\frac{z}{\sqrt{d}}\right)$.
\end{lem}

Let us assume Lemmas~\ref{lem near diag estimate evxy} and~\ref{lem near diag estimate evxy integrable} for now, and prove Prop.~\ref{prop ultim estimate evxy}.

\begin{proof}[Proof of Proposition~\ref{prop ultim estimate evxy}]
First, note that if~\eqref{eq prop ultim estimate evxy} holds for $z \in B_{T_xM}(0,b_n\ln d)\setminus \{0\}$, then the same estimate holds for $z\in B_{T_xM}(0,b_n\ln d)$ since both sides of the equality vanish when $z=0$. In the sequel we assume that $z \neq 0$.

Let $\alpha \in \left(0,\frac{1}{2r+1}\right)$, let $d \geq d_3$ and let $x\in M$. Then for any $z \in T_xM$ such that $\Norm{z} \geq d^{-\alpha}$, we have:
\begin{equation}
\label{eq norm inverse}
\left(1-e^{-\Norm{z}^2}\right)^{-r} \leq \left(1-\exp\left(-d^{-2\alpha}\right)\right)^{-r}.
\end{equation}
Since $1-e^{-t} = t\left(1-\frac{t}{2}+ O(t^2)\right)$ as $t \to 0$, there exists $\tilde{C}_0$ such that for all $t \in (0,1)$,
\begin{equation}
\label{eq def C0 tilde}
\norm{\left(1-e^{-t}\right)^{-r} - t^{-r}} \leq \tilde{C}_0 t^{1-r}.
\end{equation}
Hence, by~\eqref{eq norm inverse}, for any $d \geq d_3$, for any $x \in M$ and any $z \in T_xM$ such that $\Norm{z} \geq d^{-\alpha}$, we have:
\begin{equation}
\label{eq norm inverse 2}
\left(1-e^{-\Norm{z}^2}\right)^{-r} \leq \left(d^{2r\alpha}+\tilde{C}_0d^{(2r-2)\alpha}\right) \leq d^{2r \alpha} \left(1+\tilde{C}_0\right).
\end{equation}

Let $\beta = 1 -(2r+1)\alpha$ and $\beta' = 1 - \alpha$, then $\beta$ and $\beta' \in (0,1)$. By Lemma~\ref{lem near diag estimate evxy}, there exists $\tilde{K}_\beta >0$ such that: for all $d \geq d_3$, $\forall x \in M$, $\forall z \in B_{T_xM}(0,b_n \ln d)$,
\begin{equation*}
\norm{\left(\frac{\pi}{d}\right)^{2nr} \det\left(\ev^d_{x,y}\left(\ev^d_{x,y}\right)\right) - \left(1-e^{-\Norm{z}^2}\right)^r} \leq \tilde{K}_\beta d^{\beta-1} = \tilde{K}_\beta d^{-(2r+1)\alpha},
\end{equation*}
where $y = \exp_x\left(\frac{z}{\sqrt{d}}\right)$. Then, by~\eqref{eq norm inverse 2}, we have: $\forall d \geq d_3$, $\forall x \in M$, $\forall z \in B_{T_xM}(0,b_n \ln d)$ such that $\Norm{z}\geq d^{-\alpha} = d^{\beta'-1}$,
\begin{equation*}
\norm{\left(\frac{\pi}{d}\right)^{2nr} \det\left(\ev^d_{x,y}\left(\ev^d_{x,y}\right)\right)\left(1-e^{-\Norm{z}^2}\right)^{-r}-1} \leq \tilde{K}_\beta d^{-\alpha}\left(1+\tilde{C}_0\right),
\end{equation*}

Besides, let $d \geq d_{\beta'}$ and $x \in M$, then for all $z \in B_{T_xM}\left(0,d^{-\alpha}\right)\setminus \{0\}$ we have:
\begin{equation*}
\norm{\left(\frac{\pi}{d}\right)^{2nr} \det\left(\ev^d_{x,y}\left(\ev^d_{x,y}\right)\right)\left(1-e^{-\Norm{z}^2}\right)^{-r}-1} \leq \tilde{C}d^{-\alpha},
\end{equation*}
by Lemma~\ref{lem near diag estimate evxy integrable}. Finally, for all $d \geq \max(d_{\beta'},d_3)$, $\forall x \in M$, $\forall z \in B_{T_xM}\left(0,b_n \ln d\right)\setminus \{0\}$, we have:
\begin{equation*}
\norm{\left(\frac{\pi}{d}\right)^{2nr} \det\left(\ev^d_{x,y}\left(\ev^d_{x,y}\right)\right)\left(1-e^{-\Norm{z}^2}\right)^{-r}-1} \leq d^{-\alpha} \max\left(\tilde{C},2\tilde{K}_\beta \left(1+\tilde{C}_0\right)\right).\qedhere
\end{equation*}
\end{proof}

\begin{proof}[Proof of Lemma~\ref{lem near diag estimate evxy}]
Let $d \geq d_3$, let $x \in M$ and let $z \in B_{T_xM}(0,b_n \ln d)$. We denote $y=\exp_x\left(\frac{z}{\sqrt{d}}\right)$. Since  $\frac{\Norm{z}}{\sqrt{d}}< R$, let us write eq.~\eqref{eq variance evxy} in the real normal trivialization of $\E \otimes \L^d$ about~$x$ (see Sect.~\ref{subsec real normal trivialization}). We have:
\begin{equation*}
\left(\frac{\pi}{d}\right)^n \ev_{x,y}^d\left(\ev_{x,y}^d\right)^* = \left(\frac{\pi}{d}\right)^n\begin{pmatrix}
E_d(0,0) & E_d\left(0,\frac{z}{\sqrt{d}}\right) \\ E_d\left(\frac{z}{\sqrt{d}},0\right) & E_d\left(\frac{z}{\sqrt{d}},\frac{z}{\sqrt{d}}\right)
\end{pmatrix}.
\end{equation*}
Then, by the near-diagonal estimates of Cor.~\ref{cor near diag estimates}, we have:
\begin{multline*}
\left(\frac{\pi}{d}\right)^n \begin{pmatrix}
E_d(0,0) & E_d\left(0,\frac{z}{\sqrt{d}}\right) \\ E_d\left(\frac{z}{\sqrt{d}},0\right) & E_d\left(\frac{z}{\sqrt{d}},\frac{z}{\sqrt{d}}\right)
\end{pmatrix} =\\
\begin{pmatrix}
\Id_{(\E\otimes\L^d)_x} & e^{-\frac{1}{2}\Norm{z}^2}\left(\kappa\left(\frac{z}{\sqrt{d}}\right)\right)^{-\frac{1}{2}}\Id_{(\E\otimes\L^d)_x} \\
e^{-\frac{1}{2}\Norm{z}^2}\left(\kappa\left(\frac{z}{\sqrt{d}}\right)\right)^{-\frac{1}{2}}\Id_{(\E\otimes\L^d)_x} & \left(\kappa\left(\frac{z}{\sqrt{d}}\right)\right)^{-1}\Id_{(\E\otimes\L^d)_x}
\end{pmatrix} + O\left(\frac{(\ln d)^{2n+8}}{d}\right),
\end{multline*}
where the error term does not depend on $(x,z)$. Recall that $\kappa$ satisfies~\eqref{eq estimate kappa}. Hence for all $z \in B(0,b_n \ln d)$,
\begin{equation*}
\kappa\left(\frac{z}{\sqrt{d}}\right) = 1 + O\left(\frac{(\ln d)^2}{d}\right),
\end{equation*}
uniformly in $x$ and $z$. Let $\beta \in (0,1)$, then we have:
\begin{equation}
\label{eq odet near diag}
\left(\frac{\pi}{d}\right)^n \ev_{x,y}^d\left(\ev_{x,y}^d\right)^* = \begin{pmatrix}
1 & e^{-\frac{1}{2}\Norm{z}^2} \\ e^{-\frac{1}{2}\Norm{z}^2} & 1
\end{pmatrix} \otimes \Id_{(\E\otimes\L^d)_x} + O\left(d^{\beta-1}\right),
\end{equation}
and the constant in the term $O\!\left(d^{\beta-1}\right)$ does not depend on $(x,z)$. Since the dominant term on the right-hand side of~\eqref{eq odet near diag} has bounded coefficients, we get the result by taking the determinant of~\eqref{eq odet near diag}.
\end{proof}

\begin{proof}[Proof of Lemma~\ref{lem near diag estimate evxy integrable}]
Let $d\geq \max(d_0,d_3)$ and let $x \in M$. Recall that $D^k_{(z,w)}$ denotes the $k$-th differential at $(z,w)$ of a map from $T_x\X \times T_x\X$ to $\End\left(\left(\E \otimes \L^d\right)_x\right)$.

The Chern connection reads $D + \mu^d_x$ in the real normal trivialization about $x$, where $\mu^d_x$ is a $1$-form on $B_{T_x\X}(0,2R)$. By definition of the real normal trivialization, $\mu^d_x(0)=0$. Besides $\mu^d_x(z)$ is a smooth function of $(x,z)$ and grows at most linearly in $d$. By compactness of $M$, there exist $A$ and $B>0$ such that $\Norm{\mu^d_x\left(\frac{z}{\sqrt{d}}\right)}\leq A + B\sqrt{d}$ for all $x \in M$ and all $z \in B_{T_x\X}(0,R)$. Hence, there exists $K_1 >0$ independent of $x$ such that, for any  smooth section $S$ of $\R\left(\E \otimes \L^d\right) \boxtimes \R\left(\E \otimes \L^d\right)^*$ over $B_{T_xM}(0,R)\times B_{T_xM}(0,R)$, we have:
\begin{equation*}
\Norm{D_{(z,w)}\left(S\left(\frac{z}{\sqrt{d}},\frac{w}{\sqrt{d}}\right)\right)} \leq K_1 \Norm{S\left(\exp_x\left(\frac{z}{\sqrt{d}}\right),\exp_x\left(\frac{w}{\sqrt{d}}\right)\right)}_{\mathcal{C}^1},
\end{equation*}
where $\Norm{\cdot}_{\mathcal{C}^1}$ was defined in Section~\ref{subsec far off diagonal estimates}.  Since we use the exponential chart, we can argue similarly for the Levi--Civita connection. This gives a similar result for the higher derivatives of $S$. For all $k \in \N$, there exists $K_k>0$ independent of $x$ such that, for any smooth section $S$ of $\R\left(\E \otimes \L^d\right) \boxtimes \R\left(\E \otimes \L^d\right)^*$ over $B_{T_xM}(0,R)\times B_{T_xM}(0,R)$, we have:
\begin{equation}
\label{eq def Kk}
\Norm{D^k_{(z,w)}\left(S\left(\frac{z}{\sqrt{d}},\frac{w}{\sqrt{d}}\right)\right)} \leq K_k \Norm{S\left(\exp_x\left(\frac{z}{\sqrt{d}}\right),\exp_x\left(\frac{w}{\sqrt{d}}\right)\right)}_{\mathcal{C}^k}.
\end{equation}
Since $d \geq d_0$, by eq.~\eqref{eq def Kk} and Thm.~\ref{thm off diag estimates} we have: $
\forall z,w \in B_{T_xM}(0,R)$,
\begin{equation}
\label{eq estimate D2Ed}
\Norm{D^2_{(z,w)} \left(E_d\left(\frac{z}{\sqrt{d}},\frac{w}{\sqrt{d}}\right)\right)} \leq K_2 \Norm{E_d\left(\exp_x\left(\frac{z}{\sqrt{d}}\right),\exp_x\left(\frac{w}{\sqrt{d}}\right)\right)}_{\mathcal{C}^2} \leq C_2 K_2 d^{n+1}.
\end{equation}

Let $x \in M$ and $z \in B_{T_xM}(0,b_n \ln d)\setminus\{0\}$. We denote $y=\exp_x\left(\frac{z}{\sqrt{d}}\right)$. Let us write eq.~\eqref{eq variance evxy}, in the real normal trivialization of $\E \otimes \L^d$ about~$x$, as in the proof of Lem.~\ref{lem near diag estimate evxy}. We have:
\begin{equation*}
\ev_{x,y}^d\left(\ev_{x,y}^d\right)^* = \begin{pmatrix}
E_d(0,0) & E_d\left(0,\frac{z}{\sqrt{d}}\right) \\ E_d\left(\frac{z}{\sqrt{d}},0\right) & E_d\left(\frac{z}{\sqrt{d}},\frac{z}{\sqrt{d}}\right)
\end{pmatrix}.
\end{equation*}
Then, by elementary operations on rows and columns,
\begin{multline}
\label{eq elementery operations}
\frac{1}{\Norm{z}^{2r}}\det\left(\ev^d_{x,y}\left(\ev^d_{x,y}\right)^*\right) = \frac{1}{\Norm{z}^{2r}}\det \begin{pmatrix}
E_d(0,0) & E_d\left(0,\frac{z}{\sqrt{d}}\right) \\ E_d\left(\frac{z}{\sqrt{d}},0\right) & E_d\left(\frac{z}{\sqrt{d}},\frac{z}{\sqrt{d}}\right)
\end{pmatrix}\\
=\det \begin{pmatrix}
E_d(0,0) & \frac{1}{\Norm{z}}\left(
E_d\left(0,\frac{z}{\sqrt{d}}\right)-E_d(0,0)\right) \\ \frac{1}{\Norm{z}}\left(
E_d\left(\frac{z}{\sqrt{d}},0\right)-E_d(0,0)\right) & \frac{1}{\Norm{z}^2}\left(\begin{aligned}
E_d\left(\frac{z}{\sqrt{d}},\frac{z}{\sqrt{d}}\right)&-E_d\left(\frac{z}{\sqrt{d}},0\right)\\-E_d&\left(0,\frac{z}{\sqrt{d}}\right)+E_d(0,0)
\end{aligned}\right)
\end{pmatrix}.
\end{multline}

By Taylor's formula, for all $z \in B_{T_xM}(0,b_n \ln d)\setminus \{0\}$ we have:
\begin{multline}
\label{eq Taylor 1}
\Norm{E_d\left(0,\frac{z}{\sqrt{d}}\right)-E_d(0,0)-D_{(0,0)}E_d\cdot \left(0,\frac{z}{\sqrt{d}}\right)}\leq\\
 \frac{\Norm{z}^2}{2d} \left(\sup_{w \in [0,z]}\Norm{D^2_{(0,w)} E_d\left(\exp_x\left(\frac{\cdot}{\sqrt{d}}\right),\exp_x\left(\frac{\cdot}{\sqrt{d}}\right)\right)}\right).
\end{multline}
Then, by~\eqref{eq estimate D2Ed}, we have:
\begin{equation}
\label{eq Taylor 2}
\left(\frac{\pi}{d}\right)^n \frac{1}{\Norm{z}}\Norm{E_d\left(0,\frac{z}{\sqrt{d}}\right)-E_d(0,0)-D_{(0,0)}E_d\cdot \left(0,\frac{z}{\sqrt{d}}\right)} \leq \Norm{z} C_2 K_2 \pi^n.
\end{equation}
Similarly, for all $z \in B_{T_xM}(0,b_n \ln d)\setminus \{0\}$ we have:
\begin{equation}
\label{eq Taylor 3}
\left(\frac{\pi}{d}\right)^n \frac{1}{\Norm{z}}\Norm{E_d\left(\frac{z}{\sqrt{d}},0\right)-E_d(0,0)-D_{(0,0)}E_d\cdot \left(\frac{z}{\sqrt{d}},0\right)} \leq \Norm{z} C_2 K_2 \pi^n.
\end{equation}
A second order Taylor's formula gives:
\begin{multline*}
\Norm{\left(E_d(\frac{z}{\sqrt{d}},\frac{z}{\sqrt{d}})-E_d(\frac{z}{\sqrt{d}},0)-E_d(0,\frac{z}{\sqrt{d}})+E_d(0,0)\right) - D^2_{(0,0)}E_d \left((0,\frac{z}{\sqrt{d}}),(\frac{z}{\sqrt{d}},0)\right)}\\
\leq \left(\frac{\Norm{z}}{\sqrt{d}}\right)^3 \left(\sup_{w \in [0,z]}\Norm{D^3_{(0,w)} E_d\left(\exp_x\left(\frac{\cdot}{\sqrt{d}}\right),\exp_x\left(\frac{\cdot}{\sqrt{d}}\right)\right)}\right),
\end{multline*}
and since $d \geq d_0$, by Thm.~\ref{thm off diag estimates} and eq.~\eqref{eq def Kk} we have:
\begin{multline}
\label{eq Taylor 4}
\left(\frac{\pi}{d}\right)^n \frac{1}{\Norm{z}^2}\Norm{\left(E_d\left(\frac{z}{\sqrt{d}},\frac{z}{\sqrt{d}}\right)-E_d\left(\frac{z}{\sqrt{d}},0\right)-E_d\left(0,\frac{z}{\sqrt{d}}\right)+E_d(0,0)\right) -\right.\\
\left.D^2_{(0,0)}E_d \left(\left(0,\frac{z}{\sqrt{d}}\right)\left(\frac{z}{\sqrt{d}},0\right)\right)}\leq \Norm{z} C_3 K_3 \pi^n.
\end{multline}
Finally, by Equations~\eqref{eq Taylor 2}, \eqref{eq Taylor 3} and~\eqref{eq Taylor 4},
\begin{equation}
\label{matrix distance 1}
\begin{aligned}
\left(\frac{\pi}{d}\right)^n & \begin{pmatrix}
E_d(0,0) & \frac{1}{\Norm{z}}\left(
E_d\left(0,\frac{z}{\sqrt{d}}\right)-E_d(0,0)\right) \\ \frac{1}{\Norm{z}}\left(
E_d\left(\frac{z}{\sqrt{d}},0\right)-E_d(0,0)\right) & \frac{1}{\Norm{z}^2}\left(\begin{aligned}
E_d\left(\frac{z}{\sqrt{d}},\frac{z}{\sqrt{d}}\right)&-E_d\left(\frac{z}{\sqrt{d}},0\right)\\-E_d&\left(0,\frac{z}{\sqrt{d}}\right)+E_d(0,0)
\end{aligned}\right)
\end{pmatrix}\\
&=\left(\frac{\pi}{d}\right)^n \begin{pmatrix}
E_d(0,0) & \frac{1}{\Norm{z}}D_{(0,0)}E_d\cdot \left(0,\frac{z}{\sqrt{d}}\right) \\ \frac{1}{\Norm{z}}D_{(0,0)}E_d\cdot \left(\frac{z}{\sqrt{d}},0\right) & \frac{1}{\Norm{z}^2} D^2_{(0,0)}E_d \left(\left(0,\frac{z}{\sqrt{d}}\right)\left(\frac{z}{\sqrt{d}},0\right)\right)
\end{pmatrix} + O\!\left(\Norm{z}\right),
\end{aligned}
\end{equation}
where the error term is uniform in $x$ and $d$.

On the other hand, for every $x \in M$ and every $z \in T_xM\setminus \{0\}$, the diagonal estimates of Sect.~\ref{subsec diagonal estimates} give (see~\eqref{eq diag values in chart}):
\begin{align*}
\left(\frac{\pi}{d}\right)^n \frac{1}{\Norm{z}^2} D^2_{(0,0)}E_d \left(\left(0,\frac{z}{\sqrt{d}}\right)\left(\frac{z}{\sqrt{d}},0\right)\right) &= \frac{\pi^n}{d^{n+1}} D^2_{(0,0)}E_d \left(\left(0,\frac{z}{\Norm{z}}\right)\left(\frac{z}{\Norm{z}},0\right)\right)\\
&= \Id_{\R(\E\otimes\L^d)_x} + O(d^{-1}),
\end{align*}
where the error term is independent of $x$ and $z$. Similarly,
\begin{align*}
\left(\frac{\pi}{d}\right)^n \frac{1}{\Norm{z}}D_{(0,0)}E_d\cdot \left(0,\frac{z}{\sqrt{d}}\right) &=\left(\frac{\pi}{d}\right)^n \frac{1}{\sqrt{d}}D_{(0,0)}E_d\cdot \left(0,\frac{z}{\Norm{z}}\right) = O(d^{-1}),\\
\left(\frac{\pi}{d}\right)^n \frac{1}{\Norm{z}}D_{(0,0)}E_d\cdot \left(\frac{z}{\sqrt{d}},0\right) &=\left(\frac{\pi}{d}\right)^n \frac{1}{\sqrt{d}}D_{(0,0)}E_d\cdot \left(\frac{z}{\Norm{z}},0\right) = O(d^{-1}),\\
\intertext{and}
\left(\frac{\pi}{d}\right)^n E_d(0,0) &= \Id_{\R(\E\otimes\L^d)_x} + O(d^{-1}).
\end{align*}
Thus
\begin{equation}
\label{matrix distance 2}
\left(\frac{\pi}{d}\right)^n \begin{pmatrix}
E_d(0,0) & \frac{1}{\Norm{z}}D_{(0,0)}E_d\cdot \left(0,\frac{z}{\sqrt{d}}\right) \\ \frac{1}{\Norm{z}}D_{(0,0)}E_d\cdot \left(\frac{z}{\sqrt{d}},0\right) & \frac{1}{\Norm{z}^2} D^2_{(0,0)}E_d \left(\left(0,\frac{z}{\sqrt{d}}\right)\left(\frac{z}{\sqrt{d}},0\right)\right)
\end{pmatrix} = \Id + O(d^{-1}),
\end{equation}
where the error term is uniform in $(x,z)$. By~\eqref{matrix distance 1} and~\eqref{matrix distance 2}, there exist $\tilde{C}_1$ and $\tilde{C}_2 >0$ such that we have: $\forall d \geq \max(d_0,d_3)$, $\forall x \in M$, $\forall z \in B_{T_xM}(0,b_n \ln d) \setminus \{0\}$,
\begin{multline}
\label{matrix distance 3}
\Norm{\left(\frac{\pi}{d}\right)^n \begin{pmatrix}
E_d(0,0) & \frac{1}{\Norm{z}}\left(
E_d\left(0,\frac{z}{\sqrt{d}}\right)-E_d(0,0)\right) \\ \frac{1}{\Norm{z}}\left(
E_d\left(\frac{z}{\sqrt{d}},0\right)-E_d(0,0)\right) & \frac{1}{\Norm{z}^2}\left(\begin{aligned}
E_d\left(\frac{z}{\sqrt{d}},\frac{z}{\sqrt{d}}\right)-E_d\left(\frac{z}{\sqrt{d}},0\right)&\\-E_d\left(0,\frac{z}{\sqrt{d}}\right)+E_d(0,0)&
\end{aligned}\right)
\end{pmatrix}-\Id\ }\\
\leq \tilde{C}_1 \Norm{z} + \tilde{C}_2 \frac{1}{d}.
\end{multline}

Let $\beta \in [0,1)$, then for all $d \geq \max(d_0,d_3)$, for all $x\in M$ and all $z \in B_{T_xM}\left(0,d^{\beta-1}\right)$, we have: $\tilde{C}_1 \Norm{z} + \tilde{C}_2 d^{-1} \leq d^{\beta-1}\left(\tilde{C}_1 + \tilde{C}_2\right)$. Let $d_\beta \in \N$ be such that $(d_\beta)^{\beta-1}\left(\tilde{C}_1 + \tilde{C}_2\right) \leq \frac{1}{2}$. Since the determinant is a smooth function, there exists $\tilde{C}_3>0$ such that, for every operator $\Lambda$, if $\Norm{\Lambda} \leq \frac{1}{2}$, then $\norm{\det\left(\Id + \Lambda\right) - 1} \leq \tilde{C}_3 \Norm{\Lambda}$. Hence, by eq.~\eqref{eq elementery operations} and~\eqref{matrix distance 3}, we have: for all $d\geq d_\beta$ , for all $x \in M$, for all $z \in B_{T_xM}\left(0,d^{\beta-1}\right)\setminus\{0\}$,
\begin{equation}
\norm{\frac{1}{\Norm{z}^{2r}}\left(\frac{\pi}{d}\right)^{2rn}\det\left(\ev^d_{x,y}\left(\ev^d_{x,y}\right)^*\right) - 1} \leq \left(\tilde{C}_1+\tilde{C}_2\right) \tilde{C}_3 d^{\beta-1}.
\end{equation}

Recall that $\tilde{C}_0$ was defined in the proof of Prop.~\ref{prop ultim estimate evxy} (see eq.~\eqref{eq def C0 tilde}) and that, for all $x \in M$, for all $z \in B_{T_xM}(0,1)\setminus \{0\}$, we have:
\begin{equation*}
\norm{\frac{\Norm{z}^{2r}}{\left(1-e^{-\Norm{z}^2}\right)^{r}} -1} \leq \tilde{C}_0 \Norm{z}^2.
\end{equation*}
Then we have: $\forall d \geq d_\beta$, $\forall x \in M$, $\forall z \in B_{T_xM}\left(0,d^{\beta-1}\right)\setminus\{0\}$,
\begin{multline*}
\norm{\left(\frac{\pi}{d}\right)^{2nr} \det\left(\ev^d_{x,y}\left(\ev^d_{x,y}\right)^*\right) \left(1-e^{-\Norm{z}^2}\right)^{-r} - 1}\\
\begin{aligned}
&= \norm{\left(\frac{\pi}{d}\right)^{2nr} \frac{1}{\Norm{z}^{2r}}\det\left(\ev^d_{x,y}\left(\ev^d_{x,y}\right)^*\right) \frac{\Norm{z}^{2r}}{\left(1-e^{-\Norm{z}^2}\right)^{r}} - 1}\\
&\leq \frac{\Norm{z}^{2r}}{\left(1-e^{-\Norm{z}^2}\right)^{r}} \norm{\frac{1}{\Norm{z}^{2r}}\left(\frac{\pi}{d}\right)^{2rn}\det\left(\ev^d_{x,y}\left(\ev^d_{x,y}\right)^*\right) - 1} + \norm{\frac{\Norm{z}^{2r}}{\left(1-e^{-\Norm{z}^2}\right)^{r}} -1}\\
&\leq \left(1 + \tilde{C}_0 d^{2\beta-2}\right)\left(\tilde{C}_1+\tilde{C}_2\right) \tilde{C}_3 d^{\beta-1} + \tilde{C}_0 d^{2\beta-2}\\
&\leq d^{\beta-1} \left(\left(\tilde{C}_1+\tilde{C}_2\right) \tilde{C}_3\left(1+\tilde{C}_0\right) + \tilde{C}_0\right) = d^{\beta-1} \tilde{C},
\end{aligned}
\end{multline*}
where we define $\tilde{C}>0$ by the equality on the last line.
\end{proof}

We now want to compute the limit of the conditional distribution of $\frac{\pi^n}{d^{n+1}}\left(\nabla^d_x s_d, \nabla^d_y s_d\right)$ given that $s_d(x) = 0 = s_d(y)$ for $(x,y) \in \Delta_d$. It is enough to compute the limit of $\Lambda_d(x,y)$ as $d \to +\infty$. Recall that $\Lambda_d$ is defined by Def.~\ref{def Lambda dxy}. Since we work near the diagonal, we can write everything in the real normal trivialization centered at $x$ (see Sect.~\ref{subsec real normal trivialization}).

\begin{lem}
\label{lem near diagonal Lambda dxy 1}
Let $x \in M$ and let $\nabla^d$ be a real metric connection which is trivial over $B_{T_xM}(0,R)$ in the real normal trivialization about $x$. Let $\beta \in (0,1)$, then, in the real normal trivialization about $x$, we have: $\forall z \in B_{T_xM}(0,b_n \ln d)$,
\begin{multline*}
\frac{\pi^n}{d^{n+1}}\begin{pmatrix}
\partial_x\partial_y^\sharp E_d(0,0) & \partial_x\partial_y^\sharp E_d\left(0,\frac{z}{\sqrt{d}}\right)\\
\partial_x\partial_y^\sharp E_d\left(\frac{z}{\sqrt{d}},0\right) & \partial_x\partial_y^\sharp E_d\left(\frac{z}{\sqrt{d}},\frac{z}{\sqrt{d}}\right)
\end{pmatrix} =\\
\begin{pmatrix}
\Id_{T^*_xM} & e^{-\frac{1}{2}\Norm{z}^2}\left(\Id_{T^*_xM}-z^* \otimes z\right) \\ e^{-\frac{1}{2}\Norm{z}^2}\left(\Id_{T^*_xM}-z^* \otimes z\right) & \Id_{T^*_xM}
\end{pmatrix} \otimes \Id_{\R\left(\E \otimes \L^d \right)_x} + O\!\left(d^{\beta-1}\right),
\end{multline*}
where the error term does not depend on $(x,z)$.
\end{lem}

\begin{proof}
Let $x \in M$ and let us choose an orthonormal basis $\left(\deron{}{x_1},\dots,\deron{}{x_n}\right)$ of $T_xM$. We denote the corresponding coordinates on $T_xM \times T_xM$ by $(z_1,\dots,z_n,w_1,\dots,w_n)$ and by $\partial_{z_i}$ and $\partial_{w_j}$ the associated partial derivatives. Let $\left(dx_1,\dots,dx_n\right)$ denote the dual basis of $\left(\deron{}{x_1},\dots,\deron{}{x_n}\right)$. By definition of $\nabla^d$ and $\partial_x\partial_y^\sharp E_d$ (see eq.~\eqref{eq def sharp 2}), for all $z,w \in B_{T_xM}(0,R)$, the matrix of $\partial_x\partial_y^\sharp E_d(z,w)$ in the orthonormal basis $(dx_1,\dots,dx_n)$ is:
\begin{equation*}
\begin{pmatrix}
\partial_{z_i}\partial_{w_j}E_d(z,w)
\end{pmatrix}_{1\leq i,j \leq n}.
\end{equation*}
Note that this is a matrix with values in $\End\left(\R\left(\E \otimes \L^d \right)_x\right)$. Recall that we defined the function $\xi_d$ by~\eqref{eq def xi d}. Then, by Cor.~\ref{cor near diag estimates}, for all $z,w \in B_{T_xM}(0,b_n \ln d)$, we have:
\begin{equation*}
\partial_{z_i}\partial_{w_j}E_d\left(\frac{z}{\sqrt{d}},\frac{w}{\sqrt{d}}\right) = \left(\frac{d}{\pi}\right)^n \partial_{z_i}\partial_{w_j}\xi_d\left(\frac{z}{\sqrt{d}},\frac{w}{\sqrt{d}}\right) \Id_{\R\left(\E \otimes \L^d \right)_x} + O\!\left((\ln d)^{2n+8}\right).
\end{equation*}
Then, eq.~\eqref{eq derivative xi d zw} shows that:
\begin{align*}
\partial_{z_i}\partial_{w_j}&\xi_d\left(\frac{z}{\sqrt{d}},\frac{w}{\sqrt{d}}\right) = \exp\left(-\frac{1}{2}\Norm{z-w}^2\right)\kappa\left(\frac{z}{\sqrt{d}}\right)^{-\frac{1}{2}}\kappa\left(\frac{w}{\sqrt{d}}\right)^{-\frac{1}{2}} \times\\
&\left(d\delta_{ij}-d(z_i-w_i)(z_j-w_j)-\frac{\sqrt{d}(z_j-w_j)\partial_{z_i}\kappa\left(\frac{z}{\sqrt{d}}\right)}{2\kappa\left(\frac{z}{\sqrt{d}}\right)}+\frac{\sqrt{d}(z_i-w_i)\partial_{w_j}\kappa\left(\frac{w}{\sqrt{d}}\right)}{2\kappa\left(\frac{w}{\sqrt{d}}\right)}\right)\\
&= d\exp\left(-\frac{1}{2}\Norm{z-w}^2\right)\left(\delta_{ij}-(z_i-w_i)(z_j-w_j)\right)+ O\!\left((\ln d)^4\right),
\end{align*}
where we used the fact that, uniformly in $z \in B_{T_xM}(0,b_n \ln d)$, we have:
\begin{align*}
&\kappa\left(\frac{z}{\sqrt{d}}\right) = 1 + O\!\left(\frac{(\ln d)^2}{d}\right)\\
\text{and} \ \forall i \in \{1,\dots,n\}, \qquad &\partial_{z_i}\kappa\left(\frac{z}{\sqrt{d}}\right) = O\!\left(\frac{\ln d}{\sqrt{d}}\right).
\end{align*}
Hence, for all $z,w \in B_{T_xM}(0,b_n \ln d)$, we have:
\begin{multline*}
\frac{\pi^n}{d^{n+1}}\partial_{z_i}\partial_{w_j}E_d\left(\frac{z}{\sqrt{d}},\frac{w}{\sqrt{d}}\right) =\\
\exp\left(-\frac{1}{2}\Norm{z-w}^2\right)\left(\delta_{ij}-(z_i-w_i)(z_j-w_j)\right) \Id_{\R\left(\E \otimes \L^d \right)_x} + O\!\left(\frac{(\ln d)^{2n+8}}{d}\right),
\end{multline*}
where the error term is independent of $x,z$ and $w$. Furthermore, for any $\beta \in (0,1)$, the term $O\!\left(\frac{(\ln d)^{2n+8}}{d}\right)$ can be replaced by $O\left(d^{\beta-1}\right)$. Finally, for all $z,w \in B_{T_xM}(0,b_n \ln d)$, we have:
\begin{multline*}
\frac{\pi^n}{d^{n+1}}\partial_x\partial_y^\sharp E_d(z,w) =\\
\exp\left(-\frac{1}{2}\Norm{z-w}^2\right)\left(\Id_{T^*_xM} - (z-w)^* \otimes (z-w)\right)\otimes \Id_{\R\left(\E \otimes \L^d \right)_x}+ O\left(d^{\beta-1}\right),
\end{multline*}
which yields the result.
\end{proof}

A similar proof, using Cor.~\ref{cor near diag estimates} and the expressions~\eqref{eq derivative xi d z} and~\eqref{eq derivative xi d w} for the partial derivatives of $\xi_d$ yields the following.

\begin{lem}
\label{lem near diagonal Lambda dxy 2}
Let $x \in M$ and let $\nabla^d$ be a real metric connection which is trivial over $B_{T_xM}(0,R)$ in the real normal trivialization about $x$. Let $\beta \in (0,1)$, then, in the real normal trivialization about $x$, we have: $\forall z \in B_{T_xM}(0,b_n \ln d)$,
\begin{align*}
\frac{\pi^n}{d^{n+\frac{1}{2}}}\begin{pmatrix}
\partial_x E_d(0,0) & \partial_x E_d\left(0,\frac{z}{\sqrt{d}}\right)\\
\partial_x E_d\left(\frac{z}{\sqrt{d}},0\right) & \partial_x E_d\left(\frac{z}{\sqrt{d}},\frac{z}{\sqrt{d}}\right)
\end{pmatrix} &=
e^{-\frac{1}{2}\Norm{z}^2}\begin{pmatrix}
0 &  z^* \\ - z^* & 0
\end{pmatrix} \otimes \Id_{\R\left(\E \otimes \L^d \right)_x} + O\!\left(d^{\beta-1}\right)\\
\frac{\pi^n}{d^{n+\frac{1}{2}}}\begin{pmatrix}
\partial_y^\sharp E_d(0,0) & \partial_y^\sharp E_d\left(0,\frac{z}{\sqrt{d}}\right)\\
\partial_y^\sharp E_d\left(\frac{z}{\sqrt{d}},0\right) & \partial_y^\sharp E_d\left(\frac{z}{\sqrt{d}},\frac{z}{\sqrt{d}}\right)
\end{pmatrix} &=
e^{-\frac{1}{2}\Norm{z}^2}\begin{pmatrix}
0 &  -z \\ z & 0
\end{pmatrix} \otimes \Id_{\R\left(\E \otimes \L^d \right)_x} + O\!\left(d^{\beta-1}\right),
\end{align*}
where $z^* \in T^*_xM$ is to be understood as the constant map $t \mapsto z^*$ from $\R$ to $T^*_xM$ and $z \in T_xM$ is to be understood as the evaluation on $z$ from $T^*_xM$ to $\R$. Moreover, the error terms do not depend on $(x,z)$.
\end{lem}

We would like to get a similar asymptotic for the last term in the conditional variance operator~\eqref{eq conditional variance full}, namely:
\begin{equation*}
\begin{pmatrix}
E_d(0,0) & E_d\left(0,\frac{z}{\sqrt{d}}\right)\\
E_d\left(\frac{z}{\sqrt{d}},0\right) & E_d\left(\frac{z}{\sqrt{d}},\frac{z}{\sqrt{d}}\right)
\end{pmatrix}^{-1}.
\end{equation*}
Unfortunately, this term is singular on $\Delta$, and this kills all hope to get a uniform estimate on $B_{T_xM}(0,b_n\ln d)\setminus \{0\}$. Instead, we obtain a uniform estimate on $B_{T_xM}(0,b_n\ln d)\setminus B_{T_xM}(0,\rho)$ for some $\rho>0$. We need to carefully check how this estimate depends on $\rho$.

\begin{lem}
\label{lem near diag ev inverse}
Let $\beta \in (0,1)$ and $\rho \in (0,1)$. Let $x \in M$ and $z \in B_{T_xM}(0,b_n \ln d)$ such that $\Norm{z} \geq \rho$. Then, in the real normal trivialization about $x$, we have:
\begin{multline*}
\left(\frac{d}{\pi}\right)^{n} \begin{pmatrix}
E_d(0,0) & E_d\left(0,\frac{z}{\sqrt{d}}\right)\\
E_d\left(\frac{z}{\sqrt{d}},0\right) & E_d\left(\frac{z}{\sqrt{d}},\frac{z}{\sqrt{d}}\right)
\end{pmatrix}^{-1} =\\
\frac{1}{1-e^{-\Norm{z}^2}} \begin{pmatrix}
1 & -e^{-\frac{1}{2}\Norm{z}^2} \\ -e^{-\frac{1}{2}\Norm{z}^2} & 1
\end{pmatrix} \otimes \Id_{\R\left(\E \otimes \L^d \right)_x}\left(\Id + O\!\left(\frac{d^{\beta-1}}{1-e^{-\frac{1}{2}\rho^2}}\right)\right).
\end{multline*}
Here, the notation $O\!\left(\frac{d^{\beta-1}}{1-e^{-\frac{1}{2}\rho^2}}\right)$ means a quantity such that there exists $C>0$ and $\epsilon >0$, independent of $x,z,d$ and $\rho$, such that whenever $\frac{d^{\beta-1}}{1-e^{-\frac{1}{2}\rho^2}}\leq \epsilon$, the norm of this quantity is smaller than $C\frac{d^{\beta-1}}{1-e^{-\frac{1}{2}\rho^2}}$.
\end{lem}

\begin{proof}
By eq.~\eqref{eq variance evxy} and~\eqref{eq odet near diag}, we have:
\begin{equation*}
\left(\frac{\pi}{d}\right)^{n} \begin{pmatrix}
E_d(0,0) & E_d\left(0,\frac{z}{\sqrt{d}}\right)\\
E_d\left(\frac{z}{\sqrt{d}},0\right) & E_d\left(\frac{z}{\sqrt{d}},\frac{z}{\sqrt{d}}\right)
\end{pmatrix} = \begin{pmatrix}
1 & e^{-\frac{1}{2}\Norm{z}^2} \\ e^{-\frac{1}{2}\Norm{z}^2} & 1
\end{pmatrix} \otimes \Id_{\R\left(\E \otimes \L^d \right)_x} + O\!\left(d^{\beta-1}\right),
\end{equation*}
where the error term is independent of $(x,z)$. Besides,
\begin{equation}
\label{eq inverse of the limit}
\begin{aligned}
\left(\begin{pmatrix}
1 & e^{-\frac{1}{2}\Norm{z}^2} \\ e^{-\frac{1}{2}\Norm{z}^2} & 1
\end{pmatrix} \otimes \Id_{\R\left(\E \otimes \L^d \right)_x} \right)^{-1} =&\\
\frac{1}{1-e^{-\Norm{z}^2}} &\begin{pmatrix}
1 & -e^{-\frac{1}{2}\Norm{z}^2} \\ -e^{-\frac{1}{2}\Norm{z}^2} & 1
\end{pmatrix} \otimes \Id_{\R\left(\E \otimes \L^d \right)_x},
\end{aligned}
\end{equation}
and the eigenvalues of
\begin{equation*}
\begin{pmatrix}
1 & e^{-\frac{1}{2}\Norm{z}^2} \\ e^{-\frac{1}{2}\Norm{z}^2} & 1
\end{pmatrix} \otimes \Id_{\R\left(\E \otimes \L^d \right)_x}
\end{equation*}
are $1-e^{-\frac{1}{2}\Norm{z}^2}$ and $1+e^{-\frac{1}{2}\Norm{z}^2}$, which shows that:
\begin{equation*}
\Norm{\left(\begin{pmatrix}
1 & e^{-\frac{1}{2}\Norm{z}^2} \\ e^{-\frac{1}{2}\Norm{z}^2} & 1
\end{pmatrix} \otimes \Id_{\R\left(\E \otimes \L^d \right)_x} \right)^{-1}} \leq \frac{1}{1-e^{-\frac{1}{2}\Norm{z}^2}},
\end{equation*}
where $\Norm{\cdot}$ is the operator norm on $\End\left(\R^2 \otimes \R\left(\E \otimes \L^d \right)_x\right)$. Then, if $\Norm{z} \geq \rho$, we have:
\begin{equation*}
\frac{1}{1-e^{-\frac{1}{2}\Norm{z}^2}} \leq \frac{1}{1-e^{-\frac{1}{2}\rho^2}}.
\end{equation*}
Thus,
\begin{multline}
\label{eq interm 1}
\left(\frac{\pi}{d}\right)^{n} \begin{pmatrix}
E_d(0,0) & E_d\left(0,\frac{z}{\sqrt{d}}\right)\\
E_d\left(\frac{z}{\sqrt{d}},0\right) & E_d\left(\frac{z}{\sqrt{d}},\frac{z}{\sqrt{d}}\right)
\end{pmatrix} \\
=\begin{pmatrix}
1 & e^{-\frac{1}{2}\Norm{z}^2} \\ e^{-\frac{1}{2}\Norm{z}^2} & 1
\end{pmatrix} \otimes \Id_{\R\left(\E \otimes \L^d \right)_x} \left(\Id + O\!\left(\frac{d^{\beta-1}}{1-e^{-\frac{1}{2}\rho^2}}\right)\right).
\end{multline}
Taking the inverse of eq.~\eqref{eq interm 1}, we get:
\begin{multline*}
\left(\frac{d}{\pi}\right)^{n} \begin{pmatrix}
E_d(0,0) & E_d\left(0,\frac{z}{\sqrt{d}}\right)\\
E_d\left(\frac{z}{\sqrt{d}},0\right) & E_d\left(\frac{z}{\sqrt{d}},\frac{z}{\sqrt{d}}\right)
\end{pmatrix}^{-1}=\\
\left(\begin{pmatrix}
1 & e^{-\frac{1}{2}\Norm{z}^2} \\ e^{-\frac{1}{2}\Norm{z}^2} & 1
\end{pmatrix} \otimes \Id_{\R\left(\E \otimes \L^d \right)_x} \right)^{-1}\left(\Id + O\!\left(\frac{d^{\beta-1}}{1-e^{-\frac{1}{2}\rho^2}}\right)\right),
\end{multline*}
where we used the mean value inequality and the fact that the differential of $\Lambda \mapsto \Lambda^{-1}$ is bounded from above on the closed ball of center $\Id$ and radius $\frac{1}{2}$. Finally, eq.~\eqref{eq inverse of the limit} gives the result.
\end{proof}

Recall that $\Lambda_x(z)$ is defined for $x \in M$ and $z \in T_xM\setminus \{0\}$ by Def.~\ref{def Lambda z inf}. Recall also that $\Lambda_d\left(x,y\right)$ is defined by Def.~\ref{def Lambda dxy}.

\begin{lem}
\label{lem near diag asymptotic Lambda d xy}
Let $\beta \in (0,1)$ and $\rho \in (0,1)$. Let $x \in M$ and $z \in B_{T_xM}(0,b_n \ln d)$ such that $\Norm{z} \geq \rho$. We denote $y = \exp_x\left(\frac{z}{\sqrt{d}}\right)$. Let $\nabla^d$ be any real metric connection. Then, in the real normal trivialization about $x$, we have:
\begin{equation*}
\Lambda_d\left(x,y\right) = \Lambda_x(z) + O\!\left(\frac{d^{\beta-1}}{(1-e^{-\frac{1}{2}\rho^2})^2}\right),
\end{equation*}
where the constant in the error term does not depend on $(x,z)$, $d$ or $\rho$.
\end{lem}

\begin{proof}
We know that $\Lambda_d(x,y)$ does not depend on the choice of $\nabla^d$ (see Rem.\ref{rem choice of nablad}). Hence, we can compute $\Lambda_d\left(x,y\right)$ with $\nabla^d$ trivial over $B_{T_xM}(0,R)$ in the real normal trivialization of $\E \otimes \L^d$ about $x$. 

Let $\beta \in (0,1)$ and $\rho \in (0,1)$, we apply Lemmas~\ref{lem near diagonal Lambda dxy 2} and \ref{lem near diag ev inverse} for $\frac{\beta}{2}$. Then, in the real normal trivialization about $x$, we have:
\begin{multline}
\label{eq interm 2}
\frac{\pi^n}{d^{n+1}} \begin{pmatrix}
\partial_xE_d(x,x) & \partial_xE_d(x,y)\\
\partial_xE_d(y,x) & \partial_xE_d(y,y)
\end{pmatrix}
\begin{pmatrix}
E_d(x,x) & E_d(x,y)\\
E_d(y,x) & E_d(y,y)
\end{pmatrix}^{-1}
\begin{pmatrix}
\partial_y^\sharp E_d(x,x) & \partial_y^\sharp E_d(x,y)\\
\partial_y^\sharp E_d(y,x) & \partial_y^\sharp E_d(y,y)
\end{pmatrix}\\
= \left( \begin{pmatrix}
0 &  z^* \\ - z^* & 0
\end{pmatrix} \otimes \Id_{\R\left(\E \otimes \L^d \right)_x} + O\!\left(d^{\frac{\beta}{2}-1}\right)\right) \left(\begin{pmatrix}
1 & e^{-\frac{1}{2}\Norm{z}^2} \\ e^{-\frac{1}{2}\Norm{z}^2} & 1
\end{pmatrix}\otimes \Id_{\R\left(\E \otimes \L^d \right)_x}\right)^{-1}\times \\
\left(\Id + O\!\left(\frac{d^{\frac{\beta}{2}-1}}{1-e^{-\frac{1}{2}\rho^2}}\right)\right)\left(\begin{pmatrix}
0 &  -z \\ z & 0
\end{pmatrix} \otimes \Id_{\R\left(\E \otimes \L^d \right)_x} + O\!\left(d^{\frac{\beta}{2}-1}\right)\right).
\end{multline}
Since, $\rho \leq \norm{z} < b_n \ln d$, the norm of
\begin{equation*}
\left(\begin{pmatrix}
1 & e^{-\frac{1}{2}\Norm{z}^2} \\ e^{-\frac{1}{2}\Norm{z}^2} & 1
\end{pmatrix}\otimes \Id_{\R\left(\E \otimes \L^d \right)_x}\right)^{-1}
\end{equation*}
is smaller than $\left(1-e^{-\frac{1}{2}\Norm{z}^2}\right)^{-1} \leq \left(1-e^{-\frac{1}{2}\rho^2}\right)^{-1}$, and the norms of the other matrices appearing in~\eqref{eq interm 2} are $O\!\left(\ln d\right)$. Hence, the expression~\eqref{eq interm 2} equals:
\begin{equation}
\label{eq conditioning uniform}
\begin{aligned}
\frac{e^{-\Norm{z}^2}}{1-e^{-\Norm{z}^2}} \begin{pmatrix}
0 &  z^* \\ - z^* & 0
\end{pmatrix}\hspace{-1mm}\begin{pmatrix}
1 & -e^{-\frac{1}{2}\Norm{z}^2} \\ -e^{-\frac{1}{2}\Norm{z}^2} & 1
\end{pmatrix}\hspace{-1mm}
\begin{pmatrix}
0 &  -z \\ z & 0
\end{pmatrix}\! \otimes \Id_{\R\left(\E \otimes \L^d \right)_x}\! + O\!\left(\frac{d^{\beta-1}}{(1-e^{-\frac{1}{2}\rho^2})^2}\right)&\\
=\frac{e^{-\Norm{z}^2}}{1-e^{-\Norm{z}^2}}
\begin{pmatrix}
z^*\otimes z & e^{-\frac{1}{2}\Norm{z}^2}z^*\otimes z \\ e^{-\frac{1}{2}\Norm{z}^2}z^*\otimes z & z^*\otimes z
\end{pmatrix} \otimes \Id_{\R\left(\E \otimes \L^d \right)_x} + O\!\left(\frac{d^{\beta-1}}{(1-e^{-\frac{1}{2}\rho^2})^2}\right),&
\end{aligned}
\end{equation}
where the error term is independent of $(x,z)$. Finally, eq.~\eqref{eq conditioning uniform} and Lemma~\ref{lem near diagonal Lambda dxy 1} yield the result.
\end{proof}

\begin{lem}
\label{lem near diagonal cond exp}
Let $\beta \in (0,1)$ and $\rho \in (0,1)$. Let $x \in M$ and $z \in B_{T_xM}(0,b_n \ln d)$ such that $\Norm{z} \geq \rho$. We denote $y = \exp_x\left(\frac{z}{\sqrt{d}}\right)$. Let $\nabla^d$ be any real metric connection. Then,
\begin{multline*}
\left(\frac{\pi^n}{d^{n+1}}\right)^r \espcond{\odet{\nabla^d_{x}s_d}\!\odet{\nabla^d_{y}s_d}}{\ev_{x,y}^d(s_d)=0} =\\
\esp{\odet{X(\Norm{z}^2)}\odet{Y(\Norm{z}^2)}} + O\!\left(f(\rho^2)^{\frac{r(n+1)}{2}+4}d^{\beta-1}\right),
\end{multline*}
where the constant in the error term does not depend on $(x,z)$, $d$ or $\rho$.
\end{lem}

\begin{proof}
Let $x \in M$ and $z \in B_{T_xM}(0,b_n \ln d)\setminus\{0\}$, let $y = \exp_x\left(\frac{z}{\sqrt{d}}\right)$ then we have:
\begin{multline*}
\left(\frac{\pi^n}{d^{n+1}}\right)^r \espcond{\odet{\nabla^d_{x}s_d}\!\odet{\nabla^d_{y}s_d}}{\ev_{x,y}^d(s_d)=0}\\
\begin{aligned}
&=\espcond{\odet{\left(\frac{\pi^n}{d^{n+1}}\right)^\frac{1}{2}\nabla^d_{x}s_d}\!\odet{\left(\frac{\pi^n}{d^{n+1}}\right)^\frac{1}{2}\nabla^d_{y}s_d}}{\ev_{x,y}^d(s_d)=0}\\
&=\esp{\odet{L'_d(x)}\odet{L'_d(y)}},
\end{aligned} 
\end{multline*}
where $\left(L'_d(x),L'_d(y)\right)$ is a centered Gaussian vector in 
\begin{equation*}
\R\left(\E \otimes \L^d \right)_x \otimes T^*_xM \oplus \R\left(\E \otimes \L^d \right)_y \otimes T^*_yM
\end{equation*}
with variance operator $\Lambda_d(x,y)$. We can consider $\left(L'_d(x),L'_d(y)\right)$ as a random vector in $\R^2 \otimes \R\left(\E \otimes \L^d \right)_x \otimes T^*_xM$, via the real normal trivialization 	about $x$. From now on, we work in this trivialization. Let $\rho \in (0,1)$ and $\beta \in (0,1)$, we assume that $\rho \leq \Norm{z} < b_n \ln d$. Then, by Lemma~\ref{lem near diag asymptotic Lambda d xy}, we have:
\begin{equation*}
\Lambda_d\left(x,y\right) = \Lambda_x(z) + O\!\left(\frac{d^{\beta-1}}{(1-e^{-\frac{1}{2}\rho^2})^2}\right).
\end{equation*}
Moreover, by Cor.~\ref{cor Lambda z det and inverse} and Rem.~\ref{rem norm Lambda z inverse}, $\Norm{\Lambda_x(z)^{-1}} \leq f\left(\Norm{z}^2\right)\leq f(\rho^2)$. Hence, we have:
\begin{equation*}
\Lambda_d\left(x,y\right) = \Lambda_x(z) \left(\Id + O\!\left(f(\rho^2)\frac{d^{\beta-1}}{(1-e^{-\frac{1}{2}\rho^2})^2}\right)\right) = \Lambda_x(z) \left(\Id + O\!\left(f(\rho^2)^3d^{\beta-1}\right)\right),
\end{equation*}
where we used the fact that $\frac{1}{1-e^{-\frac{1}{2}\rho^2}} \leq f(\rho^2)$ (see the proof of Cor.~\ref{cor Lambda z det and inverse}). Then, we get:
\begin{equation}
\label{eq det near diag cond exp}
\det\left(\Lambda_d\left(x,y\right)\right) = \det\left(\Lambda_x(z)\right) \left(1+  O\!\left(f(\rho^2)^3 d^{\beta-1}\right)\right)
\end{equation}
and
\begin{equation*}
\Lambda_d\left(x,y\right)^{-1} = \Lambda_x(z)^{-1} \left(\Id + O\!\left(f(\rho^2)^3 d^{\beta-1}\right)\right) = \Lambda_x(z)^{-1} + O\!\left(f(\rho^2)^4 d^{\beta-1}\right).
\end{equation*}
Thus there exists $K>0$ and $\epsilon >0$ such that, whenever $f(\rho^2)^4 d^{\beta-1}\leq \epsilon$,
\begin{equation*}
\Norm{\Lambda_d(x,y)^{-1} - \Lambda_x(z)^{-1}} \leq K f(\rho^2)^4 d^{\beta-1}.
\end{equation*}

By the mean value inequality, for every $L=(L_1,L_2) \in \R^2 \otimes T^*_xM \otimes \R\!\left(\E \otimes \L^d\right)_x$ we have:
\begin{multline*}
\norm{\exp\left(-\frac{1}{2}\prsc{\left(\Lambda_d(x,y)^{-1}-\Lambda_x(z)^{-1}\right)L}{L}\right)-1}\\
\leq \frac{K}{2}\Norm{L}^2 f(\rho^2)^4 d^{\beta-1} \exp\left(\frac{K}{2}\Norm{L}^2f(\rho^2)^4 d^{\beta-1}\right),
\end{multline*}
whenever $f(\rho^2)^4 d^{\beta-1} \leq \epsilon$. Let $\dx L$ denote the normalized Lebesgue measure on this vector space, and recall that we defined $\left(L_x(0),L_x(z)\right)$ above (Def.~\ref{def Lx 0 z}). Then, we have:
\begin{align*}
&\begin{aligned}
(2\pi)^{nr} \left\lvert \det\left(\Lambda_d(x,y)\right)^\frac{1}{2}\right. &\esp{\odet{L'_d(x)}\odet{L'_d(y)}}\\
&-\left. \det\left(\Lambda_x(z)\right)^\frac{1}{2}\esp{\odet{L_x(0)}\odet{L_x(z)}}\right\rvert
\end{aligned}\\
&\begin{aligned}
\leq \int \odet{L_1} & \odet{L_2} \exp\left(-\frac{1}{2}\prsc{\Lambda_x(z)^{-1}L}{L}\right)\times\\
&\norm{\exp\left(-\frac{1}{2}\prsc{\left(\Lambda_d(x,y)^{-1}-\Lambda_x(z)^{-1}\right)L}{L}\right)-1} \dx L
\end{aligned}\\
&\begin{aligned}
\leq \frac{K}{2} f(\rho^2)^4 d^{\beta-1}\int \odet{L_1}& \odet{L_2} \Norm{L}^2\times\\
&\exp\left(-\frac{1}{2}\prsc{\left(\Lambda_x(z)^{-1}-\frac{K}{2} f(\rho^2)^4 d^{\beta-1}\Id\right)L}{L}\right)\dx L,
\end{aligned}
\end{align*}
whenever $f(\rho^2)^4 d^{\beta-1}\leq \epsilon$. Since $\Norm{\Lambda_x(d)} <2$ by Cor.~\ref{cor Lambda z det and inverse}, the smallest eigenvalue of $\Lambda_x(z)^{-1}$ is larger than $\frac{1}{2}$. Thus, if $f(\rho^2)^4 d^{\beta-1} \leq \frac{1}{2K}$, for every $L$ we have:
\begin{equation*}
\prsc{\left(\Lambda_x(z)^{-1}-\frac{K}{2} f(\rho^2)^4 d^{\beta-1}\Id\right)L}{L} \geq \frac{1}{4}\Norm{L}^2.
\end{equation*}
Hence, the last integral above is bounded by:
\begin{equation*}
\int \odet{L_1} \odet{L_2} \Norm{L}^2 \exp\left(-\frac{1}{8}\Norm{L}^2\right)\dx L < +\infty.
\end{equation*}

Then, we have:
\begin{multline*}
\det\left(\Lambda_d(x,y)\right)^\frac{1}{2}\esp{\odet{L'_d(x)}\odet{L'_d(y)}} =\\
\det\left(\Lambda_x(z)\right)^\frac{1}{2}\esp{\odet{L_x(0)}\odet{L_x(z)}} + O\!\left(f(\rho^2)^4 d^{\beta-1}\right),
\end{multline*}
and by~\eqref{eq det near diag cond exp}, we obtain:
\begin{multline*}
\esp{\odet{L'_d(x)}\odet{L'_d(y)}} =\\
\esp{\odet{L_x(0)}\odet{L_x(z)}} \left(1 + O\!\left(f(\rho^2)^3 d^{\beta-1}\right)\right)\\
 + \det\left(\Lambda_x(z)\right)^{-\frac{1}{2}}O\!\left(f(\rho^2)^4 d^{\beta-1}\right)\left(1 + O\!\left(f(\rho^2)^3 d^{\beta-1}\right)\right).
\end{multline*}
Since, for all $t>0$ we have (see Lem.~\ref{lem eigenvalues Lambda z}):
\begin{align*}
\frac{1}{1+e^{-\frac{1}{2}t}} & \leq 1, & \frac{1}{1-e^{-\frac{1}{2}t}} &\leq f(t) & &\text{and} & \frac{1+e^{-\frac{1}{2}t}}{1-e^{-t}+te^{-\frac{1}{2}t}} &\leq f(t),
\end{align*}
by Cor.~\ref{cor Lambda z det and inverse} we have: $\det\left(\Lambda_x(z)\right)^{-\frac{1}{2}} \leq f(\rho^2)^{\frac{r(n+1)}{2}}$. Besides, by Cor.~\ref{cor same distribution}, we have:
\begin{equation*}
\esp{\odet{L_x(0)}\odet{L_x(z)}}= \esp{\odet{X(\Norm{z}^2}\odet{Y(\Norm{z}^2}},
\end{equation*}
and by Lemma~\ref{lem cond exp bounded} this quantity is bounded from above by $n^r$. Finally, we have:
\begin{align*}
\esp{\odet{L'_d(x)}\odet{L'_d(y)}} =& \esp{\odet{X(\Norm{z}^2}\odet{Y(\Norm{z}^2}}\\
 &+ O\!\left(f(\rho^2)^{4+\frac{r(n+1)}{2}} d^{\beta-1}\right).\qedhere
\end{align*}
\end{proof}

The following corollary is not necessary to the proof of Thm.~\ref{thm asymptotics variance} but is worth mentioning.

\begin{cor}
\label{cor near diagonal cond exp}
Let $\beta \in (0,1)$. Let $x \in M$ and $z \in B_{T_xM}(0,b_n \ln d)\setminus \{0\}$. We denote $y = \exp_x\left(\frac{z}{\sqrt{d}}\right)$. Let $\nabla^d$ be any real metric connection. Then, we have:
\begin{multline*}
\left(\frac{\pi^n}{d^{n+1}}\right)^r \espcond{\odet{\nabla^d_{x}s_d}\!\odet{\nabla^d_{y}s_d}}{\ev_{x,y}^d(s_d)=0} =\\
\esp{\odet{X(\Norm{z}^2}\odet{Y(\Norm{z}^2}} + O\!\left(d^{\beta-1}\right),
\end{multline*}
where the error term depends on $z$ but not on $x$.
\end{cor}

\begin{proof}
Let us fix, $\beta$, $x$ and $z$, then we set $\rho = \Norm{z}$ and we apply Lemma~\ref{lem near diagonal cond exp}.
\end{proof}

Before we can conclude the proof of Thm.~\ref{thm asymptotics variance}, we need one last lemma.

\begin{lem}
\label{lem bounded cond exp Lambda d}
Let $x \in M$ and $z \in B_{T_xM}(0,b_n \ln d)\setminus \{0\}$. We denote $y = \exp_x\left(\frac{z}{\sqrt{d}}\right)$. Let $\beta \in (0,1)$ and let $\nabla^d$ be any real metric connection. Then, we have:
\begin{equation*}
\left(\frac{\pi^n}{d^{n+1}}\right)^r \espcond{\odet{\nabla^d_{x}s_d}\!\odet{\nabla^d_{y}s_d}}{\ev_{x,y}^d(s_d)=0} \leq \frac{(2r)!}{r!}n^r + O\!\left(d^{\beta-1}\right),
\end{equation*}
where the error term is independent of $(x,z)$.
\end{lem}

\begin{proof}
Let $x \in M$, let $z \in B_{T_xM}(0,b_n \ln d)\setminus \{0\}$ and let $y = \exp_x\left(\frac{z}{\sqrt{d}}\right)$. As in the proof of Lem.~\ref{lem near diagonal cond exp}, let $\left(L'_d(x),L'_d(y)\right)$ be a centered Gaussian vector in $\R^2 \otimes \R\left(\E \otimes \L^d \right)_x \otimes T^*_xM$ whose variance operator is $\Lambda_d(x,y)$, read in the real normal trivialization about $x$. In the sequel, we work in this trivialization. We have:
\begin{multline*}
\left(\frac{\pi^n}{d^{n+1}}\right)^r \espcond{\odet{\nabla^d_{x}s_d}\!\odet{\nabla^d_{y}s_d}}{\ev_{x,y}^d(s_d)=0}=\\ \esp{\odet{L'_d(x)}\odet{L'_d(y)}}.
\end{multline*}

The proof follows the same lines as that of Lem.~\ref{lem cond exp bounded}, the main difference being that the variance operator is not explicit. An additional difficulty comes from the fact that the estimate for $\Lambda_d(x,y)$ given by Lemma~\ref{lem near diag asymptotic Lambda d xy} is not uniform in $z \in B_{T_xM}(0,b_n \ln d)\setminus \{0\}$, hence it is useless here. Fortunately, we only need to bound its trace, which is bounded from above by that of the unconditional variance operator:
\begin{equation*}
\frac{\pi^n}{d^{n+1}}\begin{pmatrix}
\partial_x\partial_y^\sharp E_d(0,0) & \partial_x\partial_y^\sharp E_d\left(0,\frac{z}{\sqrt{d}}\right)\\
\partial_x\partial_y^\sharp E_d\left(\frac{z}{\sqrt{d}},0\right) & \partial_x\partial_y^\sharp E_d\left(\frac{z}{\sqrt{d}},\frac{z}{\sqrt{d}}\right)
\end{pmatrix},
\end{equation*}
and Lemma~\ref{lem near diagonal Lambda dxy 1} allows us to bound the latter.

By the Cauchy-Schwarz inequality,
\begin{equation}
\label{step 0}
\esp{\odet{L'_d(x)}\odet{L'_d(y)}} \leq \esp{\odet{L'_d(x)}^2}^\frac{1}{2} \esp{\odet{L'_d(y)}^2}^\frac{1}{2}.
\end{equation}
Let $\Lambda_{d,1}(x,y)$ and $\Lambda_{d,2}(x,y)$ denote the variance operators of $L'_d(x)$ and $L'_d(y)$ respectively, so that:
\begin{equation}
\label{eq def Lambda 1 and 2}
\Lambda_d(x,y) = \begin{pmatrix}
\Lambda_{d,1}(x,y) & * \\ * & \Lambda_{d,2}(x,y)
\end{pmatrix}.
\end{equation}
Let us choose orthonormal bases of $T_xM$ and $\R\left(\E \otimes \L^d\right)_x$. We denote by $\left(L'_d(x)_{ij}\right)_{\substack{1\leq i \leq r\\1\leq j\leq n}}$ the coefficients of the matrix of $L'_d(x)$ in these bases, and by $\left(L'_d(x)_{i}\right)_{1\leq i \leq r}$ its rows. As in the proof of Lem.~\ref{lem cond exp bounded}, we have:
\begin{equation}
\label{step 1}
\begin{aligned}
\odet{L'_d(x)}^2 &= \det\left(L'_d(x) \left(L'_d(x)\right)^*\right) = \det\left(\prsc{L'_d(x)_i}{L'_d(x)_j}\right)\\
& \leq \Norm{L'_d(x)_1}^2 \cdots \Norm{L'_d(x)_r}^2.
\end{aligned}
\end{equation}
Then, we have:
\begin{equation}
\label{step 2}
\begin{aligned}
\esp{\Norm{L'_d(x)_1}^2 \cdots \Norm{L'_d(x)_r}^2} &= \esp{\prod_{i=1}^r \left(\sum_{j=1}^n \left(L'_d(x)_{ij}\right)^2\right)}\\
&= \sum_{1 \leq j_1,\dots,j_r\leq n} \esp{\prod_{i=1}^r \left(L'_d(x)_{i(j_i)}\right)^2}.
\end{aligned}
\end{equation}

Let $j_1,\dots,j_r \in \{1,\dots,r\}$, we denote $X_i = L'_d(x)_{i(j_i)}$. Then, by Wick's formula (see \cite[lem.~11.6.1]{TA2007}), we have:
\begin{equation*}
\esp{\prod_{i=1}^r \left(L'_d(x)_{i(j_i)}\right)^2} = \esp{\prod_{i=1}^r (X_i)^2} = \sum_{\left(\{a_i,b_i\}\right)} \prod_{i=1}^r \esp{X_{\lfloor \frac{a_i}{2} \rfloor}X_{\lfloor \frac{b_i}{2} \rfloor}},
\end{equation*}
where we sum over all the partitions into pairs $\left(\{a_i,b_i\}\right)_{1 \leq i \leq r}$ of $\{1,\dots,2r\}$. Hence, by Cauchy-Schwarz inequality again, we get:
\begin{align*}
\esp{\prod_{i=1}^r \left(L'_d(x)_{i(j_i)}\right)^2} &\leq \sum_{\left(\{a_i,b_i\}\right)} \prod_{i=1}^r \esp{\left(X_{\lfloor \frac{a_i}{2}\rfloor}\right)^2}^\frac{1}{2} \esp{\left(X_{\lfloor \frac{b_i}{2}\rfloor}\right)^2}^\frac{1}{2}\\
& \leq \sum_{\left(\{a_i,b_i\}\right)} \prod_{k=1}^{2r} \esp{\left(X_{\lfloor \frac{k}{2}\rfloor}\right)^2}^\frac{1}{2}\\
& \leq \sum_{\left(\{a_i,b_i\}\right)} \prod_{l=1}^{r} \esp{\left(X_l\right)^2}\\
&\leq \frac{(2r)!}{2^r r!} \prod_{i=1}^{r} \esp{\left(L'_d(x)_{i(j_i)}\right)^2}.
\end{align*}
Thus, we have:
\begin{equation}
\label{step 3}
\begin{aligned}
\sum_{1 \leq j_1,\dots,j_r\leq n} \esp{\prod_{i=1}^r \left(L'_d(x)_{i(j_i)}\right)^2} &\leq \frac{(2r)!}{2^r r!} \sum_{1 \leq j_1,\dots,j_r\leq n} \prod_{i=1}^{r} \esp{\left(L'_d(x)_{i(j_i)}\right)^2}\\
&\leq \frac{(2r)!}{2^r r!}\prod_{i=1}^{r} \left(\sum_{j=1}^n \esp{\left(L'_d(x)_{ij}\right)^2}\right)\\
&\leq \frac{(2r)!}{2^r r!} \left(\sum_{i=1}^r \sum_{j=1}^n \esp{\left(L'_d(x)_{ij}\right)^2}\right)^r\\
&\leq \frac{(2r)!}{2^r r!} \tr\left(\Lambda_{d,1}(x,y)\right)^r,
\end{aligned}
\end{equation}
where $\tr$ stands for the trace operator. Finally, by~\eqref{step 1}, \eqref{step 2} and~\eqref{step 3}, we have:
\begin{equation*}
\esp{\odet{L'_d(x)}^2} \leq \frac{(2r)!}{2^r r!} \tr\left(\Lambda_{d,1}(x,y)\right)^r,
\end{equation*}
and similarly,
\begin{equation*}
\esp{\odet{L'_d(y)}^2} \leq \frac{(2r)!}{2^r r!} \tr\left(\Lambda_{d,2}(x,y)\right)^r.
\end{equation*}
Thus, by~\eqref{step 0}, we get:
\begin{equation}
\label{step 4}
\begin{aligned}
\esp{\odet{L'_d(x)}\odet{L'_d(y)}} &\leq \frac{(2r)!}{2^r r!}\tr\left(\Lambda_{d,1}(x,y)\right)^\frac{r}{2} \tr\left(\Lambda_{d,2}(x,y)\right)^\frac{r}{2}\\
&\leq \frac{(2r)!}{2^r r!} \tr\left(\Lambda_d(x,y)\right)^r.
\end{aligned}
\end{equation}

Let $\beta \in (0,1)$, by eq.~\eqref{step 4}, we only need to prove that $\tr\left(\Lambda_d(x,y)\right) \leq 2n + O\!\left(d^{\beta-1}\right)$ to complete the proof. By eq.~\eqref{eq variance evxy},
\begin{equation*}
\begin{pmatrix}
E_d(x,x) & E_d(x,y)\\ E_d(y,x) & E_d(y,y)
\end{pmatrix}
\end{equation*}
is a variance operator. Hence it is a positive symmetric operator and so is its inverse. Besides, by~\eqref{eq variance jxy}, we know that:
\begin{equation*}
\begin{pmatrix}
\partial_y^\sharp E_d(x,x) & \partial_y^\sharp E_d(x,y)\\
\partial_y^\sharp E_d(y,x) & \partial_y^\sharp E_d(y,y)
\end{pmatrix} = 
\begin{pmatrix}
\partial_xE_d(x,x) & \partial_xE_d(x,y)\\
\partial_xE_d(y,x) & \partial_xE_d(y,y)
\end{pmatrix}^*.
\end{equation*}
Then, the diagonal coefficients of:
\begin{equation*}
\begin{pmatrix}
\partial_xE_d(x,x) & \partial_xE_d(x,y)\\
\partial_xE_d(y,x) & \partial_xE_d(y,y)
\end{pmatrix}
\begin{pmatrix}
E_d(x,x) & E_d(x,y)\\ E_d(y,x) & E_d(y,y)
\end{pmatrix}^{-1}
\begin{pmatrix}
\partial_y^\sharp E_d(x,x) & \partial_y^\sharp E_d(x,y)\\
\partial_y^\sharp E_d(y,x) & \partial_y^\sharp E_d(y,y)
\end{pmatrix}
\end{equation*}
are non-negative, and so is its trace. Finally, by the definition of $\Lambda_d(x,y)$ (Def.~\ref{def Lambda dxy}), we have:
\begin{equation}
\label{trace}
\tr\left(\Lambda_d(x,y)\right) \leq \frac{\pi^n}{d^{n+1}}\tr \begin{pmatrix}
\partial_x\partial_y^\sharp E_d(x,x) & \partial_x\partial_y^\sharp E_d(x,y)\\
\partial_x\partial_y^\sharp E_d(y,x) & \partial_x\partial_y^\sharp E_d(y,y)
\end{pmatrix}.
\end{equation}

Note that what we have done so far works for any choice of connection since $\Lambda_d(x,y)$ is independent of this choice. However, the right-hand side of eq.~\eqref{trace} depends on the choice $\nabla^d$. We use a real metric connection that is trivial on $B_{T_xM}(0,R)$ in the real normal trivialization about $x$. Then, by Lemma~\ref{lem near diagonal Lambda dxy 1}, we have:
\begin{equation*}
\tr\left(\Lambda_d(x,y)\right) \leq 2n + O\!\left(d^{\beta-1}\right).\qedhere
\end{equation*}
\end{proof}

\subsubsection{Conclusion of the proof}
\label{subsubsec conclusion of the proof}

We can now prove Theorem~\ref{thm asymptotics variance}.

\begin{lem}
\label{lem very short distance}
Let $\alpha >0$, let $\phi \in \mathcal{C}^0(M)$ and let $x \in M$, then we have:
\begin{multline*}
\norm{\int_{B_{T_xM}\left(0,d^{-\alpha}\right)} \phi\left(\exp_x\left(\frac{z}{\sqrt{d}}\right)\right) \kappa\left(\frac{z}{\sqrt{d}}\right)^\frac{1}{2} \left(\frac{1}{d^r} D_d(x,z) - D_{n,r}(\Norm{z}^2)\right) \dx z}\\
= \Norm{\phi}_\infty O\!\left(d^{(r-n)\alpha}\right),
\end{multline*}
where the error term does not depend on $x$ or $\phi$.
\end{lem}

\begin{proof}
We have:
\begin{multline*}
\norm{\int_{B_{T_xM}\left(0,d^{-\alpha}\right)} \phi\left(\exp_x\left(\frac{z}{\sqrt{d}}\right)\right) \kappa\left(\frac{z}{\sqrt{d}}\right)^\frac{1}{2} \left(\frac{1}{d^r} D_d(x,z) - D_{n,r}(\Norm{z}^2)\right) \dx z}\\
\leq \Norm{\phi}_\infty \left(\sup_{B_{T_xM}\left(0,b_n\frac{\ln d}{\sqrt{d}}\right)} \norm{\kappa}^\frac{1}{2} \right) \int_{B_{T_xM}\left(0,d^{-\alpha}\right)} \left(\frac{1}{d^r} \norm{D_d(x,z)} + \norm{D_{n,r}(\Norm{z}^2)}\right) \dx z.
\end{multline*}
Since $\kappa(z) = 1 + O\left(\Norm{z}^2\right)$ uniformly in $x$ (see~\eqref{eq estimate kappa}), we have:
\begin{equation*}
\sup_{B_{T_xM}\left(0,b_n \frac{\ln d}{\sqrt{d}}\right)}\norm{\kappa}^\frac{1}{2} = 1 + O\!\left(\frac{(\ln d)^2}{d}\right),
\end{equation*}
and this term is bounded. Thus, we only need to consider the integrals of $\frac{1}{d^r} \norm{D_d(x,z)}$ and $\norm{D_{n,r}(\Norm{z}^2)}$. By Lemma~\ref{lem cond exp bounded}, we have:
\begin{equation}
\label{integrability}
\begin{aligned}
\int_{B\left(0,d^{-\alpha}\right)} \norm{D_{n,r}(\Norm{z}^2)}\dx z \leq& \vol{\S^{n-1}}\int_{\rho=0}^{d^{-\alpha}} \frac{\esp{\odet{X(\rho^2)}\odet{Y(\rho^2}}}{\left(1-e^{-\rho^2}\right)^\frac{r}{2}} \rho^{n-1}\dx \rho\\
&+ (2\pi)^r \left(\frac{\vol{\S^{n-r}}}{\vol{\S^n}}\right)^2 \vol{B_{T_xM}\left(0,d^{-\alpha}\right)}\\
\leq& \frac{n^r}{2}\vol{\S^{n-1}} \int_{t=0}^{d^{-2\alpha}} \frac{t^\frac{n-2}{2}}{\left(1-e^{-t}\right)^\frac{r}{2}}\dx t + O\!\left(d^{-n\alpha}\right).
\end{aligned}
\end{equation}
Then, since there exists $C>0$ such that $\frac{t}{1-e^{-t}}\leq C$ for all $t \in (0,1]$, we get:
\begin{equation}
\label{eq integrability easy}
\int_{t=0}^{d^{-2\alpha}} \frac{t^\frac{n-2}{2}}{\left(1-e^{-t}\right)^\frac{r}{2}}\dx t \leq C \int_{t=0}^{d^{-2\alpha}} t^\frac{n-2-r}{2} \dx t = O\!\left(d^{(r-n)\alpha}\right).
\end{equation}
Hence, $\displaystyle\int_{B\left(0,d^{-\alpha}\right)} \norm{D_{n,r}(\Norm{z}^2)}\dx z = O\!\left(d^{(r-n)\alpha}\right)$. Let us denote $y = \exp_x\left(\frac{z}{\sqrt{d}}\right)$. By the definition of $D_d(x,z)$ (cf.~\eqref{def Dd xz}), we have:
\begin{align*}
\frac{1}{d^r}\norm{D_d(x,z)} &\leq \frac{1}{d^r}\frac{\espcond{\odet{\nabla^d_{x}s_d}\!\odet{\nabla^d_{y}s_d}}{\ev_{x,y}^d(s_d)=0}}{\odet{\ev_{x,y}^d}}\\
&+ \frac{1}{d^r}\frac{\espcond{\odet{\nabla^d_{x}s_d}}{s_d(x)=0}}{\odet{\ev_x^d}}\frac{\espcond{\odet{\nabla^d_{y}s_d}}{s_d(y)=0}}{\odet{\ev_y^d}}.
\end{align*}
Then, let $\beta \in (0,1)$ and $\beta' \in \left(0,\frac{1}{2r+1}\right)$, by Prop.~\ref{prop ultim estimate evxy} and Lem.~\ref{lem bounded cond exp Lambda d} we have:
\begin{align*}
\frac{1}{d^r} \frac{\espcond{\odet{\nabla^d_{x}s_d}\!\odet{\nabla^d_{y}s_d}}{\ev_{x,y}^d(s_d)=0}}{\odet{\ev_{x,y}^d}} &\leq \frac{\frac{(2r)!}{r!}n^r + O\!\left(d^{\beta-1}\right)}{\left(1-e^{-\Norm{z}^2}\right)^\frac{r}{2}} \left(1 + O\!\left(d^{-\beta'}\right)\right)\\
&\leq C \left(\frac{1}{1-e^{-\Norm{z}^2}}\right)^\frac{r}{2},
\end{align*}
for some large $C$. By a polar change of coordinates similar to~\eqref{integrability} and~\eqref{eq integrability easy}, we show that the integral of this term over $B_{T_xM}\left(0,d^{-\alpha}\right)$ is a $O\!\left(d^{(r-n)\alpha}\right)$. Finally, by Lem.~\ref{lem estimates odet evx} and~\ref{lem estimates cond exp x} we have:
\begin{equation*}
\frac{1}{d^r}\frac{\espcond{\odet{\nabla^d_{x}s_d}}{s_d(x)=0}}{\odet{\ev_x^d}}\frac{\espcond{\odet{\nabla^d_{y}s_d}}{s_d(y)=0}}{\odet{\ev_y^d}} = O\!\left(1\right).
\end{equation*}
Hence the integral of this term over $B_{T_xM}\left(0,d^{-\alpha}\right)$ is a $O\!\left(d^{-n \alpha}\right)$.
\end{proof}

Recall that we defined $\alpha_0 = \dfrac{n-r}{2(2r+1)(2n+1)}$ (see Ntn.~\ref{ntn alpha0}). Let us denote $\alpha_1 = \dfrac{\alpha_0}{n-r}$.

\begin{lem}
\label{lem short distance}
Let $\alpha \in (0,\alpha_1)$, let $\phi \in \mathcal{C}^0(M)$ and $x \in M$, then we have:
\begin{multline*}
\norm{\int_{d^{-\alpha}\leq \Norm{z} < b_n \ln d} \phi\left(\exp_x\left(\frac{z}{\sqrt{d}}\right)\right) \kappa\left(\frac{z}{\sqrt{d}}\right)^\frac{1}{2} \left(\frac{1}{d^r} D_d(x,z) - D_{n,r}(\Norm{z}^2)\right) \dx z}\\
= \Norm{\phi}_\infty O\!\left(d^{(r-n)\alpha}\right),
\end{multline*}
where the error term does not depend on $x$ or $\phi$.
\end{lem}

\begin{proof}
As in the proof of Lemma~\ref{lem very short distance}, since $\kappa^\frac{1}{2}$ is bounded on $B_{T_xM}\left(0,b_n\frac{\ln d}{\sqrt{d}}\right)$ uniformly in $x \in M$, we only need to prove that:
\begin{equation*}
\norm{\frac{1}{d^r} D_d(x,z) - D_{n,r}(\Norm{z}^2)} = O\!\left(d^{(r-n)\alpha - \alpha'}\right)
\end{equation*}
for some $\alpha' >0$. Then, since $\vol{B_{T_xM}(0,b_n\ln d)} = O\!\left((\ln d)^n\right) = O\left(d^{\alpha'}\right)$, we get the result by integrating over $B_{T_xM}(0,b_n\ln d) \setminus B_{T_xM}(0,d^{-\alpha})$.

Since $\alpha \in (0,\alpha_1)$, we have $0< n\alpha < \frac{1}{2r+1}$ and we can choose a positive $\beta \in \left(n\alpha, \frac{1}{2r+1}\right)$. Let $\beta' \in (0,1)$ be such that:
\begin{equation}
\label{eq def beta'}
1 - 2\alpha(8+r(n+1)) - \beta < \beta' < 1 - 2\alpha(8+r(n+1)) - n\alpha <1.
\end{equation}
We already know that $-\beta < -n\alpha$, so we only need to check that $0 <1-\alpha(16+2rn+2r+n)$ to ensure the existence of such a $\beta'$. This goes as follows:
\begin{equation*}
1-\alpha(16+2rn+2r+n) > 1-2\alpha_1(8+rn+n+r) = \frac{3rn+n+r-7}{(2r+1)(2n+1)}>0.
\end{equation*}

By Lemma~\ref{lem near diagonal cond exp}, for every $x \in M$ and $z \in B_{T_xM}(0,b_n\ln d)$ such that $\Norm{z}\geq d^{-\alpha}$ we have:
\begin{multline}
\label{eq near diag step 1}
\left(\frac{\pi^n}{d^{n+1}}\right)^r \espcond{\odet{\nabla^d_{x}s_d}\!\odet{\nabla^d_{y}s_d}}{\ev_{x,y}^d(s_d)=0} =\\
\esp{\odet{X(\Norm{z}^2)}\odet{Y(\Norm{z}^2)}} + O\!\left(f(d^{-2\alpha})^{\frac{r(n+1)}{2}+4}d^{\beta'-1}\right),
\end{multline}
where, as usual, $y$ stands for $\exp_x\left(\frac{z}{\sqrt{d}}\right)$. Recall that we have: $f(t) \sim \frac{12}{t^2}$ as $t \to 0$ (cf.~Rem.~\ref{rem norm Lambda z inverse}). Then, we get:
\begin{equation*}
f(d^{-2\alpha})^{\frac{r(n+1)}{2}+4} = O\!\left(d^{2\alpha \left(8 +r(n+1)\right)}\right).
\end{equation*}
We set $\alpha' = 1 - 2\alpha \left(8 +r(n+1)\right)-\beta' - n \alpha$, so that the error term in~\eqref{eq near diag step 1} is a $O\!\left(d^{-n\alpha -\alpha'}\right)$. By~\eqref{eq def beta'}, we have $\alpha'>0$.

By Prop.~\ref{prop ultim estimate evxy}, applied for $\beta$, and eq.~\eqref{eq near diag step 1} we have:
\begin{multline}
\label{eq near diag step 2}
\frac{1}{d^r}\frac{\espcond{\odet{\nabla^d_{x}s_d}\!\odet{\nabla^d_{y}s_d}}{\ev_{x,y}^d(s_d)=0}}{\odet{\ev_{x,y}^d}}=\\
\frac{\esp{\odet{X(\Norm{z}^2)}\odet{Y(\Norm{z}^2)}} + O\!\left(d^{-n\alpha - \alpha'}\right)}{\left(1-e^{-\Norm{z}^2}\right)^\frac{r}{2}}\left(1 + O\!\left(d^{-\beta}\right)\right),
\end{multline}
for all $x \in M$ and $z \in T_xM$ such that $d^{-\alpha}\leq \Norm{z} < b_n \ln d$. Since
\[\left(1-e^{-d^{-2\alpha}}\right)^{-\frac{r}{2}} = O\!\left(d^{r\alpha}\right),\]
and the numerator of~\eqref{eq near diag step 2} is bounded (cf.~Lemma~\ref{lem cond exp bounded}), the right-hand side of equation~\eqref{eq near diag step 2} equals:
\begin{equation*}
\frac{\esp{\odet{X(\Norm{z}^2)}\odet{Y(\Norm{z}^2)}}}{\left(1-e^{-\Norm{z}^2}\right)^\frac{r}{2}} + O\!\left(d^{(r-n)\alpha - \alpha'}\right) + O\!\left(d^{r \alpha -\beta}\right).
\end{equation*}
Moreover, $n\alpha+\alpha' = 1 - 2\alpha \left(8 +r(n+1)\right)-\beta' < \beta$ (see eq.~\eqref{eq def beta'}), so that we have:
\begin{multline*}
\frac{1}{d^r}\frac{\espcond{\odet{\nabla^d_{x}s_d}\!\odet{\nabla^d_{y}s_d}}{\ev_{x,y}^d(s_d)=0}}{\odet{\ev_{x,y}^d}}=\\
\frac{\esp{\odet{X(\Norm{z}^2)}\odet{Y(\Norm{z}^2)}}}{\left(1-e^{-\Norm{z}^2}\right)^\frac{r}{2}} + O\!\left(d^{(r-n)\alpha - \alpha'}\right).
\end{multline*}

On the other hand, by Lemmas~\ref{lem estimates odet evx} and~\ref{lem estimates cond exp x}, we have:
\begin{multline*}
\frac{1}{d^r}\frac{\espcond{\odet{\nabla^d_{x}s_d}}{s_d(x)=0}}{\odet{\ev_x^d}}\frac{\espcond{\odet{\nabla^d_{y}s_d}}{s_d(y)=0}}{\odet{\ev_y^d}} =\\
(2 \pi)^r \left(\frac{\vol{\S^{n-r}}}{\vol{\S^n}} \right)^2+ O\!\left(d^{-1}\right).
\end{multline*}
Once again, eq.~\eqref{eq def beta'} shows that $n\alpha+\alpha'<\beta<1$. A fortiori $(n-r)\alpha+\alpha' <1$. Thus, for all $x\in M$ and $z \in T_xM$ such that $d^{-\alpha}\leq \Norm{z}< b_n \ln d$, we have:
\begin{equation*}
\norm{\frac{1}{d^r} D_d(x,z) - D_{n,r}(\Norm{z}^2)} = O\!\left(d^{(r-n)\alpha - \alpha'}\right),
\end{equation*}
where $\alpha' >0$ and the error term is independent of $(x,z)$.
\end{proof}

\begin{prop}
\label{prop main estimate near diag}
Let $\alpha \in (0,\alpha_0)$, let $\phi_1$ and $\phi_2 \in \mathcal{C}^0(M)$, we have the following asymptotic as $d \to +\infty$:
\begin{multline*}
\frac{1}{d^r}\int_{x \in M} \left(\int_{z \in B_{T_xM}\left(0,b_n \ln d\right)} \phi_1(x)\phi_2\left(\exp_x\left(\frac{z}{\sqrt{d}}\right)\right) D_d(x,z) \kappa\left(\frac{z}{\sqrt{d}}\right)^\frac{1}{2}\dx z \right) \rmes{M}\\
=\int_{x \in M} \left(\int_{z \in B_{T_xM}\left(0,b_n \ln d\right)}\hspace{-1mm} \phi_1(x)\phi_2\left(\!\exp_x\left(\frac{z}{\sqrt{d}}\right)\!\right) D_{n,r}(\Norm{z}^2) \dx z \right) \rmes{M}\\
+ \Norm{\phi_1}_\infty\Norm{\phi_2}_\infty O\!\left(d^{-\alpha}\right),
\end{multline*}
where the error term does not depend on $\left(\phi_1,\phi_2\right)$.
\end{prop}

\begin{proof}
Let $\alpha \in (0,\alpha_0)$, we set $\alpha' = \frac{\alpha}{n-r} \in (0,\alpha_1)$. Let $\phi_1,\phi_2 \in \mathcal{C}^0(M)$ and let $x \in M$, we apply Lemmas~\ref{lem very short distance} and~\ref{lem short distance} for $\alpha'$ and $\phi_2$. Then, we have:
\begin{multline}
\label{mult1}
\frac{1}{d^r}\int_{z \in B_{T_xM}\left(0,b_n \ln d\right)} \phi_1(x)\phi_2\left(\exp_x\left(\frac{z}{\sqrt{d}}\right)\right) D_d(x,z) \kappa\left(\frac{z}{\sqrt{d}}\right)^\frac{1}{2}\dx z\\
=\int_{z \in B_{T_xM}\left(0,b_n \ln d\right)} \phi_1(x)\phi_2\left(\exp_x\left(\frac{z}{\sqrt{d}}\right)\right) D_{n,r}(\Norm{z}^2) \kappa\left(\frac{z}{\sqrt{d}}\right)^\frac{1}{2}\dx z\\
+ \norm{\phi_1(x)}\Norm{\phi_2}_{\infty} O\!\left(d^{(r-n)\alpha'}\right),
\end{multline}
and the error term can be rewritten as $O\!\left(d^{-\alpha}\right)$.

Since $\kappa(z)^\frac{1}{2} = 1 + O\left(\Norm{z}^2\right)$ (cf.~\eqref{eq estimate kappa}), there exists $C>0$ independent of $x$ such that for all $z \in B_{T_xM}(0,R)$, $\norm{\kappa(z)^\frac{1}{2}-1} \leq C\Norm{z}^2$. Then, we get:
\begin{multline*}
\norm{\int_{z \in B_{T_xM}\left(0,b_n \ln d\right)} \phi_1(x)\phi_2\left(\exp_x\left(\frac{z}{\sqrt{d}}\right)\right) D_{n,r}(\Norm{z}^2) \left(\kappa\left(\frac{z}{\sqrt{d}}\right)^\frac{1}{2}-1\right)\dx z}\\
\begin{aligned}
&\leq \norm{\phi_1(x)}\Norm{\phi_2}_{\infty} C\frac{(b_n \ln d)^2}{d} \int_{z \in B\left(0,b_n \ln d\right)} \norm{D_{n,r}(\Norm{z}^2)} \dx z\\
&\leq \norm{\phi_1(x)}\Norm{\phi_2}_{\infty} \frac{C}{2}\frac{(b_n \ln d)^2}{d} \vol{\S^{n-1}} \int_{t=0}^{(b_n \ln d)^2} \norm{D_{n,r}(t)}t^\frac{n-2}{2} \dx t.
\end{aligned}
\end{multline*}
Since $\norm{D_{n,r}(t)}t^\frac{n-2}{2}$ is integrable on $(0,+\infty)$ (Lem.~\ref{lem integrability Dnr}) and $\alpha <1$, we have:
\begin{multline}
\label{mult2}
\int_{z \in B_{T_xM}\left(0,b_n \ln d\right)} \phi_1(x)\phi_2\left(\exp_x\left(\frac{z}{\sqrt{d}}\right)\right) D_{n,r}(\Norm{z}^2) \kappa\left(\frac{z}{\sqrt{d}}\right)^\frac{1}{2}\dx z=\\
\int_{z \in B_{T_xM}\left(0,b_n \ln d\right)} \phi_1(x)\phi_2\left(\exp_x\left(\frac{z}{\sqrt{d}}\right)\right) D_{n,r}(\Norm{z}^2) \dx z + \norm{\phi_1(x)}\Norm{\phi_2}_{\infty} O\!\left(d^{-\alpha}\right),
\end{multline}
where the error term in independent of $x$. By~\eqref{mult1} and~\eqref{mult2}, we have:
\begin{multline*}
\frac{1}{d^r}\int_{z \in B_{T_xM}\left(0,b_n \ln d\right)} \phi_1(x)\phi_2\left(\exp_x\left(\frac{z}{\sqrt{d}}\right)\right) D_d(x,z) \kappa\left(\frac{z}{\sqrt{d}}\right)^\frac{1}{2}\dx z=\\
\int_{z \in B_{T_xM}\left(0,b_n \ln d\right)} \phi_1(x)\phi_2\left(\exp_x\left(\frac{z}{\sqrt{d}}\right)\right) D_{n,r}(\Norm{z}^2) \dx z + \norm{\phi_1(x)}\Norm{\phi_2}_{\infty} O\!\left(d^{-\alpha}\right),
\end{multline*}
uniformly in $x \in M$. Integrating this relation over $M$ yields the result.
\end{proof}

Now, let $\alpha \in (0,\alpha_0)$, let $\phi_1$ and $\phi_2 \in \mathcal{C}^0(M)$, then by eq.~\eqref{eq variance2}, Prop.~\ref{prop off diagonal is small}, eq.~\eqref{eq double int rescaled} and Prop.~\ref{prop main estimate near diag} we have:
\begin{multline}
\label{variance 3}
\var{\rmes{d}}\left(\phi_1,\phi_2\right)=\\
\frac{d^{r-\frac{n}{2}}}{(2\pi)^r} \int_{x \in M} \left(\int_{z \in B_{T_xM}\left(0,b_n \ln d\right)} \phi_1(x)\phi_2\left(\exp_x\left(\frac{z}{\sqrt{d}}\right)\right) D_{n,r}(\Norm{z}^2) \dx z \right) \rmes{M}\\
+ \Norm{\phi_1}_{\infty}\Norm{\phi_2}_{\infty}  O\!\left(d^{r-\frac{n}{2}-\alpha}\right),
\end{multline}
where the error term is independent of $\left(\phi_1,\phi_2\right)$. Then, we have:
\begin{multline*}
\norm{\int_{z \in B_{T_xM}\left(0,b_n \ln d\right)} \left(\phi_1(x)\phi_2\left(\exp_x\left(\frac{z}{\sqrt{d}}\right)\right)-\phi_1(x)\phi_2(x)\right) D_{n,r}(\Norm{z}^2) \dx z} \\
\leq \Norm{\phi_1}_\infty \varpi_{\phi_2}\left(\frac{b_n \ln d}{\sqrt{d}}\right) \int_{z \in B\left(0,b_n \ln d\right)} \norm{D_{n,r}(\Norm{z}^2)} \dx z,
\end{multline*}
where $\varpi_{\phi_2}$ is the continuity modulus of $\phi_2$ (see Def.~\ref{def continuity modulus}). Besides, by a polar change of coordinates, we have:
\begin{equation}
\label{polar coordinates}
\int_{z \in B\left(0,b_n \ln d\right)} \norm{D_{n,r}(\Norm{z}^2)} \dx z = \frac{1}{2}\vol{\S^{n-1}} \int_{t=0}^{(b_n \ln d)^2} \norm{D_{n,r}(t)}t^\frac{n-2}{2} \dx t,
\end{equation}
and this quantity is bounded, by Lemma~\ref{lem integrability Dnr}. Then,
\begin{multline}
\label{variance 3bis}
\int_{z \in B_{T_xM}\left(0,b_n \ln d\right)} \phi_1(x)\phi_2\left(\exp_x\left(\frac{z}{\sqrt{d}}\right)\right) D_{n,r}(\Norm{z}^2) \dx z =\\
\phi_1(x)\phi_2(x) \int_{z \in B_{T_xM}\left(0,b_n \ln d\right)} D_{n,r}(\Norm{z}^2) \dx z + \Norm{\phi_1}_{\infty} \varpi_{\phi_2}\left(\frac{b_n \ln d}{\sqrt{d}}\right) O(1),
\end{multline}
where the error term is independent of $\left(\phi_1,\phi_2\right)$.

Let $\beta \in \left(0,\frac{1}{2}\right)$, then there exists $C_\beta >0$ such that for all $d \in \N^*$, $b_n\frac{\ln d}{\sqrt{d}} \leq C_\beta d^{-\beta}$. Since $\varpi_{\phi_2}$ is a non-decreasing function, we have $\varpi_{\phi_2}\left(b_n\frac{\ln d}{\sqrt{d}}\right) \leq \varpi_{\phi_2}\left(C_\beta d^{-\beta}\right)$. By~\eqref{variance 3}, \eqref{polar coordinates} and~\eqref{variance 3bis}, we obtain:
\begin{multline}
\label{variance 4}
\var{\rmes{d}}\left(\phi_1,\phi_2\right) =\\
d^{r-\frac{n}{2}} \frac{\vol{\S^{n-1}}}{(2\pi)^r}\left(\int_{M} \phi_1\phi_2\rmes{M}\right) \left(\frac{1}{2}\int_{t=0}^{(b_n \ln d)^2} D_{n,r}(t)t^\frac{n-2}{2} \dx t \right)\\
+ \Norm{\phi_1}_{\infty} \Norm{\phi_2}_{\infty} O\!\left(d^{r-\frac{n}{2}-\alpha}\right) + \Norm{\phi_1}_{\infty}\varpi_{\phi_2}\left(C_\beta d^{-\beta}\right) O\!\left(d^{r-\frac{n}{2}}\right).
\end{multline}

By Lemma~\ref{lem integrability D infinity}, we have: $\norm{D_{n,r}(t)} = O\!\left(te^{-\frac{t}{2}}\right)$. Then there exists some $C>0$ such that, for all $t$ large enough,
\begin{equation*}
\norm{D_{n,r}(t)} t^\frac{n-2}{2} \leq C e^{-\frac{t}{4}}.
\end{equation*}
Then, for $d$ large enough we have:
\begin{equation}
\label{variance 6}
\norm{\int_{t=(b_n \ln d)^2}^{+\infty} D_{n,r}(t)t^\frac{n-2}{2} \dx t} \leq C \int_{t=(b_n \ln d)^2}^{+\infty} e^{-\frac{t}{4}} \dx t \leq 4C \exp \left(-\frac{1}{4}b_n^2 (\ln d)^2\right) = O\!\left(d^{-1}\right).
\end{equation}
By equations~\eqref{variance 4} and~\eqref{variance 6}, we get:
\begin{multline}
\var{\rmes{d}}\left(\phi_1,\phi_2\right) = d^{r-\frac{n}{2}} \left(\int_M \phi_1\phi_2\rmes{M}\right) \frac{\vol{\S^{n-1}}}{(2\pi)^r}\left(\frac{1}{2}\int_{0}^{+\infty} \norm{D_{n,r}(t)}t^\frac{n-2}{2} \dx t \right)\\
+ \Norm{\phi_1}_{\infty}\Norm{\phi_2}_{\infty} O\!\left(d^{r-\frac{n}{2}-\alpha}\right) + \Norm{\phi_1}_{\infty}\varpi_{\phi_2}\left(C_\beta d^{-\beta}\right) O\!\left(d^{r-\frac{n}{2}}\right).
\end{multline}
Finally, recall that we defined ${I}_{n,r}$ by eq.~\eqref{eq def Inr} and $D_{n,r}$ by Def.~\ref{def D}. Hence, we have:
\begin{align*}
\mathcal{I}_{n,r} = \frac{1}{2}\int_{0}^{+\infty} \norm{D_{n,r}(t)}t^\frac{n-2}{2} \dx t,
\end{align*}
and this quantity is finite by Lemma~\ref{lem integrability Dnr}. This concludes the proof of Theorem~\ref{thm asymptotics variance}.

\section{Proofs of the corollaries}
\label{sec proofs of the corollaries}

\subsection{Proof of Corollary~\ref{cor concentration}}
\label{subsec proof of cor concentration}

Corollary~\ref{cor concentration} is a direct consequence of Thm.~\ref{thm asymptotics variance} and the Markov inequality. Let $\phi \in \mathcal{C}^0(M)$, then, by~\eqref{eq variance of linear statistics} we have:
\begin{equation*}
\var{\prsc{\rmes{d}}{\phi}} = O\!\left(d^{r-\frac{n}{2}}\right),
\end{equation*}
where the error term depends on $\phi$. Now, let $\alpha > \frac{r}{2}-\frac{n}{4}$ and $\epsilon >0$. We have:
\begin{align*}
\P\left( \norm{\prsc{\rmes{d}}{\phi}-\rule{0pt}{4mm}\esp{\prsc{\rmes{d}}{\phi}}} > d^{\alpha} \epsilon \right)  &= \P\left( d^{-\alpha}\norm{\prsc{\rmes{d}}{\phi}-\esp{\prsc{\rmes{d}}{\phi}}} > \epsilon \right)\\
&\leq \frac{1}{\epsilon^2} \var{d^{-\alpha}\prsc{\rmes{d}}{\phi}}\\
&\leq \frac{1}{\epsilon^2} d^{-2\alpha}\var{\prsc{\rmes{d}}{\phi}}.
\end{align*}

\subsection{Proof of Corollary~\ref{cor connected components}}
\label{subsec proof of cor connected components}

We obtain Cor.~\ref{cor connected components} as a consequence of Cor.~\ref{cor concentration}. Let $U\subset M$ be an open subset. We denote by $\phi_U \in \mathcal{C}^0(M)$ the function such that $\phi_U(x)$ is the geodesic distance from $x$ to the complement of $U$ in $(M,g)$. Then we have:
\begin{equation*}
U = \left\{ x \in M \mvert \phi_U(x) >0 \right\},
\end{equation*}
and $\phi_U$ is non-negative. Hence, $Z_d \cap U = \emptyset$ if and only if $\prsc{\rmes{d}}{\phi_U}=0$. Let $\epsilon >0$ such that:
\begin{equation*}
\epsilon < \frac{1}{2}\left(\int_M \phi_U \rmes{M}\right) \frac{\vol{\S^{n-r}}}{\vol{\S^n}}.
\end{equation*}
Then, by Thm.~\ref{thm reminder expectation test function}, for $d$ large enough we have:
\begin{equation*}
d^{-\frac{r}{2}}\esp{\prsc{\rmes{d}}{\phi_U}} - \epsilon \geq \frac{1}{2}\left(\int_M \phi_U \rmes{M}\right) \frac{\vol{\S^{n-r}}}{\vol{\S^n}} >0.
\end{equation*}
Thus, for $d$ large enough, we have:
\begin{align*}
\P\left( Z_d \cap U = \emptyset \right) &= \P\left( \prsc{\rmes{d}}{\phi_U} = 0 \right)\\
&\leq \P\left( \prsc{\rmes{d}}{\phi_U} < \esp{\prsc{\rmes{d}}{\phi_U}} - d^{\frac{r}{2}}\epsilon\right)\\
&\leq \P\left( \norm{\prsc{\rmes{d}}{\phi_U}-\rule{0pt}{4mm}\esp{\prsc{\rmes{d}}{\phi_U}}} > d^{\frac{r}{2}} \epsilon \right).
\end{align*}
And by Cor.~\ref{cor concentration}, this is a $O\!\left(d^{-\frac{n}{2}}\right)$.

\subsection{Proof of Corollary~\ref{cor as convergence}}
\label{subsec proof of cor as convergence}

In this section we assume that $n \geq 3$. We consider a random sequence $(s_d)_{d \in \N}$ of sections of increasing degree, distributed according to the probability measure $\dx \nu = \bigotimes_{d \in \N} \dx \nu_d$ on $\prod_{d \in \N} \R \H$. Strictly speaking, $\rmes{s_d}$ is not defined for small $d$. However, $\dx \nu$-almost surely, $\rmes{s_d}$ is well-defined for all $d \geq d_1$, so the statement of Cor.~\ref{cor as convergence} makes sense.

Our proof follows the lines of the proof of Shiffman and Zelditch \cite[sect.~3.3]{SZ1999} in the complex case. First, we prove that for every fixed $\phi \in \mathcal{C}^0(M)$ we have:
\begin{equation}
\label{eq as convergence single function}
d^{-\frac{r}{2}}\prsc{\rmes{s_d}}{\phi} \xrightarrow[d \to +\infty]{} \frac{\vol{\S^{n-r}}}{\vol{\S^n}} \left(\int_M \phi \rmes{M}\right).
\end{equation}
Then we use a separability argument to get the result. In the complex algebraic setting of~\cite{SZ1999}, the scaled volume of $s_d^{-1}(0) \subset \X$ is a deterministic constant, independent of $d$. In our real algebraic setting this is not the case.

Let $\phi \in \mathcal{C}^0(M)$, then we have:
\begin{equation*}
\esp{\sum_{d \in \N} \left(d^{-\frac{r}{2}}\left(\prsc{\rmes{s_d}}{\phi} - \rule{0pt}{4mm}\esp{\prsc{\rmes{d}}{\phi}}\right)\right)^2} = \sum_{d \in \N} d^{-r} \var{\prsc{\rmes{d}}{\phi}} < +\infty,
\end{equation*}
since $d^{-r} \var{\prsc{\rmes{d}}{\phi}} = O\!\left(d^{- \frac{n}{2}}\right)$ by Cor.~\ref{cor variance of linear statistics}. Hence, $\dx \nu$-almost surely, we have:
\begin{equation*}
\sum_{d \in \N} \left(d^{-\frac{r}{2}}\left(\prsc{\rmes{s_d}}{\phi} - \rule{0pt}{4mm}\esp{\prsc{\rmes{d}}{\phi}}\right)\right)^2 < +\infty,
\end{equation*}
and
\begin{equation*}
\left(d^{-\frac{r}{2}}\prsc{\rmes{s_d}}{\phi} - \rule{0pt}{4mm}d^{-\frac{r}{2}}\esp{\prsc{\rmes{d}}{\phi}}\right) \xrightarrow[d \to +\infty]{} 0.
\end{equation*}
Then, by Thm.~\ref{thm reminder expectation test function}, $\prsc{\rmes{s_d}}{\phi}$ satisfies~\eqref{eq as convergence single function} $\dx \nu$-almost surely.

Let $\left(\phi_k\right)_{k \in \N}$ be a dense sequence in the separable space $\left(\mathcal{C}^0(M),\Norm{\cdot}_\infty\right)$. Without loss of generality, we can assume that $\phi_0 = \mathbf{1}$, the unit constant function on $M$. Then, $\dx \nu$-almost surely, we have:
\begin{equation}
\label{eq as convergence countable}
\forall k \in \N, \qquad d^{-\frac{r}{2}}\prsc{\rmes{s_d}}{\phi_k} \xrightarrow[d \to +\infty]{} \frac{\vol{\S^{n-r}}}{\vol{\S^n}} \left(\int_M \phi_k \rmes{M}\right).
\end{equation}

Let $\underline{s}=\left(s_d\right)_{d \in \N} \in \prod_{d \in \N} \R \H$ be a fixed sequence such that \eqref{eq as convergence countable} holds. For every $\phi \in \mathcal{C}^0(M)$ and $k \in \N$ we have:
\begin{multline*}
\norm{d^{-\frac{r}{2}}\prsc{\rmes{s_d}}{\phi} - \frac{\vol{\S^{n-r}}}{\vol{\S^n}} \left(\int_M \phi \rmes{M}\right)}\\
\begin{aligned}
\leq& \norm{d^{-\frac{r}{2}}\prsc{\rmes{s_d}}{\phi}-d^{-\frac{r}{2}}\prsc{\rmes{s_d}}{\phi_k}}\\
&+ \frac{\vol{\S^{n-r}}}{\vol{\S^n}}\norm{\int_M \phi_k \rmes{M}-\int_M \phi \rmes{M}}\\
&+ \norm{d^{-\frac{r}{2}}\prsc{\rmes{s_d}}{\phi_k} - \frac{\vol{\S^{n-r}}}{\vol{\S^n}} \left(\int_M \phi_k \rmes{M}\right)}\\
\leq& \Norm{\phi-\phi_k}_\infty \left(d^{-\frac{r}{2}}\prsc{\rmes{s_d}}{\mathbf{1}}+ \frac{\vol{\S^{n-r}}}{\vol{\S^n}}\vol{M}\right)\\
&+\norm{d^{-\frac{r}{2}}\prsc{\rmes{s_d}}{\phi_k} - \frac{\vol{\S^{n-r}}}{\vol{\S^n}} \left(\int_M \phi_k \rmes{M}\right)}.
\end{aligned}
\end{multline*}

Recall that $\phi_0 = \mathbf{1}$. Then, by~\eqref{eq as convergence countable}, the sequence $(d^{-\frac{r}{2}}\prsc{\rmes{s_d}}{\mathbf{1}})_{d \in \N}$ converges. Hence it is bounded by some positive constant $K_{\underline{s}}$. Let $\phi \in \mathcal{C}^0(M)$ and let $\epsilon >0$. Let $k \in \N$ be such that:
\begin{equation*}
\Norm{\phi - \phi_k}_\infty \leq \epsilon \left(K_{\underline{s}} +\frac{\vol{\S^{n-r}}}{\vol{\S^n}}\vol{M}\right)^{-1}.
\end{equation*}
Then, for every $d$ large enough we have:
\begin{equation*}
\norm{d^{-\frac{r}{2}}\prsc{\rmes{s_d}}{\phi_k} - \frac{\vol{\S^{n-r}}}{\vol{\S^n}} \left(\int_M \phi_k \rmes{M}\right)} \leq \epsilon,
\end{equation*}
and
\begin{equation*}
\norm{d^{-\frac{r}{2}}\prsc{\rmes{s_d}}{\phi} - \frac{\vol{\S^{n-r}}}{\vol{\S^n}} \left(\int_M \phi \rmes{M}\right)} \leq 2\epsilon.
\end{equation*}
Thus, $\phi$ satisfies~\eqref{eq as convergence single function}.

Finally, whenever \eqref{eq as convergence countable} is satisfied we have: for every $\phi \in \mathcal{C}^0(M)$, $\phi$ satisfies~\eqref{eq as convergence single function}. Since the condition~\eqref{eq as convergence countable} is satisfied $\dx \nu$-almost surely, this proves Cor.~\ref{cor as convergence}.

\bibliographystyle{amsplain}
\bibliography{Varianceofthevolumeofrandomrealalgebraicsubmanifolds}

\end{document}